\documentclass[12pt,reqno,twoside,a4paper]{amsart}%
\usepackage[mathcal]{euler}
\usepackage[full]{textcomp}
     \usepackage{newpxtext} 
     \usepackage{cabin} 
     \usepackage[varqu,varl]{inconsolata} 
\usepackage{amsmath,amssymb,amsthm,graphbox,color}%
\usepackage{url}
\usepackage{tikz-cd}
\usepackage[colorinlistoftodos,textsize=footnotesize]{todonotes}

\numberwithin{equation}{section}
\makeatletter
\renewcommand{\section}{\@startsection {section}{1}{\z@}%
                                   {-3.5ex \@plus -1ex \@minus -.2ex}%
                                   {.5\linespacing}%
                                   {\normalfont\scshape\centering}}
\makeatother

\newtheorem{theorem}{Theorem}[section]
\newtheorem{lemma}[theorem]{Lemma}
\newtheorem{corollary}[theorem]{Corollary}
\newtheorem{proposition}[theorem]{Proposition}
\newtheorem{conjecture}[theorem]{Conjecture}
  
\theoremstyle{definition}
\newtheorem{definition}[theorem]{Definition}

\theoremstyle{remark}
\newtheorem*{example}{Example}
\newtheorem*{remark}{Remark}

\def\beq#1\eeq{\begin{equation}#1\end{equation}}
\makeatletter
\@ifpackageloaded{euler}{
 \newcommand{\onto}{\to\mkern-14mu\to}
 \def\hbar{{\mathchar'26\mkern-6.5muh}}
 
}{
 \DeclareMathSymbol{\onto}{\mathrel}{AMSa}{"10}
 \renewcommand{\hbar}{{\mathchar'26\mkern-9muh}}
 
}
\makeatother
 
\newcommand{\into}{\hookrightarrow}
\DeclareMathOperator{\id}{id}

\newcommand{\Ghom}{\mathcal{G}}
\newcommand{\E}{\mathcal{E}}
\newcommand{\V}{\mathcal{V}}
\newcommand{\W}{\mathcal{W}}
\newcommand{\Tree}{\mathcal{T}}

\renewcommand{\Im}{\operatorname{Im}}
\newcommand{\s}{\mathsf{s}}
\newcommand{\m}{\mathsf{m}}
\newcommand{\abs}[1]{\lvert#1\rvert}
\newcommand{\Abs}[1]{\left|#1\right|}
\newcommand{\A}{\mathcal{A}}
\newcommand{\N}{\mathbb N}
\newcommand{\norm}[1]{\lVert#1\rVert}

\newcommand{\Vect}{\mathsf{Vec}}
\newcommand{\Alg}{\mathsf{Alg}}
\newcommand{\Xc}{\mathsf C}
\DeclareMathOperator{\Hom}{Hom}

\DeclareMathOperator{\Qcat}{\mathsf{Quilt}}
\DeclareMathOperator{\Qop}{\textit{Quilt}}
\DeclareMathOperator{\Brace}{\textit{Brace}}
\newcommand{\bphi}{\boldsymbol{\phi}}
\DeclareMathOperator{\sgn}{sgn}
\newcommand{\Rep}{\mathcal{R}}
\newcommand{\mQuilt}{\textit{mQuilt}}
\DeclareMathOperator{\ad}{ad}
\newcommand{\dH}{\delta^{\mathrm H}}
\newcommand{\dS}{\delta^{\mathrm S}}
\newcommand{\comp}{\mathbin{\bar\circ}}
\newcommand{\Gerst}{\textit{Gerst}}
\newcommand{\Susp}{\mathcal{S}}
\newcommand{\Linfty}{\textit{L}_\infty}
\newcommand{\Place}{\mathcal{C}}
\DeclareMathOperator{\Wordcat}{\mathsf{F_2S}}
\DeclareMathOperator{\Wordop}{\textit{F}_2\textit{S}}
\newcommand{\Comcat}{\Simplex\!'}
\DeclareMathOperator{\Comop}{\textit{Com}}
\DeclareMathOperator{\Treecat}{\mathsf{Tree}}
\DeclareMathOperator{\Lieop}{\textit{Lie}}
\newcommand{\com}[1]{\langle#1\rangle}
\newcommand{\perm}{\mathbb{S}}
\DeclareMathOperator{\Ext}{Ext}
\newcommand{\Simplex}{\boldsymbol{\Delta}}
\DeclareMathOperator{\Endop}{\textit{End}}
\DeclareMathOperator{\tot}{tot}
\newcommand{\rep}{\mathcal{L}}
\newcommand{\Had}{\mathbin{\underset{\raisebox{.25ex}{$\scriptscriptstyle\mathrm H$}}{\otimes}}}
\newcommand{\Hadtimes}{\mathbin{\underset{\raisebox{.25ex}{$\scriptscriptstyle\mathrm H$}}{\times}}}
\newcommand{\homone}{\mathcal{K}}
\newcommand{\homtwo}{\mathcal{J}}
\DeclareMathOperator{\Sh}{Sh}
\DeclareMathOperator{\Char}{Char}
\DeclareMathOperator{\Obj}{Obj}

\DeclareMathOperator{\CQ}{\textit{ColorQuilt}}
\DeclareMathOperator{\SemiD}{\Simplex_+}
\DeclareMathOperator{\MultiSemiD}{\textit{Multi}\!\SemiD}
\DeclareMathOperator{\MultiD}{\textit{Multi}\Simplex}\DeclareMathOperator{\NSOp}{\textit{NS\hspace{.05em}Op}}
\DeclareMathOperator{\MC}{MC}
\newcommand{\Qquotient}{\mathcal H}
\newcommand{\bp}{\mathbf p}
\newcommand{\bq}{\mathbf q}
\newcommand{\bzeta}{\boldsymbol\zeta}
\newcommand{\bxi}{\boldsymbol\xi}
\DeclareMathOperator{\Ind}{Clr}
\newcommand{\etab}{\boldsymbol\eta}
\DeclareMathOperator{\Set}{\mathsf{Set}}
\newcommand{\op}{\mathrm{op}}
\DeclareMathOperator{\Sq}{Sq}
\newcommand{\Wone}{\mathfrak{I}}
\newcommand{\Wtwo}{\mathfrak{S}}
\newcommand{\onecat}{\mbox{\boldmath{$1$}}}
\newcommand{\twocat}{\mbox{\boldmath{$2$}}}

\newcommand{\crush}[1]{\makebox[.5em]{$#1$}}
\newcommand{\Qtwo}[2]{\Qwrap{\begin{array}{|c|}\hline\crush#1\\ \hline \crush#2 \\ \hline \end{array}}}
\newcommand{\Qthreeb}[3]{\Qwrap{\begin{array}{|c|}\hline\crush#1\\ \hline \crush#2 \\ \hline \crush#3 \\ \hline\end{array}}}
\newcommand{\Qfive}[5]{\Qwrap{\begin{array}{|c|c|}\hline\multicolumn{2}{|c|}{#1}\\ \hline \crush#2&\crush#3 \\ \hline \crush#4&\crush#5 \\ \hline\end{array}}}
\newcommand{\blank}{{\color[gray]{.8}\makebox[.5em]{\rule[-.4em]{1.125em}{1.25em}}}}
\newcommand{\Qthreea}[3]{\Qwrap{\begin{array}{|c|c|}\hline\multicolumn{2}{|c|}{#1}\\\hline#2 & #3\\\hline\end{array}}}
\newcommand{\Qwrap}[1]{\mathord{\raisebox{0pt}[\height+.5ex][\depth+.5ex]{$\,\setlength{\arraycolsep}{.33em}#1\,$}}}
\newcommand{\Qfoura}[4]{\Qwrap{\begin{array}{|c|c|}%
\hline%
\multicolumn{2}{|c|}{#1}\\%
\hline%
#2 & \\%
\cline{1-1}%
#4 & \raisebox{1ex}[0pt]{#3}\\%
\hline%
\end{array}}}
\newcommand{\Qfourb}[4]{\Qwrap{\begin{array}{|c|c|}%
\hline%
\multicolumn{2}{|c|}{#1}\\%
\hline%
#2 & \multicolumn{1}{||c|}{}\\%
\cline{1-1} %
#4 & \multicolumn{1}{||c|}{\raisebox{1ex}[0pt]{#3}}\\%
\hline%
\end{array}}}
\newcommand{\Qfourc}[4]{\Qwrap{\begin{array}{|c|c|}\hline\multicolumn{2}{|c|}{#1}\\\hline\multicolumn{2}{|c|}{#2}\\\hline#3 & #4\\\hline\end{array}}}
\newcommand{\Qsevena}[7]{\Qwrap{\begin{array}{|c|c|}\hline\multicolumn{2}{|c|}{#1}\\ \hline#2&#3\\ \hline#4&#5\\ \hline#6&#7\\ \hline\end{array}}}
\newcommand{\Qnineb}[9]{\Qwrap{\begin{array}{|c|c|c|}
\hline
\multicolumn{3}{|c|}{#1}\\
\hline
#2 & \multicolumn{2}{|c|}{#3}\\
\hline
#4 & #5 & #6 \\
\hline
#7 & #8 & #9 \\
\hline
\end{array}}}
\newcommand{\Qninea}[9]{\Qwrap{\begin{array}{|c|c|c|}
\hline
\multicolumn{3}{|c|}{#1}\\
\hline
\multicolumn{2}{|c|}{#2} & #3\\
\hline
#4 & #5 & #6 \\
\hline
#7 & #8 & #9 \\
\hline
\end{array}}}
\newcommand{\Qsixa}[6]{\Qwrap{\begin{array}{|c|c|}\hline\multicolumn{2}{|c|}{#1}\\\hline\multicolumn{2}{|c|}{#2}\\\hline#3 & #4\\\hline#5 & #6\\\hline\end{array}}}
\newcommand{\Qsixc}[6]{\Qwrap{\begin{array}{|c|c|c|}\hline\multicolumn{3}{|c|}{#1}\\\hline#2&\multicolumn{2}{|c|}{#3}\\\hline#4&#5&#6\\\hline\end{array}}}
\newcommand{\Qeightb}[8]{\Qwrap{\begin{array}{|c|c|c|}
\hline
\multicolumn{3}{|c|}{#1} \\
\hline
#2 & \multicolumn{2}{|c|}{#3}\\
\hline
 #4 &#5& \\%
\cline{1-2} 
\blankstrut#7&#8  & \raisebox{1ex}[0pt]{$#6$} \\
\hline
\end{array}}}
\newcommand{\Qsixe}[6]{\Qwrap{\begin{array}{|c|c|c|}\hline\multicolumn{3}{|c|}{#1}\\\hline\multicolumn{2}{|c|}{#2}&#3\\\hline#4&#5&#6\\\hline\end{array}}}
\newcommand{\Qeighta}[8]{\Qwrap{\begin{array}{|c|c|c|}
\hline
\multicolumn{3}{|c|}{#1} \\
\hline
\multicolumn{2}{|c|}{#2} & #3\\
\hline
 #4 &#5& \\%
\cline{1-2} 
\blankstrut#7&#8  & \raisebox{1ex}[0pt]{$#6$} \\
\hline
\end{array}}}
\newcommand{\Qfourd}[4]{\Qwrap{\begin{array}{|c|c|c|}\hline\multicolumn{3}{|c|}{#1}\\\hline#2&#3&#4\\\hline\end{array}}}
\newcommand{\Qsevenb}[7]{\Qwrap{\begin{array}{|c|c|c|}\hline\multicolumn{3}{|c|}{#1}\\\hline#2&#3&#4\\\hline#5&#6&#7\\\hline\end{array}}}
\newcommand{\Qsixd}[6]{\Qwrap{\begin{array}{|c|c|c|}\hline\multicolumn{3}{|c|}{#1}\\\hline#2&#3&\multicolumn{1}{||c|}{}\\\cline{1-2}\blankstrut#5&#6&\multicolumn{1}{||c|}{\raisebox{1ex}[0pt]{$#4$}}\\\hline\end{array}}}
\newcommand{\Qsixb}[6]{\Qwrap{\begin{array}{|c|c|c|}\hline\multicolumn{3}{|c|}{#1}\\\hline#2&#3&\multicolumn{1}{|c|}{}\\\cline{1-2}\blankstrut#5&#6&\multicolumn{1}{|c|}{\raisebox{1ex}[0pt]{$#4$}}\\\hline\end{array}}}
\newcommand{\Qfivea}[5]{\Qwrap{\begin{array}{|c|c|}\hline\multicolumn{2}{|c|}{#1}\\\hline#2&\\ \cline{1-1}#4&#3\\ \cline{1-1} #5&\\ \hline\end{array}}}
\newcommand{\Qfiveb}[5]{\Qwrap{\begin{array}{|c|c|c|}\hline\multicolumn{3}{|c|}{#1}\\ \hline \multicolumn{2}{|c|}{#2}&\\ \cline{1-2}  #4&#5&\raisebox{1ex}[0pt]{$#3$}\\ \hline
\end{array}}}
\newcommand{\Qseven}[7]{\Qwrap{\begin{array}{|c|c|c|}\hline\multicolumn{3}{|c|}{#1}\\ \hline\multicolumn{2}{|c|}{#2}&\\ \cline{1-2} \blankstrut#4&#5&#3\\ \cline{1-2} \blankstrut#6&#7& \\ \hline
\end{array}}}
\newcommand{\Qfivec}[5]{\Qwrap{\begin{array}{|c|c|c|}\hline\multicolumn{3}{|c|}{#1}\\\hline\multicolumn{2}{|c|}{#2}&\multicolumn{1}{||c|}{}\\\cline{1-2}\blankstrut#4&#5&\multicolumn{1}{||c|}{\raisebox{1ex}[0pt]{$#3$}}\\\hline\end{array}}}
\newcommand{\blankstrut}{\rule{0pt}{.9em}}
\newcommand{\Qeightc}[8]{\Qwrap{\begin{array}{|c|c|c|}\hline\multicolumn{3}{|c|}{#1}\\ \hline\multicolumn{2}{|c}{#2} & \multicolumn{1}{@{\,}|c|}{#3} \\ \hline #4 & #5 & \multicolumn{1}{||c|}{}\\ \cline{1-2} #6 & \blankstrut #7 & \multicolumn{1}{||c|}{\raisebox{1ex}[0pt]{#8}}\\ \hline \end{array}}}

\begin{document}
\title[Operations on the Hochschild Bicomplex]{Operations on the Hochschild Bicomplex\\ of a Diagram of Algebras}
\author{Eli Hawkins}
\address{Department of Mathematics\\
The University of York\\
United Kingdom}
\email{eli.hawkins@york.ac.uk}
\begin{abstract}
A diagram of algebras is a functor valued in a category of associative algebras.
I construct an operad acting on the Hochschild bicomplex of a diagram of algebras. Using this operad, I give a direct proof that the Hochschild cohomology of a diagram of algebras is a Gerstenhaber algebra. I also show that the total complex is an $\Linfty$-algebra. The same results are true for the reduced  and asimplicial subcomplexes and asimplicial cohomology. This structure governs deformations of diagrams of algebras through the Maurer-Cartan equation.
\end{abstract}
\maketitle
\tableofcontents

\section{Introduction}
Let $\Bbbk$ be a field and $A$ an associative algebra over $\Bbbk$. In \cite{ger1963,ger1964}, Gerstenhaber investigated the Hochschild cohomology $H^\bullet(A,A)$ and showed that this is naturally equipped with a  commutative-associative product of degree $0$ and a compatible Lie bracket of degree $-1$. This structure is now known as a \emph{Gerstenhaber algebra}. In particular, the graded Lie algebra structure comes from a differential graded (dg) Lie algebra structure on the Hochschild  complex $C^\bullet(A,A)$. Deformations of $A$ are described by the Maurer-Cartan equation for this dg-Lie algebra. The algebraic structure of $C^\bullet(A,A)$ is the subject of the \emph{Deligne conjecture} \cite{g-l-t2015,g-v1995,kau2007,m-s2002,m-s2003,vor2000}.

Subsequently, Gerstenhaber and Voronov \cite{g-v1995} showed that $C^\bullet(A,A)$ has the further structure of a \emph{brace algebra}. Indeed, this only depends upon $A$ as a vector space, not an algebra. The brace algebra elucidates the structure of  Gerstenhaber's original proof that $H^\bullet(A,A)$ is a Gerstenhaber algebra.

Let $\Alg_\Bbbk$ be the category of associative algebras over $\Bbbk$ and $\Xc$ some small category. A \emph{diagram of algebras} is a covariant functor $\A:\Xc\to\Alg_\Bbbk$. 

Gerstenhaber and Schack \cite{g-s1988a,g-s1988b} generalized Hochschild cohomology  to diagrams of algebras (although $\A$ is contravariant there). The complex for computing $H^\bullet(\A,\A)$ is naturally expressed as the total complex of a double complex $C^{\bullet,\bullet}(\A,\A)$. They showed that this cohomology is a Gerstenhaber algebra indirectly by a ``cohomology comparison theorem''  that $H^\bullet(\A,\A)$ is naturally identified with the Hochschild cohomology of a certain single algebra. They also defined \emph{asimplicial} cohomology $H_{\mathrm a}^\bullet(\A,\A)$ from the subcomplex  $C_{\mathrm a}^{\bullet,\bullet}(\A,\A)\subset C^{\bullet,\bullet}(\A,\A)$ where the second degree is strictly positive. They showed that deformations of $\A$ are described to first order in $H_{\mathrm a}^\bullet(\A,\A)$.

In \cite{haw2018}, I pointed out that an algebraic quantum field theory is a diagram of algebras, and thus asimplicial cohomology can be used to study infinitesimal deformations and symmetries of algebraic quantum field theories. Moreover, the full Hochschild cohomology suggests a generalization of algebraic quantum field theory and a definition of generalized symmetries.

A diagram of algebras $\A:\Xc\to\Alg_\Bbbk$ consists of 3 pieces of information. For each object  $x\in\Obj\Xc$, there is a vector space $\A(x)$. For each object $x\in\Obj\Xc$ there is an associative multiplication, $\hat\m[x]:\A(x)\otimes\A(x)\to\A(x)$. For each $\Xc$-morphism, $\phi:x\to y$, there is a linear map $\A[\phi]:\A(x)\to\A(y)$. 

These structures satisfy some conditions: Multiplication is associative, $\A$ is functorial, and $\A[\phi]$ is an algebra homomorphism. This last condition means that,
\[
\A[\phi]\circ\hat\m[x] = \hat\m[y]\circ\left(\A[\phi]\otimes\A[\phi]\right) .
\]
The left hand side is quadratic, but the right hand side is \emph{cubic} (linear in $\hat\m[y]$ and quadratic in $\A[\phi]$). 

A \emph{deformation} of $\A$ is a cochain that when added to $\hat\m$ and $\A$ gives another diagram of algebras.
This shows that a deformation of $\A$ must satisfy a cubic equation.

Deformations of a single algebra satisfy the Maurer-Cartan equation in a dg-Lie algebra. This is always a quadratic equation, therefore deformations of a diagram of algebras \emph{cannot} be described by a dg-Lie algebra. The natural generalization of a dg-Lie algebra is an $\Linfty$-algebra; the Maurer-Cartan equation generalizes to that setting, but may have arbitrarily high degree.
This suggests that $C_{\mathrm a}^{\bullet,\bullet}(\A,\A)$ may have the structure of an $\Linfty$-algebra such that deformations of $\A$ are characterized by the Maurer-Cartan equation.

A single morphism between two algebras is equivalant to a diagram of algebras $\A:\Xc\to\Alg_\Bbbk$, where $\Xc$ has two objects and only one nonidentity morphism; Borisov \cite{bor2007} has constructed the $\Linfty$-algebra in this case. Fr\'egier, Markl, and Yau \cite{f-m-y2009} gave a construction for the $\Linfty$-algebra from a resolution of the colored operad governing diagrams of algebras over $\Xc$, but they did not present such a resolution, except in the aforementioned case of a single morphism. 

The main purposes of this paper are to provide a direct proof that $H^\bullet(\A,\A)$ and $H_{\mathrm a}^\bullet(\A,\A)$ are Gerstenhaber algebras, to explicitly construct for the first time an $\Linfty$-algebra structure on $C^{\bullet,\bullet}(\A,\A)$  (with the asimplicial and reduced subcomplexes being $\Linfty$-subalgebras), and to show that the deformations of a diagram of algebras $\A$ are the solutions of the Maurer-Cartan equation in the reduced asimplicial subcomplex $\bar C_{\mathrm a}^{\bullet,\bullet}(\A,\A)$.

To do this, I will first construct an operad $\Qop$ that acts on the Hochschild bicomplex of a diagram of vector spaces. This should be seen as a generalization of the brace algebra structure on $C^\bullet(A,A)$. I then extend this to another operad $\mQuilt$ by adjoining an additional generator $\m$ that represents multiplication. Using $\mQuilt$, I construct representatives of the Gerstenhaber algebra operations and explicit homotopies for each of the identities that the operations need to satisfy on cohomology. I also construct a homomorphism $\homtwo:\Linfty\to\mQuilt$ that makes $C^{\bullet,\bullet}(\A,\A)$ an $\Linfty$-algebra. All of these structures restrict to the asimplicial subcomplex $C_{\mathrm a}^{\bullet,\bullet}(\A,\A)$ and reduced subcomplex $\bar C^{\bullet,\bullet}(\A,\A)$.

Kontsevich's formality theorem  and classification of formal deformation quantizations \cite{kon2003} is based upon the dg-Lie algebra structure of the Hochschild complex of an algebra, but that is really treated as an $\Linfty$-algebra whose higher operations happen to vanish. The formality map is an $\Linfty$-quasi-isomorphism. The $\Linfty$ structure that I present here opens up the possibility of similar results for diagrams of algebras, such as a classification of formal deformation quantizations of field theories.

\begin{remark}
My results are stated in terms of diagrams of algebras over a field, but they are true more generally. In fact, $\A$ could be a functor from $\Xc$ to a nonsymmetric colored operad (multicategory) of abelian groups, equipped with a ``multiplication'' in the sense of \cite{g-v1995}. In particular, the results are true for diagrams of algebras over a commutative ring. I have focussed on diagrams of algebras over a field because that is my primary interest and I want to avoid overly esoteric notation and terminology.
\end{remark}
\begin{remark}
A Gerstenhaber algebra contains 2 structures on the underlying graded vector space: a commutative algebra on the vector space and a Lie algebra on the suspended (degree shifted) vector space. This creates a small dilemma when describing this with an operad. There are 2 ways to proceed:
\begin{enumerate}
\item
The operad $\Gerst$ in \cite{l-v2012} acts directly on the  vector space. Consequently, $\Comop$ is a suboperad, but $\Lieop$ appears as the desuspended Lie operad, $\Susp^{-1}\Lieop$, with a \emph{symmetric} bracket. 
\item
Alternately, one can treat the bracket with greater respect and use the suspended Gerstenhaber operad $\Susp\Gerst$ acting on the suspended vector space, so that $\Susp\Comop$ and $\Lieop$ are suboperads. This makes the commutative product look rather odd.
\end{enumerate}
I will be taking the second approach here, because this is more convenient for $\Qop$ and is more natural for $\Linfty$.
\end{remark}

\subsection{Synopsis}
I begin in Section \ref{Old operads} by reviewing two known operads. $\Brace$ governs brace algebras. $\Wordop$ governs Gerstenhaber-Voronov ``homotopy G-algebras''.

In Section~\ref{Quilt section}, I define a new operad, $\Qop$, as a suboperad of the Hadamard product $\Wordop\Had\Brace$. 

In Section~\ref{Representation}, I define a representation of $\Qop$ on the Hochschild bicomplex of a diagram of vector spaces. To this end, I first describe colored operads $\NSOp$, $\MultiSemiD$, and $\CQ\subset \MultiSemiD\Hadtimes\NSOp$. $\CQ$ acts on the Hochschild bicomplex (as a collection of vector spaces) and the action of an element of $\Qop$ on a cochain is defined by a sum over elements of $\CQ$. I define the reduced subcomplex and show that this and the asimplicial subcomplex are $\Qop$-subalgebras.

In Section~\ref{mQuilts}, I construct another operad, $\mQuilt$, from $\Qop$ and an additional generator. I define a representation of $\mQuilt$ on the Hochschild bicomplex of a diagram of algebras. 

In Section~\ref{Gerstenhaber}, I use $\mQuilt$ to express an associative product and a bracket on the Hochschild bicomplex. Using $\mQuilt$, I write explicit homotopies for each of the relations of a Gerstenhaber algebra, thus showing that the Hochschild cohomology of a diagram of algebras is a Gerstenhaber algebra. I also show that over a field of characteristic $2$, the squaring operation descends to cohomology.

In Section~\ref{L-infinity}, I first construct a homomorphism $\homone:\Linfty\to\Qop$. I then modify this to get another homomorphism $\homtwo:\Linfty\to\mQuilt$. In this way, the Hochschild complex of a diagram of algebras is an $\Linfty$-algebra. In Section~\ref{Maurer-Cartan}, I discuss the Maurer-Cartan equation that governs deformations of diagrams of algebras. Finally, in Section~\ref{One or two}, I briefly discuss the $\Linfty$ algebra for the cases of a single algebra or a single morphism between algebras.

Appendix~\ref{Notation} is a summary of notation.

\section{Some Combinatorial Operads}
\label{Old operads}
In this section, I describe two known operads and  my notations for trees and words.
\subsection{Extensions}
\label{Com}
Recall the standard notation that for any $n\in\N=\{0,1,2,\dots\}$, $[n]=\{0,\dots,n\}\subset\N$. The \emph{simplex} category $\Simplex$ is the category whose set of objects is $\{[n]\mid n\in\N\}$ and whose  morphisms are weakly increasing functions.

The following slight variant of this notation will be useful here.
\begin{definition}
\label{Comcat def}
For any $n\in\N$, let $\com{n} := \{1,2,\dots,n\}$. 
Let $\Comcat$ be the category whose set of objects is $\{\com{n} \mid n\geq1\}$ and whose morphisms are weakly increasing maps. 
\end{definition}
There is an obvious isomorphism from $\Simplex$ to $\Comcat$ such that $[n]\mapsto \com{n+1}$.

Let's define something like a short exact sequence in $\Comcat$:
\[
\com{n} \stackrel\alpha\into \com{l} \stackrel\beta\onto \com{m} ,
\]
with $\alpha$  injective and $\beta$  surjective. We can't require $\beta\circ\alpha$ to be $0$, because there is no $0$; the best that we can do instead is to require $\beta\circ\alpha$ to be a constant function with image (say) $a\in\com{m}$. To make this exact, require the image of $\alpha$ to be $\beta^{-1}(a)$. This is enough information to uniquely determine $l$, $\alpha$, and $\beta$. 
\begin{definition}
\label{Ext def}
\emph{The extension of $\com{m}$ by $\com{n}$ at $a\in\com{m}$} in $\Comcat$ is
\[
\com{n} \stackrel\alpha\into \com{n+m-1} \stackrel\beta\onto \com{m} ,
\]
where $\alpha(j) := j+a-1$, and 
\[
\beta(k) = 
\begin{cases}
k & k<a \\
a & a\leq k < a+n\\
k+1-n & k\geq a+n .
\end{cases}
\]
\end{definition}

Rather trivially,  the operad $\Comop$ for commutative algebras can be constructed from this. Let $\Comop(n)$ be the free abelian group spanned by one element, $\com{n}$. Let the permutation group act trivially on $\Comop(n)$. Let the partial composition $\com{n}\circ_a\com{m}$ be the sum of extensions of $\com{n}$ by $\com{m}$ at $a$; since there is only one such extension, this is just
\[
\com{n}\circ_a\com{m} = \com{n+m-1} .
\]

Other operads can be constructed from categories over $\Comcat$ in a less trivial way.

\subsection{The $\Brace$ operad}
In this paper, ``tree'' will always refer to a planar rooted tree. There are many ways of defining these, so here is mine:
\begin{definition}
\label{Tree}
A \emph{tree} $T=(\V_T,\E_T,\trianglelefteq_T)$ consists of a finite set $\V_T$ (the \emph{vertices}) a subset $\E_T\subset\V_T\times\V_T$ (the \emph{edges}), and a partial order $\trianglelefteq_T$ on $\V_T$ such that:
\begin{enumerate}
\item
For all $u\in \V_T$, $(u,u)\notin\E_T$.
\item
$\E_T$ generates a partial order, $\leq_T$, such that $(u,v)\in\E_T\implies u<_Tv$.
\item
There exists a unique $\leq_T$-minimal element (called \emph{the root of $T$}).
\item
\label{Minimality}
If $(u,v)\in\E_T$, then there does not exist any $t\in\V_T$ such that $u<_Tt<_Tv$.
\item
\label{Comparability}
Two vertices are $\vartriangleleft_T$-comparable if and only if they are $\leq_T$-incomparable. (I.e., $\forall u,v\in\V_T$, precisely one is true of: $u=v$, $u<_Tv$, $u>_Tv$, $u\vartriangleleft_Tv$, and $u\vartriangleright_Tv$.)
\item
\label{Inheritance}
If $u \vartriangleleft_T v$, $u\leq_T u'$, and $v\leq_T v'$, then $u'\vartriangleleft_T v'$.
\end{enumerate}
\end{definition}
\begin{definition}
If $(u,v)\in\E_T$, then $u$ is the \emph{parent} of $v$ and $v$ is a \emph{child} of $u$.  
\end{definition}

Think of these relations as 
\begin{gather*}
\text{above}\leq_T\text{below}\\
\text{left}\trianglelefteq_T\text{right}.
\end{gather*}
 The roots are at the top and the leaves are at the bottom, so this is botanically incorrect.

\begin{remark}
The partial order  $\trianglelefteq_T$ restricts to a total order on the leaves of $T$. It also restricts to a total order on the set of children of any vertex. Those total orders are other ways of defining a planar rooted tree.
\end{remark}

\begin{remark}
These axioms imply that a vertex has at most one parent. Suppose that $v$ has two parents, $u\neq w$. By \ref{Tree}\eqref{Minimality}, these are not $\leq_T$-related. By \ref{Tree}\eqref{Comparability}, they must be $\vartriangleleft_T$-related; say $u\vartriangleleft_T w$. By \ref{Tree}\eqref{Inheritance}, this implies that $v\vartriangleleft_T v$, which is a contradiction.
\end{remark}

\begin{definition}
\label{Treecat def}
$\Treecat(n)$ is the set of trees with vertex set $\com{n}$, and
\[
\Treecat := \bigcup_{n\geq1} \Treecat(n) .
\]
For $T\in\Treecat(n)$, the \emph{arity} is $\#T:=n$, so $\V_T=\com{\#T}$.

For $T\in\Treecat(n)$ and a permutation $\sigma\in\perm_n$, $T^\sigma\in\Treecat(n)$ is defined by replacing any $u\in\com{n}$ with $\sigma^{-1}(u)$.

$\Brace$ is the $\perm$-module \cite{l-v2012} that is the free abelian group spanned by $\Treecat$.
\end{definition}

$\Treecat$ is the set of objects of a category. For two trees $S,T\in\Treecat$, a morphism $\alpha:S\to T$ is a $\Comcat$-morphism $\alpha:\com{\#S}\to\com{\#T}$ such that
\begin{enumerate}
\item
\label{Tree morpism1}
$(u,v)\in\E_S$ $\implies$ $(\alpha u,\alpha v)\in\E_T$ or $\alpha u = \alpha v$.
\item
\label{Tree morphism2}
$\alpha u \vartriangleleft_T \alpha v$ $\implies$ $u\vartriangleleft_S v$.
\end{enumerate}

However,  all that we will need from this category is the following notion of  an extension. 
\begin{definition}
\label{Tree extension}
Let $\com{n} \stackrel\alpha\into \com{n+m-1} \stackrel\beta\onto \com{m}$ be the $\Comcat$-extension at $a\in\com{m}$ as in Definition~\ref{Ext def}. 
Given $S\in\Treecat(m)$ and $T\in\Treecat(n)$, \emph{an extension of $S$ by $T$ at $a$} is a tree $U\in\Treecat(n+m-1)$ such that:
\begin{itemize}
\item
$T$, relabelled by $\alpha$, is a subtree of $U$ in the sense that 
\[
(r,s)\in\E_T \iff (\alpha t,\alpha s)\in \E_U
\]
and
\[
r\vartriangleleft_T t \iff \alpha r \vartriangleleft_U \alpha t ;
\]
\item
$S$ is the quotient of $U$ by the image of $T$, relabelled by $\beta$ in the sense that 
\[
(u, w)\in\E_U \implies (\beta u ,\beta w)\in\E_S\text{ or }\beta u = \beta w = a
\]
and
\[
\beta u \vartriangleleft_U \beta w \implies u\vartriangleleft_T w   .
\]
\end{itemize} 

Denote the set of such extensions as $\Ext(S,T,a)$.
\end{definition}
Essentially, this means that $T$ is identified with a subtree of $U$, which is contracted to a single vertex, $a$ to give $S$.

\begin{remark}
An extension is called a \emph{greffe} (graft) in \cite{cha2002,foi2002}.
\end{remark}

\begin{definition}
\label{Tree composition}
For $S,T\in\Treecat$ and $a\in\com{\#S}$, the \emph{partial composition} is
\[
S \circ_a T = \sum_{U\in\Ext(S,T,a)} U \ \in\Brace\ .
\]
\end{definition}
With this operation, $\Brace$ is an operad (of abelian groups). A representation of $\Brace$ is a brace algebra \cite{g-v1995}.
\begin{remark}
Note that $\Brace$ is ungraded, so it can be thought of as a graded operad concentrated in degree $0$. This is inconsistent with the most obvious grading on the Hochschild complex $C^\bullet(A,A)$ of an associative algebra, $A$. Instead, $\Brace$ acts on the suspension, $\s\,C^\bullet(A,A)$, which is the same vector space with the grading decreased by $1$.
The alternative would be to desuspend $\Brace$, which is messier.
\end{remark}
The eponymous ``brace'' operations are given by the corollas in $\Treecat$. These generate $\Brace$ as an operad.

\subsection{The Gerstenhaber-Voronov operad}
In this section, I present a dg-operad, which I denote as $\Wordop$. This has appeared in several forms in the literature. It is the operad governing ``homotopy G-algebras'' as defined by Gerstenhaber and Voronov \cite{g-v1995}. It is the operad of ``spineless cacti'' \cite{g-l-t2015,kau2007}. It is the second term in a filtration by suboperads of the ``sequence'' operad of McClure and Smith \cite{m-s2002,m-s2003}. The sequence operad was renamed ``surjection operad'' by Berger and Fresse \cite{b-f2002}. The last description will be the most useful here.

\subsubsection{Words}
\label{Words sec}
\begin{definition}
\label{Word}
Given a set $A$ (the \emph{alphabet}),
a \emph{word} $W$ over $A$ is a finite, totally ordered set with a function $\pi_W:W\to A$. 
The \emph{length} of $W$ is the cardinality $\abs W$.
I will denote elements of $W$ with an underscore such that $\pi_W(\underline a) = a$, and call $\underline a$ an \emph{occurrence} of $a\in A$.
A \emph{letter} is an element of $W$, or equivalently an element of $A$ if $\pi_W$ is a bijection.
$W$ will be written as the sequence of values of $\pi_W$  (not separated by commas). 
The \emph{concatenation} of words $W$ and $V$ is the disjoint union with $W$ before $V$, denoted $W\smile V$.
\end{definition}

\begin{example}
Given a (planar rooted) tree, $T$, the \emph{corner word}, $\Place_T$, is a word over $\V_T$. \label{Corner word}
Imagine embedding $T$ in the closed lower half plane with the root on the boundary. Then start from the left and go around $T$ counterclockwise. $\Place_T$ is the list of each vertex visited in order. 

This has the following properties:
\begin{itemize}
\item
If $(u,v)\in\E_T$, then $\Place_T = \dots uv\dots vu\dots$  or $\dots uvu\dots$.
\item
$\abs{\Place_T} = \abs{\V_T} + \abs{\E_T}$.
\item
$\Place_T = \dots u \dots v\dots u\dots$ if and only if $u\leq_T v$.
\item
If $\Place_T = \dots u \dots v \dots$, then $u \ntriangleright_T v$.
\end{itemize}
$\Place_T$ is the set of places where an additional edge and vertex could be attached to $T$.
\end{example}
\begin{remark}
In \cite{cha2002,foi2002}, this is the set of ``\emph{angles}''. I have translated this as ``corners''.
\end{remark}

\subsubsection{The allowed words}
\begin{definition}
\label{Wordcat def}
Let $\Wordcat(n)$ be the set of (isomorphism classes of) words $W$ over $\com{n}$ such that 
\begin{enumerate}
\item
each element of $\com{n}$ appears in $W$ (surjectivity), 
\item
$W\neq \dots uu\dots$ (nondegeneracy), and
\item
for any $u\neq v\in\com{n}$, $W\neq \dots u\dots v\dots u\dots v\dots$ (no interlacing).
\end{enumerate}
\end{definition}
In this context, I will refer to the elements of $\com{n}$ as \emph{vertices}. In the next section, they will be the vertices of trees.
\begin{definition}
Let
\[
\Wordcat = \bigcup_{n\geq1}\Wordcat(n) .
\]

The \emph{arity} of $W\in\Wordcat(n)$ is $\#W:=n$.

\label{Degree def}
The \emph{degree} of  $W\in\Wordcat(n)$ is $\deg W := \abs{W}-n$.

For $W\in\Wordcat(n)$ and $\sigma\in\perm_n$, $W^\sigma\in\Wordcat(n)$ is defined by replacing any $u\in\com{n}$ by $\sigma^{-1}(u)$, i.e., $\pi_{W^\sigma}=\sigma^{-1}\circ\pi_W$.

\label{Wordop def}
$\Wordop$ is the graded $\perm$-module that is a free graded abelian group spanned by $\Wordcat$ with grading $\deg$.
\end{definition}
\begin{remark}
I am denoting the operad in italics and its basis in sans serif.
\end{remark}

\begin{definition}
\label{Down order}
For $W\in\Wordcat(n)$, the \emph{$\downarrow$ ordering} of $\com{n}$ is the (total) order of first occurrence in $W$.
\end{definition}

\begin{definition}
 If $W=\dots u  v \dots u\dots$, then $v$ is \emph{interposed} in $W$.

A letter $\underline u\in W$ is a \emph{caesura} \cite{b-f2002} if there is a later occurrence of $u$ in $W$.

A pair of successive letters $\underline u$ and $\underline v \in W$ is a \emph{last-first pair} if $\underline u$ is the last occurrence of $u$ and $\underline v$ is the first occurrence of $v$.
\end{definition}

\begin{remark}
If $W = \dots \underline u\,\underline v\dots\in\Wordcat$, then there are three possibilities:
\begin{itemize}
\item
$W = \dots v \dots \underline u\,\underline v\dots$, in which case $\underline u$ must be the last occurrence of $u$ (i.e., not a caesura).
\item
$W = \dots \underline u\,\underline v\dots u\dots$, which means precisely that $\underline u$ is a caesura and $v$ is an interposed vertex. This shows that there is a 1-1 correspondence between caesurae and interposed vertices (although caesurae are letters of $W$ and interposed vertices are elements of $\com{\#W}$). They always occur paired like this.
\item
$\underline u\,\underline v$ is a last-first pair.
\end{itemize}
\end{remark}

Every vertex occurs last precisely once, so $\deg W$ is the number of caesurae in $W$. Because of the 1-1 correspondence, this is also the number of interposed vertices.

\begin{lemma}
\label{Word length}
If\/ $W\in\Wordcat(n)$, and $s$ is the number of last-first pairs in $W$, then $\deg W = n-s-1$. That is, $\abs{W} = 2n-s-1$.

If\/ $W$ can be written as a concatenation of $m$ words with disjoint alphabets then $\deg W \leq n-m$.
\end{lemma}
\begin{proof}
$W\in\Wordcat(n)$ is combinatorially equivalent to a spineless cactus \cite{g-l-t2015} which is a contractible cellular complex in the plane. The numbers $1,\dots,n$ label the \emph{lobes} (2-cells). Each lobe has a base vertex (0-cell), and each of these is the base of at least one lobe.
 Each arc (1-cell) is on the boundary of a unique lobe; the word $W$ is defined by traversing around the cactus and noting the lobe for each arc, so $\abs W$ is the number of arcs. Consecutive lobes $a$ and $b$ share a base vertex if the last $a$ is followed by the first $b$ in $W$, so the number of vertices is $n-s$.

Since the cactus is contractible, its Euler characteristic is $1 = n-s - \abs W + n$, so
\[
\abs W = 2n-s-1 
\]

If $W$ can be written as a concatenation of $m$ subwords, then the last letter of a subword is a final occurrence, and the first letter of the next subword is a first occurrence. This shows that $s\geq m-1$, so
\[
\deg W = n-s-1 \leq n-m . \qedhere
\]
\end{proof}

\subsubsection{Boundary}
\begin{definition}
\label{Face def}
Given a word, $W\in\Wordcat$ and $\underline a\in W$ such that $a$ is repeated in $W$, the \emph{face} $\partial_{\underline a}W\in\Wordcat(n)$ is the word given by deleting $\underline a$ from $W$. If $a$ is not repeated, then set $\partial_{\underline a}W=0\in\Wordop$.
\end{definition}
\begin{remark}
If $a$ is repeated, then there is no subword of the form $uau$, so deleting an $a$ cannot create a consecutive repetition. This is why the faces are still words in $\Wordcat$.
\end{remark}
\begin{definition}
\label{Face sign}
Given $W\in\Wordcat$, define $\sgn_W: W\to\{1,-1\}$ as follows.
Number the caesurae in $W$ consecutively, starting from $1$.
If $\underline a\in W$ is the  $k$'th caesura, then  $\sgn_W \underline a = (-1)^{k}$.
If $\underline a$ is the last occurrence of $a$, but the previous occurrence of $a$ is the $k$'th caesura, then $\underline a$ is numbered $k+1$ and $\sgn_W\underline a=(-1)^{k+1}$. If $\underline a$ is the only occurrence of $a$, then it is numbered $1$ and $\sgn_W(\underline a)=-1$.
\end{definition}
\begin{definition}
\label{word boundary}
The \emph{boundary} of $W\in\Wordcat$ is
\[
\partial W = \sum_{\underline a \in  W}\sgn_W(\underline a)\, \partial_{\underline a}W \ \in\Wordop\ .
\]
\end{definition}
Clearly, $\deg\partial W = \deg W -1$.
\begin{remark}
This definition of $\partial$ differs from that in \cite{b-f2002} by an overall negative sign.
\end{remark}

\begin{example}
For $123242151\in\Wordcat(5)$, the numbers associated to the occurrences of repeated vertices are
\[
\underset{\color{blue}1}1\; \underset{\color{blue}2}2\; 3\; \underset{\color{blue}3}2 \;4\; \underset{\color{red}4}2\; \underset{\color{blue}4}1\; 5\; \underset{\color{red}5}1 ,
\]
so the boundary is
\begin{multline*}
\partial 123242151 = -23242151 + 13242151 - 12342151 \\ + 12324151 + 12324251 - 12324215 .
\end{multline*}
\end{example}
\begin{example}
For the word $123432151 \in\Wordcat(6)$,
\[
\underset{\color{blue}1}1\; \underset{\color{blue}2}2\;\underset{\color{blue}3}3\;4\;\underset{\color{red}4}3\,\underset{\color{red}3}2\;\underset{\color{blue}4}1\;5\;\underset{\color{red}5}1 ,
\]
so the boundary is
\begin{multline*}
\partial 123432151 = - 23432151 + 13432151 - 12432151 \\ + 12342151 - 12343151 + 12343251 - 12343215 .
\end{multline*}
\end{example}

\subsubsection{Composition}
Like $\Treecat$, $\Wordcat$ is the set of objects of a category over $\Comcat$. For two words $V,W\in\Wordcat$, a morphism $\alpha:V\to W$ is a $\Comcat$-morphism $\alpha:\com{\#V}\to\com{\#W}$ such that:
\begin{itemize}
\item
The following constructions produce the same word:
\begin{itemize}
\item
Delete from $W$ all vertices not in the image of $\alpha$, then eliminate any consecutive repetitions ($uu\mapsto u$).
\item
Apply $\alpha$ to every letter of $V$ and eliminate any consecutive repetitions.
\end{itemize}
\item
In the first construction, consecutive repetitions occur wherever letters have been deleted between vertices from the image of $\alpha$.
\end{itemize}

Note that the corner word construction defines a functor $\Place:\Treecat\to\Wordcat$. 

As for $\Treecat$, we will only need the extensions from this category.
\begin{definition}
\label{Word extension}
Again, let  $\com{n} \stackrel\alpha\into \com{n+m-1} \stackrel\beta\onto \com{m}$ be the $\Comcat$-extension at $a\in\com{m}$ from Definition~\ref{Ext def}.
Given $V\in\Wordcat(m)$ and $W\in\Wordcat(n)$,  \emph{an extension of $V$ by $W$ at $a$} is a word $X\in\Wordcat(n+m-1)$ such that:
\begin{itemize}
\item
deleting from $X$ the vertices not in the image of $\alpha$ and eliminating consecutive repetitions ($uu\mapsto u$) gives the same word as $W$ relabelled by $\alpha$;
\item
these repetitions occur wherever letters have been deleted;
\item
relabelling $X$ by $\beta$ and eliminating consecutive repetitions gives $V$.
\end{itemize}

$\Ext(V,W,a)$ denotes the set of extensions of $V$ by $W$ at $a$.
\end{definition}

\begin{remark}
If $S,T\in\Treecat$ and $a\in\com{\#S}$, then
\[
\Ext(\Place_S,\Place_T,a) = \{\Place_U \mid U\in\Ext(S,T,a)\} .
\]
\end{remark}

\begin{definition}
Suppose that $X\in\Ext(V,W,a)$. Let $\com{\#W} \stackrel\alpha\into \com{\#X} \stackrel\beta\onto \com{\#V}$ be the $\Comcat$-extension at $a$. Identify interposed vertices of $V$ and $W$ with those of $X$ as follows:
\begin{itemize}
\item
If $w\in\com{\#W}$ is interposed in $W$, then identify this with $\alpha w \in\com{\#X}$.
\item
If $v\in\com{\#V}$ is interposed in $V$ and $v\neq a$, then identify $v$  with $\beta^{-1}v\in \com{\#X}$.
\item
If $a$ is interposed in $V$, and $w$ is the first letter of $W$, then identify $a$  with $\alpha w$.
\end{itemize}
\label{Ext sign}
Define $\sgn_{V,W,v}(X)$ as the sign of the $(\deg V,\deg W)$-shuffle from
\begin{itemize}
\item
the interposed vertices of $V$ in $\downarrow$ order, identified with vertices of $X$, then the interposed vertices of $W$, identified with vertices of $X$
\item
to the interposed vertices of $X$ in $\downarrow$ order.
\end{itemize}
This defines
\[
\sgn_{V,W,a}:\Ext(V,W,a)\to\{1,-1\} .
\]
\end{definition}
\begin{remark}
If $\deg V=0$ or $\deg W =0$, then $\sgn_{V,W,a}(X)=+1$, since the shuffle must be trivial. This obviates many computations.
\end{remark}

\begin{definition}
\label{Word composition}
For $V,W\in\Wordcat$ and $a\in\com{\#V}$, the \emph{partial composition} is
\[
V\circ_aW = \sum_{X\in\Ext(V,W,a)} \sgn_{V,W,a}(X)\, X \ \in\Wordop \ .
\]
\end{definition}
This makes $\Wordop$ a dg-operad.


The construction $\Place:\Treecat\to\Wordcat$, along with appropriate signs, gives an injective operad homomorphism from $\Susp^{-1}\Brace$ to $\Wordop$. Consequently, any  homotopy G-algebra is in particular a brace algebra. Historically, these structures originally appeared together in \cite{g-v1995}.

\subsubsection{Computing extensions}
\label{Computing extensions}
Given $V,W\in\Wordcat$ and $a\in\com{\#V}$, any extension of $V$ by $W$ at $a$ is uniquely determined by a weakly increasing function, $\kappa$, from  the set of caesura occurrences of $a$ in $V$ to $W$.

Let $\com{\#W}\stackrel\alpha\into \com{\#V+\#W-1} \stackrel\beta\onto \com{\#W}$ be the $\Comcat$-extension at $a$. To each occurrence of $\underline a \in W$, associate a subword of $V$ that
\begin{itemize}
\item
begins with $\kappa(\underline a')$ where $\underline a'$ is the previous occurrence, or otherwise at the beginning of $V$,
\item
ends with $\kappa(\underline a)$ if $\underline a$ is a caesura, or otherwise at the end of $V$.
\end{itemize}
Construct $X\in\Ext(W,V,a)$ by replacing each occurrence of $u\neq a$ in $W$ with $\beta^{-1}(u)$ and replacing each occurrence of $a$ in $W$ with the associated subword of $V$, relabelled by $\alpha$.

If $a$ occurs $r+1$ times in $W$, then 
\[
\Abs{\Ext(W,V,a)} = \binom{\abs{V}+r-1}{r} .
\]

In the case of trees, the caesura occurrences of $a$ in $S$ correspond to the children of $a$. The idea of an extension of $S$ by $T$ at $a$ is to (after relabelling) replace $a$ by $T$. The word $\Place_T$ is the list of places where the edges coming from $a$ can be reattached to $T$. There are $\abs{\Place_T} = 2\,\#T-1$ of these, so if $a$ has valency $r$ (i.e., $r$ children) then
\[
\Abs{\Ext(S,T,a)} = \binom{2\,\#T+r-2}{r} .
\]

\section{The $\Qop$ Operad}
\label{Quilt section}
In this section, I define a new operad, $\Qop$.
\subsection{Quilts}
The \emph{Hadamard product} \cite[Sec.~5.3.2]{l-v2012} of $\Wordop$ and $\Brace$ is a dg-operad $\Wordop\Had\Brace$ with 
\[
\Bigl(\Wordop\Had\Brace\Bigr)(n) = \Wordop(n)\otimes\Brace(n) 
\]
as abelian groups. Since the factors are spanned by  $\Treecat(n)$ and $\Wordcat(n)$, this product is generated by
\[
\Wordcat(n) \times \Treecat(n)  .
\]
Since $\Brace$ is ungraded, the grading of the product is defined by
\[
\deg (W,T) := \deg W ,
\]
and the boundary operator is
\[
\partial (W,T) := \partial W \otimes T .
\]
The action of $\perm_n$ is diagonal:
\[
(W,T)^\sigma := (W^\sigma,T^\sigma) .
\]
The partial composition in the Hadamard product is simply
\[
(V,S) \circ_a (W,T) := (V\circ_aW)\otimes(S\circ_a T) .
\]
Note that this definition doesn't have any further signs, because $\Brace$ is ungraded.

It will be convenient to extend notation from $\Wordop$ to the Hadamard product. Let 
\beq
\label{Q boundary}
\partial_{\underline a}(W,T) := (\partial_{\underline a}W,T)
\eeq
if $a$ is repeated in $W$ and $\partial_{\underline a}(W,T):=0$ if it isn't, and $\sgn_{(W,T)}(\underline a) := \sgn_W\underline a$, so that
\[
\partial (W,T) = \sum_{\underline a\in W} \sgn_{(W,T)}(\underline a)\,\partial_{\underline a}(W,T) .
\]

\begin{definition}
\label{Quilt def}
$\Qcat(n)$ is the set of $(W,T) \in \Wordcat(n)\times \Treecat(n)$ such that
\begin{enumerate}
\item\label{Quilt1}
$W=\dots u \dots v\dots \implies u \ngtr_T v$
\item\label{Quilt2}
$W=\dots u \dots v \dots u\dots \implies v \trianglelefteq_T u$.
\end{enumerate}

A \emph{quilt} is an element of
\[
\Qcat := \bigcup_{n\geq1}\Qcat(n) ,
\]
and $\Qop\subset \Wordop\Had\Brace$ is the abelian subgroup spanned by $\Qcat$.

Denote $Q = (\W_Q,\Tree_Q) \in\Qcat$. Any definition or notation for $\Treecat$ and $\Wordcat$ extends to $\Qcat$, e.g., $\leq_Q$ means $\leq_{\Tree_Q}$.
\end{definition}
\begin{remark}
These conditions are similar to the properties of $\Place_T$ (Sec.~\ref{Corner word}) but the roles of the partial orders are switched.
\end{remark}
\begin{remark}
Condition (2) alone implies the no interleaving property  of $W$ in Definition~\ref{Wordcat def}, so $\Qop$ could have been defined just as easily inside the Hadamard product of $\Brace$ with the (full) surjection operad, rather than just $\Wordop\subset\mathit{S}$.
\end{remark}

\begin{example}
If $\deg Q=0$, then $\W_Q$ is just a total ordering of $\com{\#Q}$ such that $(u,v)\in\E_Q$ implies that $u$ precedes $v$ in $\W_Q$, i.e., a total order compatible with $\leq_Q$.
\end{example}

\begin{lemma}
\label{Face quilt}
If $Q \in\Qcat(n)$ and $a$ is repeated in $\W_Q$, then 
\[
\partial_{\underline a}Q\in\Qcat(n) ,
\]
hence $\partial Q\in\Qop$.
\end{lemma}
\begin{proof}
Recall that $\W_{\partial_{\underline a}Q}=\partial_{\underline a}\W_Q$ is defined by deleting $\underline a \in\W_Q$.
If a certain subsequence occurs in $\partial_{\underline a}\W_Q$, then it must also occur in $\W_Q$, so $\partial_{\underline a}Q$ cannot fail to satisfy Definition~\ref{Quilt def}.
\end{proof}

\begin{lemma}
\label{Quilt extension lemma}
For $P,Q\in\Qcat$, if   $W\in\Ext(\W_P,\W_Q,a)$ and $T\in\Ext(\Tree_P,\Tree_Q,a)$, then $(W,T)\in \Qcat$, and  therefore, $P\circ_a Q\in\Qop$.
\end{lemma}
\begin{proof}
Let $m=\#P$ and $n=\#Q$. Let $\com{n}\stackrel\alpha\into\com{n+m-1}\stackrel\beta\onto\com{m}$ be the $\Comcat$-extension at $a$. 

First consider 2 vertices in the image of $\alpha$. Let $s,t\in\com{n}$. If $W = \dots\alpha s\dots\alpha t\dots$, then we must have $\W_Q = \dots s\dots t\dots$, so $s \ngtr_Q t$. By Definition~\ref{Tree extension}, $\alpha$ strictly preserves the relations, therefore $\alpha s\ngtr_T\alpha t$.

If $W=\dots\alpha s\dots\alpha t\dots\alpha s\dots$, then $\W_Q = \dots s\dots t\dots s\dots$, so $s\trianglelefteq_Q t$. Since $\alpha$ preserves the relations, $\alpha s\trianglelefteq_T \alpha t$.

That covers the cases of two vertices in the image of $\alpha$, which is also $\beta^{-1}a$, so now consider $u,v\in\com{n+m-1}$ with $\beta u \neq \beta v$. If $W = \dots u\dots v\dots$, then $\W_P = \dots\beta u\dots\beta v\dots$, so $\beta u\ngtr_P\beta v$. We have assumed these are not equal, so $\beta u\ngeq_P\beta v$. By Definition~\ref{Tree extension}, $\beta$ respects the partial orders, so contrapositively, $u\ngeq_T v$, and hence $u\ngtr_Tv$. This shows \ref{Quilt def}\eqref{Quilt1}.

If $W = \dots u\dots v\dots u\dots$, then $\W_P = \dots\beta u\dots\beta v\dots\beta u\dots$. 
This implies that $\beta u\trianglelefteq_P\beta v$, but we assumed that these are not equal, so $\beta u\vartriangleleft_P\beta v$. By Axiom~\ref{Tree}\eqref{Comparability}, $u$ and $v$ are related in precisely one of 5 possible ways; since $\beta$ respects the partial orders, all of the possible relations are contradicted except $u\vartriangleleft_T v$, and so $u\trianglelefteq_T v$. This shows \ref{Quilt def}\eqref{Quilt2}.
\end{proof}
\begin{definition}
\label{Quilt extension}
$\Ext(P,Q,a) :=  \Ext(\W_P,\W_Q,a)\times \Ext(\Tree_P,\Tree_Q,a)$ and $\sgn_{P,Q,a}:=\sgn_{\W_P,\W_Q,a}$.
\end{definition}
\begin{theorem}
$\Qop \subset \Wordop\Had\Brace$ is a dg-suboperad.
\end{theorem}
\begin{proof}
The definition of a quilt is invariant under permutations of the labels of the vertices, therefore $\Qop$ is a sub-$\perm$-module.

Lemma~\ref{Face quilt} shows that $\Qop$ is closed under $\partial$, so it is a differential graded submodule.

Lemma~\ref{Quilt extension lemma} shows that $\Qop$ is closed under  partial composition.
\end{proof}

\subsection{Diagrams}
The name ``quilt'' is based on the following diagrams. A diagram consists of a rectangle subdivided into labelled rectangles and shaded regions, with some vertical lines between rectangles drawn as double lines. There is one rectangle across the top of the diagram. Any rectangle has at most one other rectangle above it and at most one other rectangle to the right.

A quilt $Q$ defines a diagram as follows. Begin with a rectangle subdivided into a grid of squares. The columns correspond to the leaves of $Q$, in $\vartriangleleft_Q$ order from left to right. The rows correspond to those vertices that occur in $\W_Q$ only once, in their order in $\W_Q$ arranged from top to bottom. 
Each vertex of $Q$ labels one rectangular region that is a union of these squares; the $u$ rectangle intersects the $v$ column if $u\leq_Q v$; the $u$ rectangle intersects with the $w$ row  if $\W_Q=\dots u\dots w\dots u\dots$. If $\W_Q=\dots uv\dots wu\dots$ (with no $u$ in between) then draw a double line on the left edge of $u$, along the part to the right of $v$, $w$, and anything in between. Otherwise, shade any unfilled regions and bound labelled rectangles by single lines.

A diagram determines a quilt $Q$ as follows. The vertices are the labelled rectangles. Call two rectangles \emph{adjacent} if they touch or are only separated (vertically or horizontally) by a shaded region.  There is an edge  $(u,v)\in\E_Q$ if $v$ is adjacent below $u$. The word $\W_Q$ begins with the vertex at the top of the diagram and is built iteratively. The next vertex in $\W_Q$ after $\underline u$ is determined as:
\begin{itemize}
\item
If there is a rectangle left adjacent to $u$ that is not yet in $\W_Q$, then use the uppermost such rectangle. Otherwise,$\dots$
\item
If $u$ is left adjacent to a double line or the right edge of the diagram --- but not to the bottom of that --- then use the highest unused rectangle adjacent to that double line or right edge.
\item
If $u$ is otherwise left adjacent to $v$, then use $v$.
\item
If $u$ is adjacent to the right edge and is at the bottom of the diagram, then $\W_Q$ is complete.
\end{itemize}
\begin{remark}
What I am calling the $\downarrow$ order of the vertices is given by passing downward through the diagram and moving temporarily left where possible.
\end{remark}

\subsection{Examples}
\begin{example}
\[
\Qwrap{\begin{array}{|c|c|}
\hline
\multicolumn{2}{|c|}{1}\\
\hline
2&\\
\cline{1-1} \hspace{-.333em}\begin{array}{c|c|}3\blankstrut&\blank\\ \hline 4&5\blankstrut\end{array}\hspace{-.167em}& \raisebox{1ex}{6} \\
\hline
\end{array}}
=
\left(162635456\;,
\raisebox{-.5\height}{\begin{tikzpicture}[scale=.7,node font=\footnotesize]\node{$1$}
child{node{$2$}
child{node{$3$}
child{node{$4$}}}
child{node{$5$}}}
child{node{$6$}};
\end{tikzpicture}} 
\right)
\]
\end{example}

\begin{example}
\[
\Qfourd1234 = \left(143234\;,\raisebox{-.5\height}{\begin{tikzpicture}[scale=.7,node font=\footnotesize]\node{$1$}
child{node{$2$}}
child{node{$3$}}
child{node{$4$}};
\end{tikzpicture}} 
\right) .
\]
The occurrences of repeated vertices are numbered,
\[
1\;\underset{\color{blue}1}4\;\underset{\color{blue}2}3\;2\;\underset{\color{red}3}3\;\underset{\color{red}2}4 .
\]
This means that the boundary of the word is $\partial 143234 = -13234+14234-14324+14323$. Each of these words is paired with the same tree. In diagrams, this gives
\[
\partial \Qfourd1234 = -\Qsevenb123\blank\blank\blank4 + \Qsixd12\blank4\blank3 - \Qsixd1\blank342\blank + \Qsevenb1\blank\blank423\blank .
\]
\end{example}

\begin{example}
Consider
\beq
Q = \Qthreea132 = 
\left(1232\;,
\raisebox{-.5\height}{
\begin{tikzpicture}[scale=.7,node font=\footnotesize]
\node{$1$}
child{node{$3$}}
child{node{$2$}};
\end{tikzpicture}} 
\right)
\label{Qexample} 
\eeq
and compute $Q\circ_2 Q$.
Because $2$ is a leaf, there is only one extension of $\Tree_Q$ by $\Tree_Q$ at $2$:
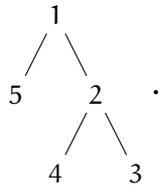
\beq
\Tree_Q\circ_2\Tree_Q=
\raisebox{-.5\height}{
\begin{tikzpicture}[scale=.7,node font=\footnotesize]
\node{$1$}
child{node{$5$}}
child{node{$2$}
child{node{$4$}}
child{node{$3$}}};
\end{tikzpicture}} .
\label{Texample}
\eeq
On the other hand, $2$ occurs once as a caesura in $\W_Q$, so each quilt extension is given by choosing a letter of $\W_Q$, cutting there, and inserting the pieces in place of the two occurrences of $2$ in $\W_Q$ (with appropriate relabelling). 

Let $\com3 \stackrel\alpha\into \com5 \stackrel\beta\onto\com3$ be the $\Comcat$-extension at $2$.
The only interposed vertex in $Q$ is $3$. The $3$ from the first $Q$ is identified with $\beta^{-1}(3) = 5$ in the extension. The $3$ from the second $Q$ is identified with $\alpha(3)=4$. Observe that the first appearances of $5$ and $4$ are in that order in $1252343$ and $1235343$, but in reversed order in $1234543$ and $1234353$, therefore
\[
1232 \circ_2 1232 = 1252343 + 1235343 - 1234543 - 1234353 .
\]
Each term is paired with the same tree \eqref{Texample} so
\[
\Qthreea132\circ_2\Qthreea132 = \Qsixc152\blank43 + 
\Qeightb1\blank25\blank3\blank4 - 
\Qwrap{\begin{array}{|c|c|c|}
\hline
\multicolumn{3}{|c|}{1}\\
\hline
\blank&\multicolumn{2}{|c|}{2}\\
\hline
5&4&3\\
\hline
\end{array}} -
\Qeightb1\blank2\blank435\blank .
\]
\end{example}
\begin{example}
Finally, consider $Q\circ_1 Q$ for the same quilt \eqref{Qexample}. Since $1$ is not repeated, the only extension of $1232$ by $1232$ at $1$ is $1232454$. The only interposed vertex in $Q$ is $3$. The $3$ from the first $Q$ is identified with $5$ and the $3$ from the second $Q$ is identified with $3$. Their order is reversed, therefore
\[
1232 \circ_1 1232 = - 1232454 .
\]

On the other hand, $\Abs{\Place_Q}=5$ and $1$ has 2 children in $Q$, so there are $\binom62=15$ extensions of $\Tree_Q$ by $\Tree_Q$ at $1$. Each of these gives an extension of 
$Q$ by $Q$ at $1$ --- and hence a term of $Q\circ_1Q$ with coefficient $-1$. I will not write them out.
\end{example}

\subsection{Relation to $\Brace$}
$\Qop$ can be thought of as a more elaborate version of $\Brace$. Whereas $\Brace$ acts on the Hochschild complex of a vector space, we will see in the next section that $\Qop$ acts on the Hochschild bicomplex of a diagram of vector spaces.

\begin{definition}
\label{Quotient}
Define $\Qquotient:\Qop\to\Brace$ by
\[
\Qquotient(Q) =
\begin{cases}
\Tree_Q & \deg Q=0\\
0 &\deg Q\neq0
\end{cases}
\]
for any $Q\in\Qcat$.
\end{definition}
For degree $0$ quilts, there are no interposed vertices, so the signs of extensions are all positive. Consequently, $\Qquotient$ is a homomorphism of dg-operads.  Indeed, it comes from a homomorphism $\Wordop\to\Comop$. 

For any tree $T$, if $W$ is the set of vertices arranged in order of first occurrence in $\Place_T$, then $(W,T)\in\Qcat$ and $\Qquotient(W,T)=T$, therefore $\Qquotient$ is surjective.

\begin{conjecture}
\label{Acyclic}
$\Qop$ is acyclic in the sense that $H_k(\Qop)=0$ for $k>0$.
\end{conjecture}
I have verified this up to arity $5$.

If this conjecture is true, then $\Qquotient$ induces an isomorphism from $H_\bullet(\Qop)$ to $\Brace$.

\section{Representation of $\Qop$}
\label{Representation}
If we forget multiplication, a diagram of algebras is a diagram of vector spaces. In this section, I will consider a diagram of vector spaces. This has its own version of Hochschild cohomology, and  I will show that the Hochschild bicomplex of a diagram of vector spaces is a $\Qop$-algebra. Multiplication will return in Section~\ref{mQuilts}.
\subsection{The nerve}
Recall that the \emph{nerve} of a  small category, $\Xc$,   is a simplicial set, i.e., a (contravariant) functor, $B_\bullet\Xc:\Simplex^\op\to\Set$, where $\Simplex$ is again the simplex category. Denote $B_p\Xc:= B_\bullet\Xc([p])$ and $\zeta^*:=B_\bullet\Xc(\zeta)$ for a $\Simplex$-morphism $\zeta$.

In particular, $B_0\Xc = \Obj\Xc$ is the set of objects, $B_1\Xc$ is the set of morphisms, and 
\[
B_2\Xc = \{(\psi,\phi)\in B_1\Xc\times B_1\Xc \mid \exists x,y,z\in B_0\Xc,\ \phi:x\to y,\ \psi:y\to z\}
\]
is the set of pairs of composable morphisms.

Let $\bphi = (\phi_1,\phi_2,\dots,\phi_{p'}) \in B_{p'}\Xc$. More specifically
\[
x_0\xleftarrow{\phi_1} x_1\xleftarrow{\phi_2}\dots\xleftarrow{\phi_{p'}}x_{p'} .
\]
For $0\leq k \leq l\leq p'$, denote $\phi_{k,l}=\phi_{k+1}\circ\dots\circ\phi_{l}$, so that $x_k\xleftarrow{\phi_{k,l}}x_l$. In particular, $\phi_k=\phi_{k-1,k}$.
The action of  $\zeta:[p]\to [p']$ is explicitly, 
\[
\zeta^*\bphi := (\phi_{\zeta(0),\zeta(1)},\dots,\phi_{\zeta(p-1),\zeta(p)}) \in B_p\Xc .
\]

However, we will mainly need the structure of the nerve as a semisimplicial (presimplicial) set, i.e., the restriction of $B_\bullet\Xc$ to the semisimplex category.
\begin{definition}
\label{Semisimplex def}
The \emph{semisimplex} category $\SemiD$ has objects $[p]$ for all $p\in\N$ and morphisms are \textbf{strictly} increasing (increasing and injective) functions.
\end{definition}
Note that a $\SemiD$-morphism is completely described by its codomain and image.

\subsection{The underlying vector spaces}
Let $\Xc$ be a small category, $\Bbbk$ a field of any characteristic, and $\Vect_\Bbbk$ the category of vector spaces over $\Bbbk$.
\begin{remark}
I am using $\Vect_\Bbbk$ for concreteness, but the results are more general. Firstly, the category of modules of any commutative ring would work. 

Note that $\Vect_\Bbbk$ becomes a nonsymmetric colored operad (multicategory) of abelian groups if we define the operations (multimorphisms) to be the multilinear maps. That is the structure that will be used here, and all results are true for an arbitrary nonsymmetric colored operad.
\end{remark}

Let $\A:\Xc\to\Vect_\Bbbk$ be a \emph{diagram of vector spaces}, i.e., a covariant functor. (In Sec.~\ref{Rep mQuilt}, I will consider a diagram of algebras.)
\begin{definition}
\label{Bicomplex def}
Given a diagram of vector spaces $\A:\Xc\to\Vect_\Bbbk$, the bigraded vector space $C^{\bullet,\bullet}(\A,\A)$ is defined by
\[
C^{p,q}(\A,\A) := \prod_{(x_0\leftarrow\dots\leftarrow x_p)\in B_p\Xc} \Hom_\Bbbk[\A(x_p)^{\otimes q},\A(x_0)] .
\]
\end{definition}
This has several relevant gradings.
\begin{definition}
\label{Gradings}
A homogeneous element $f\in C^{p,q}(\A,\A)$ has \emph{bidegree} $(p,q)$.

The total space $\tot C^{\bullet,\bullet}(\A,\A)$ is graded by the \emph{total degree} $p+q$.

The suspension $\s_2 C^{\bullet,\bullet}(\A,\A)$ is graded by the shifted bidegree $\norm f := (p,q-1)$.

The total suspension $\s \tot C^{\bullet,\bullet}(\A,\A)$ is graded by the shifted total degree $\abs f := p+q-1$.
\end{definition}
The shifted bidegree will be most relevant in this section.

The simplicial coboundary $\dS$ has bidegree $(1,0)$ and will be defined below (Def.~\ref{dS def}).

For a diagram of algebras, the Hochschild coboundary $\dH$ has degree $(0,1)$ and will be defined below (Def.~\ref{dH def}). 
The Hochschild cohomology of a diagram of algebras is the cohomology of $\tot C^{\bullet,\bullet}(\A,\A)$ with coboundary $\delta := \dS+\dH$. 

\subsection{Some colored operads}
\label{Colored operads}
\begin{definition}
\label{MultiSemiD def}
$\MultiSemiD$ is an $\N$-colored operad of sets (i.e., a symmetric multicategory with object set $\N$).
For any $n\in\N$, $\bp = (p_1,\dots,p_n)\in\N^n$, and $p'\in\N$, define $\MultiSemiD(n;\bp,p')$ as the set of $n$-tuples of $\SemiD$-morphisms,
\[
\bzeta = (\zeta_1,\dots,\zeta_n)
\]
where
\[
\zeta_a:[p_a]\to[p'] .
\]
$\perm_n$ acts on this by permuting labels. The partial composition is
\[
\bzeta \circ_a \bxi :=
(\zeta_1,\dots,\zeta_{a-1},\zeta_a\circ\xi_1,\dots ,\zeta_a\circ\xi_{\#\bxi},\zeta_{a+1},\dots,\zeta_{\#\bzeta}) .
\]
\end{definition}
\begin{remark}
This construction could be applied to any small category to produce a colored operad.

Dmitri Pavlov and Rune Haugseng have pointed out to me that $\MultiSemiD$ is the symmetric multicategory associated to the symmetric monoidal envelope of $\SemiD$.
\end{remark}
\begin{definition}
\label{NSOp def}
Let $\NSOp$ be the $\N$-colored operad of sets whose algebras are nonsymmetric operads \cite[\S1.5.3]{cav2016}. (See also \cite[\S1.5.6]{b-m2007} for the symmetric version.)
Explicitly, given $n\in\N$, $\bq=(q_1,\dots,q_n)\in\N^n$, and 
\[
q' = 1 + \sum_{a=1}^n (q_a-1)
\]
define $\NSOp(n;\bq,q')$ to be the set of pairs $(\Tree_I,I)$ where $\Tree_I\in\Treecat(n)$ and $I:\E_I\to\N$ such that:
\begin{itemize}
\item
for $(u,v)\in\E_I$, $1\leq I_{(u,v)} \leq q_u$;
\item
$(t,u),(t,v)\in\E_I$, $u\vartriangleleft_I v \implies I_{(t,u)}<I_{(t,v)}$.
\end{itemize}

An element $(\Tree_I,I)\in \NSOp(n;\bq,q')$ describes a way of composing elements $f_1,\dots, f_n$ in an arbitrary nonsymmetric operad, where $\# f_a=q_a$, and an edge $(u,v)\in\E_I$ means the composition $f_u \circ_{I_{(u,v)}} f_v$.
\end{definition}
\begin{remark}
The planarity of $\Tree_I$ is redundant. The relation $\vartriangleleft_I$ can be reconstructed from the numbers.
\end{remark}
\begin{remark}
My explicit description is different from that in \cite{b-m2007,cav2016}. 
To go from that description to mine, number the edges from each vertex in planar order and then delete any edge that does not connect two vertices.
\end{remark}

$\NSOp$ is isomorphic to a suboperad of $\MultiD$. For example
\[
\raisebox{-.5\height}{
\begin{tikzpicture}[scale=.7,node font=\footnotesize]
\node{$1$}
child{node{$2$} edge from parent node[right]{$\scriptstyle i$}};
\end{tikzpicture}}
\in \NSOp(2;q_1,q_2,q_1+q_2-1)
\]
corresponds to $\zeta \in \MultiD(2;q_1,q_2,q_1+q_2-1)$ where
\[
\zeta_1(k)  = 
\begin{cases}
k & 0\leq k\leq i-1\\
k+q_2-1 & i\leq k \leq q_1
\end{cases}
\]
and $\zeta_2(k) = k+ i-1$.

I denote the Hadamard product of operads of sets as $\Hadtimes$, since the Cartesian product is the relevant monoidal operation for $\Set$.
\begin{definition}
\label{CQ def}
$\CQ \subset\MultiSemiD\Hadtimes\NSOp$ is the $\N^2$-colored suboperad consisting of pairs $(\bzeta,I)$ such that for all $(a,b)\in\E_I$, $\max \zeta_a\leq \min\zeta_b$.
\end{definition}

\subsection{The representation of $\CQ$}
Let $\A:\Xc\to\Vect_\Bbbk$ be a diagram of vector spaces.
\begin{definition}
\label{IndQuilt action}
Let $(\bzeta,I) \in \CQ(n;\bp,\bq,p',q')$, $f_a\in C^{p_a,q_a}(\A,\A)$ for all $a\in\com{n}$, and  $\bphi = (\phi_1,\dots,\phi_{p'})$ with $\phi_k:x_k\to x_{k-1}$ in $\Xc$. First note that 
\[
f_a[\zeta_a^*\bphi] : \A(x_{\zeta_a(p_a)})^{\otimes q_a} \to \A(x_{\zeta_a(0)}) .
\] 
Now,  compose all of these, using the functor $\A$ to connect different vector spaces. 
\begin{itemize}
\item
If $u$ is the root of $\Tree_I$, then compose $\A[\phi_{0,\zeta_u(0)}]\circ f_u[\zeta_u^*\bphi]$.
\item
If $(a,b)\in\E_I$, then compose $f_a[\zeta_a^*\bphi] \circ_{I_{(a,b)}} \A[\phi_{\zeta_a(p_a),\zeta_b(0)}]\circ f_b[\zeta_b^*\bphi]$.
\item
For $a\in\com{n}$ and $1\leq j \leq q_a$, if there does \textbf{not} exist $(a,b)\in\E_I$ with $I_{(a,b)}=j$, then compose $f_a[\zeta_a^*\bphi]\circ_j \A[\phi_{\zeta_a(p_a),p'}]$.
\end{itemize}
All of these compositions together define 
\[
\rep(\bzeta,I;f_1,\dots,f_n)[\bphi] : \A(x_{p'})^{\otimes q'} \to \A(x_0) ,
\]
and hence
\[
\rep(\bzeta,I;f_1,\dots,f_n) \in C^{p',q'}(\A,\A) .
\]
This is multilinear by construction, so
\[
\rep[\bzeta,I] : C^{p_1,q_1}(\A,\A)\otimes\dots \otimes C^{p_n,q_n}(\A,\A)\to C^{p',q'}(\A,\A) .\
\]
\end{definition}
\begin{remark}
This is the motivation for Definition~\ref{CQ def}. $\CQ$ is simply the set of pairs for which $\rep$ makes sense.
\end{remark}
\begin{lemma}
\label{rep lemma}
This is a representation of $\CQ$ in the sense that 
\[
\rep[\bzeta,I]\circ_a \rep[\bxi,J] = \rep[\bzeta \circ_a \bxi,I\circ_a J]\ .
\]
\end{lemma}
\begin{proof}
Let $m=\#\bzeta=\#I$ and $n=\#\bxi=\#J$. Since the numbering of vertices is only used to keep track of them, it is apparent that everything is $\perm$-equivariant, so we can without loss of generality consider the partial composition at $1$. This simplifies notation slightly.

To check this, apply both sides to $(f_1,\dots,f_{n+m-1})$ and evaluate the result at $\bphi\in B_\bullet\Xc$. 
On the right hand side, we need to compute,
\begin{multline*}
\left(\rep[\bzeta,I]\circ_1 \rep[\bxi,J]\right)(f_1,\dots,f_{n+m-1})[\bphi] \\
= \rep\bigl(\bzeta,I;\rep(\bxi,J;f_1,\dots,f_n),f_{n+1},\dots,f_{n+m-1}\bigr)[\bphi] .
\end{multline*}
The building blocks here are $\rep(\bxi,J;f_1,\dots,f_n)[\zeta_1^*\bphi]$ and, for $a=n+1,\dots,n+m-1$,
\[
f_a[\zeta_{a-n+1}^*\bphi] = f_a[(\bzeta\circ_1\bxi)_a^*\bphi] .
\]
Applying the definition to $\rep[\bxi,J]$ expresses the first piece in terms of  
\[
f_a[\xi_a^*\zeta_1^*\bphi] = f_a[(\bzeta\circ_1\bxi)^*_a\bphi] ,
\]
for $a=1,\dots,n$.
Comparing this with the definition of $\NSOp$ shows that these pieces are composed according to $I\circ_1J$, so the result is
\[
\rep(I\circ_1J,\bzeta\circ_1\bxi;f_1,\dots,f_{n+m-1})[\bphi] .
\qedhere
\]
\end{proof} 

\begin{remark}
In \cite{g-s1988a,g-s1988b}, diagrams are \emph{contravariant} functors, a.k.a., presheaves. Everything can be translated to that setting by replacing $\Xc$ with the opposite category.
\end{remark}

\subsection{Simplicial coboundary}
Using this representation, $\rep$, of $\CQ$, it is easy to define the simplicial coboundary operator.
\begin{definition}
\label{Face map def}
For any $p\in\N$ and $0\leq i\leq p+1$, the $i$'th \emph{face map} $\varepsilon_i:[p]\to[p+1]$ is the unique $\SemiD$-morphism whose image does not include $i$.

Let {\footnotesize$1$} be the trivial tree with one vertex. This represents the identity element in $\NSOp(1;q,q)$ for any $q\in\N$.
\end{definition}

\begin{definition}
\label{dS def}
\[
\dS_i := \rep[\varepsilon_i,\mbox{\footnotesize$1$}]: C^{p,q}(\A,\A)\to C^{p+1,q}(\A,\A)
\]
and
\[
\dS := \sum_{i=0}^{p+1} (-1)^i\dS_i  .
\]
\end{definition}
\begin{definition}
The \emph{Hochschild bicomplex} of a diagram of vector spaces, $\A$, is $C^{\bullet,\bullet}(\A,\A)$ with coboundary $\dS$.
\end{definition}

\subsection{Coloring}
\begin{definition}
\label{Tree coloring}
For $T\in\Treecat(n)$ and $\bq\in\N^n$, let
\[
\Ind(T;\bq) = \{I\in\NSOp(n;\bq,q') \mid \Tree_I=T\} ,
\]
where $q'=1+\sum_{a=1}^n(q_a-1)$. An element of this set is a \emph{coloring} of $T$.
\end{definition}
\begin{definition}
\label{Word coloring}
For $W\in\Wordop(n)$, $\bp\in\N^n$, and
\[
p' = \sum_{a=1}^n p_a - \deg W ,
\]
let $\Ind(W;\bp)$ be the set of $\bzeta\in\MultiSemiD(n;\bp,p')$ such that $\bigcup_{a=1}^{\#W} \Im\zeta_a = [p']$ and for each $a\in\com{\#W}$, there exists a function $\pi_a:[p_a]\to W$ such that: 
\begin{itemize}
\item
The image of $\pi_a$ is the set of occurrences of $a$ in $W$;
\item
for an $\underline a\in W$, $\zeta_a(\pi_a^{-1}\underline a)$ is an interval (set of consecutive integers);
\item
if $W = \dots \underline a\,\underline b\dots$, then 
\[
\max \zeta_a(\pi_a^{-1}\underline a) = \min \zeta_b(\pi_b^{-1}\underline b) .
\]
\end{itemize}
An element of this set is a \emph{coloring} of $W$.
\end{definition}
\begin{remark}
Every element of $\NSOp$ arises as a coloring, but not every element of $\MultiSemiD$ does. For example, the face map $\varepsilon_i$ is not a coloring of any word.
\end{remark}
\begin{remark}
If $\bzeta\in\Ind(W)$, then it is usually --- but not always --- possible to reconstruct $W\in\Wordcat$ from $\bzeta$. For example, if $\id_{[1]}:[1]\to[1]$ is the identity function, then $(\id_{[1]},\id_{[1]})$ is a coloring of both $12$ and $21$.
\end{remark}
\begin{remark}
$\Ind(W;\bp)$ is in one-to-one correspondence with the set of functions $\omega:W\to\N$ such that for all $a$,
\[
1+p_a = \sum_{\underline a\in W}(1+\omega_{\underline a}) 
\]
(where the sum is over occurrences of $a$) and between any 2 consecutive occurrences of $a$, there is some letter $\underline b$ with $\omega_{\underline b}\neq0$.

Given such a function, $\omega$, let $\omega^-_{\underline a}$ be the sum of $\omega_{\underline b}$ over $\underline b$ before $\underline a$ in $W$. Define $\bzeta \in \Ind(W;\bp)$ by 
\[
\Im \bzeta_a = \bigcup_{\underline a \in W} \{\omega^-_{\underline a},\omega^-_{\underline a}+1,\dots,\omega^-_{\underline a}+\omega_{\underline a}\}, 
\]
where the union is over occurrences of $a$.

Conversely, $\omega_{\underline a} = \Abs{\pi_a^{-1}(\underline a)}-1$.
\end{remark}
\begin{lemma}
For any $\bzeta\in\Ind(W;\bp)$, the functions $\pi_a$ in Definition~\ref{Word coloring} are unique.
\end{lemma}
\begin{proof}
Consider some $\underline a\in\W$. 

If $\underline a$ is the first occurrence (of $a$) then $\min\pi^{-1}_a(\underline a) = 0$. 
If $W = \dots a\dots \underline b\,\underline a\dots$, then $\underline b$ is the final occurrence and 
\[
\min\pi^{-1}_a(\underline a) = \max \zeta_b = \zeta_b(p_b) .
\]

If $\underline a$ is the final occurrence, then $\max\pi^{-1}_a(\underline a) = p_a$. 
If $W= \dots \underline a\,\underline c\dots a\dots$, then $\underline c$ is the first occurrence and 
\[
\max\pi^{-1}_a(\underline a) = \min \zeta_c = \zeta_c(0) .
\]

Since $\pi^{-1}_a(\underline a)$ must be a set of consecutive integers, this determines it uniquely.
\end{proof}
\begin{lemma}
\label{Coloring completeness}
Let $V,W\in\Wordcat$, $a\in\com{\#V}$, $\bp\in\N^{\#V+\#W-1}$, and $p'=p_a+\dots+p_{a+\#W-1}$. If $\bzeta\in\Ind(V;p,\dots,p_{a-1},p',p_{a+\#W},\dots,p_{\#V+\#W-1})$ and $\bxi\in\Ind(W;p_a,\dots,p_{a+\#W-1})$, then there exists a unique $X\in\Ext(V,W,a)$ such that $\bzeta\circ_a\bxi\in\Ind(X;\bp)$. 

Conversely, if $X\in\Ext(V,W,a)$ and $\etab\in\Ind(X;\bp)$, then there exist 
\[
\bzeta\in\Ind(V;p,\dots,p_{a-1},p',p_{a+\#W},\dots,p_{\#V+\#W-1})
\] 
and $\bxi\in\Ind(W;p_a,\dots,p_{a+\#W-1})$ such that $\etab=\bzeta\circ_a\bxi$.
\end{lemma}
\begin{proof}
Again, let $\com{\#W} \stackrel\alpha\into \com{\#V+\#W-1} \stackrel\beta\onto \com{\#V}$ be the $\Comcat$-extension at $a$. 

Distinguish functions for different words by superscripts.
For the first claim, construct a word $U$ from $V$ by replacing  $v\neq a$ by $\beta^{-1}(v)$ and $a$ by $W$, relabelled by $\alpha$. Next, define functions $\pi_u:[p_u]\to U$ as follows. For $u\notin \Im\alpha$, $\pi_u(i)$ is $\pi^V_{\beta u}(i) \in W$, identified with a letter of $U$. If $w\in\com{\#W}$, then $\pi_{\alpha w}(i)$ is the letter of $U$ that comes from $\pi^W_w(i)\in W$ when $\pi^V_a(\xi_w(i))$ is replaced with $W$ and relabelled. Finally, $X$ is given by deleting from $U$ any letter that is not in the image of some $\pi_u$.

Conversely, for any $u\notin\Im\alpha$, $\zeta_{\beta u} = \eta_u$, and $\zeta_a$ is determined by
\[
\Im \zeta_a = \bigcup_{w=1}^{\#W} \Im\eta_{\alpha w} .
\]
Since $\zeta_a$ is injective, it determines $\xi_w$ for any $w\in\com{\# W}$ by
\[
\zeta_a\circ\xi_w = \eta_{\alpha w} .
\]
For any $u\notin\Im\alpha$, $\pi^V_{\beta u}(i)$ is $\pi^X_u(i)$ (after relabelling by $\beta$ and eliminating consecutive repetitions). For any $w\in\com{\#W}$,  $\pi^V_a[\xi_w(i)]$ is similarly defined by $\pi^X_{\alpha w}(i)$, and $\pi^W_w(i)$ is $\pi^X_{\alpha w}(i)$ (after relabelling by $\alpha^{-1}$).
\end{proof}
\begin{remark}
In terms of the description in Section~\ref{Computing extensions}, the extension $X$ is determined by a function $\kappa$, where $\kappa(\underline a) = w$ if there exists $0\leq i<p_{\alpha w}$ such that
\[
\pi^V_a(\xi_w(i))=\underline a \neq \pi^V_a(\xi_w(i+1)) .
\]
\end{remark}
\begin{lemma}
\label{Quilt coloring}
For $Q\in\Qcat$, if $\bzeta\in\Ind(\W_Q;\bp)$ and $I\in\Ind(\Tree_Q;\bq)$, then 
\[
(\bzeta,I)\in\CQ .
\]
\end{lemma}
\begin{proof}
Consider any $(a,b)\in\E_Q$. By definition, $a<_Qb$. By Axiom~\ref{Quilt def}\eqref{Quilt1}, every occurrence of $a$ in $W$ precedes every occurrence of $b$.
Let $\underline a=\pi_a(p_a) \in\W_Q$ be the last occurrence of $a$ and $\underline b = \pi_b(0) \in\W_Q$ the first occurrence of $b$. The last axiom of Definition~\ref{Word coloring} implies  (indirectly) that 
\[
\max \zeta_a = \zeta_a(p_1) \leq \zeta_b(0) = \min \zeta_b . \qedhere
\]
\end{proof}
In light of this, it is appropriate to call $(\bzeta,I)$ a \emph{coloring} of $Q$.
\begin{definition}
\label{Expansion}
Given $W\in\Wordcat$, $\underline a\in\W$, and $\bp\in\N^{\#W}$, define 
\[
E_a\bp := (p_1,\dots,p_{a-1},p_a+1,p_{a+1},\dots p_{\#W}) ,
\]
and $E_{W,\underline a}:\Ind(W;\bp)\to\Ind(W; E_a\bp)$  such that for any $\bzeta\in\Ind(W;\bp)$, $j\in\Im (E_{W,\underline a}\bzeta)_b$ if
\begin{itemize}
\item
$j = \zeta_b(i)$, where $\pi_b(i)$ precedes or equals $\underline a$ in $W$, or
\item
$j = 1+\zeta_b(i)$, where $\pi_b(i)$ follows or equals $\underline a$ in $W$.
\end{itemize}
\end{definition}
The idea is to expand by $1$ the part of $\bzeta$ corresponding to $\underline a$. In the above construction of $\bzeta$ from $\omega:W\to\N$, $E_{W,\underline a}$ corresponds to replacing $\omega_{\underline a}$ with $\omega_{\underline a}+1$.

\subsection{Signs}
$\bzeta\in\Ind(\W_Q;\bp)$ and $I\in\Ind(\Tree_Q;\bq)$  determine a sign, given as the sign of a shuffle. I will describe this shuffle as a permutation relating two sequences, and it will be convenient to describe these sequences as words, even though they never contain repetitions.

 I will use an alphabet consisting of what I will refer to as \emph{vertical} and \emph{horizontal} letters for each vertex of $Q$. Specifically, for each $a\in\com{\#Q}$, we have:
\[
0_{\mathrm v a},\dots,(p_a-1)_{\mathrm v a}
\]
and
\[
1_{\mathrm h a},\dots,(q_a-1)_{\mathrm h a} .
\]
\begin{definition}
\label{Initial word}
Define the word
\beq
\label{a word}
\Wone_a(p_a,q_a):= 0_{\mathrm v a}\dots(p_a-1)_{\mathrm v a}
1_{\mathrm h a}\dots(q_a-1)_{\mathrm h a} 
\eeq
for each $a\in\com{\#Q}$, and concatenate these in labelled order:
\[
\Wone(\bp,\bq) := \Wone_1(p_1,q_1)\smile\Wone_1(p_2,q_2)\smile\dots\smile\Wone_{\#Q}(p_{\#Q},q_{\#Q}) .
\]
\end{definition}
\begin{definition}
\label{Shuffled word}
For $Q\in\Qcat$, $\bzeta\in\Ind(\W_Q;\bp)$, and $I\in\Ind(\Tree_Q;\bq)$, the word 
\[
\Wtwo_Q(\bzeta,I) := \Wtwo^0_Q \smile \Wtwo^{\mathrm v}_Q(\bzeta) \smile \Wtwo^{\mathrm h}(I) 
\]
is the concatenation of the following 3 words:
\begin{enumerate}
\item
$\Wtwo^0_Q$ consists of $0_{\mathrm v a}$ for each interposed vertex $a\in\com{\#Q}$, arranged in \textbf{reverse} $\downarrow$ order.
\item
$\Wtwo^{\mathrm v}_Q(\bzeta)$ has the letter $i_{\mathrm va}$ in position $\zeta_a(x)$ for $0\leq i\leq p_a-1$ if $a$ is not interposed and $1\leq i\leq p_a-1$ if $a$ is interposed. (Positions start from $0$ here.) 
\item
$\Wtwo^{\mathrm h}(I)$ is given by going through the corner word $\Place_Q$ in order:
\begin{itemize}
\item
If $\Place_Q = \dots \underline a b\dots$ and $\underline a$ is a caesura, then replace $\underline a$ with the next unused horizontal letters for $a$ up to $(I_{(a,b)}-1)_{\mathrm h a}$.
\item
If $\underline a$ is the last occurrence, then replace it with the remaining unused horizontal letters (up to $(q_a-1)_{\mathrm ha}$).
\end{itemize}
\end{enumerate}
\end{definition}
\begin{remark}
The  horizontal word, $\Wtwo^{\mathrm h}(I)$, can be described in the same manner as the vertical word, $\Wtwo^{\mathrm v}_Q(\bzeta)$,  by using the identification of $\NSOp$ with a suboperad of $\MultiSemiD$, although there are no $0$'s.
\end{remark}

\begin{definition}
\label{Sign shuffle}
Suppose that $p_a\geq1$ for all interposed $a\in\com{\#Q}$ and $q_a\geq1$ for all $a\in\com{\#Q}$.
\label{CQ sign}
Define  $\sgn_Q(\bzeta,I)$  as the sign of the  shuffle transforming $\Wone(\bp,\bq)$ to $\Wtwo_Q(\bzeta,I)$.
\end{definition}
This permutation is a shuffle in the sense that it does not change the order within the set of  letters associated to each vertex.

\begin{remark}
If $p_a=0$ and $a$ is interposed, then there is no $0_{\mathrm va}$ to put at the beginning of $\Wtwo_Q(\bzeta,I)$, hence the first restriction. The simplest case where this can happen even though a coloring does exist is $a=2$ in 
\[
\Qfourb1234 .
\]

 In general, there are $q_a-1$ horizontal letters for $a$, so $q_a=0$ would call for $-1$ of these letters, hence the second restriction.
 \end{remark}

These restrictions are unacceptable, so what can be done about them? Note that the semisimplicial maps can be parametrized (in various ways) by sequences of integers. The number of crossings for the permutation can be expressed in terms of these integers, along with the components of $I$ and $\bq$. This is a quadratic function with integer coefficients, therefore the sign of the permutation only depends upon the integers modulo $2$. The sign is defined in general by extrapolating modulo $2$.
\begin{definition}
If it is not already defined, then $\sgn_Q(\bzeta,I)$ is defined by  replacing $q_a=0$ with $2$ or by 
\[
\sgn_Q (\bzeta,I) := \sgn_Q (E_{Q,a}E_{Q,a}\bzeta,I) .
\]
\end{definition}

\subsubsection{Examples}
\begin{example}
Consider a coloring \label{Term} $(\bzeta,I)$ of
\[
\Qtwo12 = 
\left(
12\;,
\raisebox{-.5\height}{
\begin{tikzpicture}[scale=.7,node font=\footnotesize]
\node{$1$}
child{node{$2$}};
\end{tikzpicture}} 
\right) .
\]
 The $I$ is given by the underlying tree and one number, call this  $i=I_{(1,2)}$. Since there are no repetitions in the word $12$, we must have $\Im \zeta_1 = \{0,\dots,p_1\}$ and $\Im\zeta_2 = \{p_1,\dots,p_1+p_2\}$.

The initial word is  
\begin{align*}
\Wone(\bp,\bq) &= \Wone_1(p_1,q_1)\smile\Wone_2(p_2,q_2) \\
&=0_{\mathrm v1}\dots(p_1-1)_{\mathrm v1} 1_{\mathrm h1}\dots (q_1)_{\mathrm h1}0_{\mathrm v2}\dots(p_2-1)_{\mathrm v2} 1_{\mathrm h2}\dots (q_2)_{\mathrm h2} .
\end{align*}

Since there are no interposed vertices, $\Wtwo_Q(\bzeta,I) = \Wtwo^{\mathrm v}_Q(\bzeta)\smile\Wtwo^{\mathrm h}(I)$ where
\[
\Wtwo^{\mathrm v}_Q(\bzeta) = 0_{\mathrm v1}\dots(p_1-1)_{\mathrm v1} 0_{\mathrm v2}\dots(p_2-1)_{\mathrm v2} 
\]
and then
\[
\Wtwo^{\mathrm h}(I) = 1_{\mathrm h1}\dots (i-1)_{\mathrm h1} 1_{\mathrm h2}\dots (q_2-1)_{\mathrm h2} i_{\mathrm h1} \dots (q_1-1)_{\mathrm h1} .
\]

Now look at the numbers of consecutive letters of each type, and how they move around:
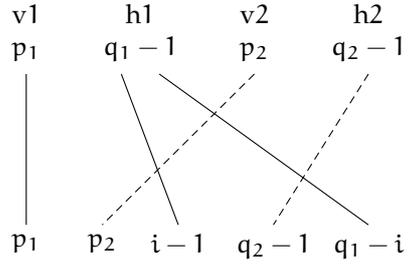
\[
\begin{tikzpicture}[node font=\footnotesize]
\draw (0,0) node[anchor=south]{$\begin{matrix}\mathrm v1\\ p_1\end{matrix}$}--(0,-2) node[anchor=north]{$p_1$};
\draw (1.25,0)  --(2,-2) node[anchor=north]{$i-1$};
\draw (1.75,0) --(4.5,-2)node[anchor=north]{$q_1-i$};
\draw (1.5,0) node[anchor=south]{$\begin{matrix}\mathrm h1\\ q_1-1\end{matrix}$};
\draw (3,0) node[anchor=south]{$\begin{matrix}\mathrm v2\\ p_2\end{matrix}$}[densely dashed]--(1,-2)node[anchor=north]{$p_2$};
\draw (4.5,0) node[anchor=south]{$\begin{matrix}\mathrm h2\\ q_2-1\end{matrix}$} [densely dashed]--(3.25,-2)node[anchor=north]{$q_2-1$};
\end{tikzpicture} .
\]
There are
\[
p_2(q_1-1) + (q_1-i)(q_2-1) 
\]
crossings, and the sign is $-1$ to this power.
\end{example}
\begin{example}
For a coloring $(\bzeta,I)$ of
\[
\Qthreea321 =
\left(3121\;,
\raisebox{-.5\height}{
\begin{tikzpicture}[scale=.7,node font=\footnotesize]
\node{$3$}
child{node{$2$}}
child{node{$1$}};
\end{tikzpicture}}
\right) 
\]
let $i = I_{(3,2)}$ and $j = I_{(3,1)}$. Of course, $1\leq i< j \leq p_3$.

Since $3$ only occurs first in $3121$, we must have $\Im \zeta_3 = \{0,\dots,p_3\}$, so $\min \zeta_1 = p_3$. This implies that $\min \zeta_2\geq p_3$. Define $k = \min\zeta_2 +1- p_3$. Since $2$ occurs once in $3121$, $\Im \zeta_2 = \{p_3+k-1,\dots,p_2+p_3+k-1\}$. This leaves
\[
\Im\zeta_1 = \{p_3,\dots,p_3+k-1\}\cup\{p_2+p_3+k-1,\dots,p_1+p_2+p_3-1\} .
\]
Note that $1\leq k \leq p_1$.

The shuffled word, $\Wtwo_Q(\bzeta,I)$, consists of $\Wtwo^0_Q=0_{\mathrm v2}$, then
\[
\Wtwo^{\mathrm v}_Q(\bzeta)= 0_{\mathrm v3}\dots(p_3-1)_{\mathrm v3} 0_{\mathrm v1}\dots (k-1)_{\mathrm v1} 1_{\mathrm v2}\dots (p_2-1)_{\mathrm v2} k_{\mathrm v1}\dots (p_1-1)_{\mathrm v1} 
\]
and finally,
\[
\Wtwo^{\mathrm h}(I)= 1_{\mathrm h3}\dots (i-1)_{\mathrm h3} 1_{\mathrm h2}\dots (q_2-1)_{\mathrm h2} i_{\mathrm h3}\dots (j-1)_{\mathrm h3} 1_{\mathrm h1}\dots (q_1-1)_{\mathrm h1} j_{\mathrm h3}\dots (q_3-1)_{\mathrm h3} .
\]
Again, look at the numbers of letters of each type:
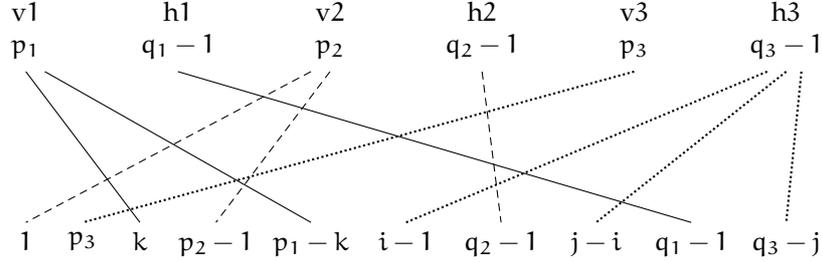
\[
\begin{tikzpicture}[node font=\footnotesize]
\draw (3.75,0)[densely dashed]--(0,-2) node[anchor=north]{$1$};
\draw (8,0) node[anchor=south]{$\begin{matrix}\mathrm v3\\ p_3\end{matrix}$}[thick,densely dotted]--(.75,-2) node[anchor=north]{$p_3$};
\draw (0,0) node[anchor=south]{$\begin{matrix}\mathrm v1\\ p_1\end{matrix}$}--(1.5,-2) node[anchor=north]{$k$};
\draw (4,0) node[anchor=south]{$\begin{matrix}\mathrm v2\\ p_2\end{matrix}$}[densely dashed]--(2.5,-2) node[anchor=north]{$p_2-1$};
\draw (.25,0)-- (3.75,-2) node[anchor=north]{$p_1-k$};
\draw (9.75,0)[thick,densely dotted]--(5,-2) node[anchor=north]{$i-1$};
\draw (6,0) node[anchor=south]{$\begin{matrix}\mathrm h2\\ q_2-1\end{matrix}$} [densely dashed]--(6.25,-2) node[anchor=north]{$q_2-1$};
\draw (10,0) node[anchor=south]{$\begin{matrix}\mathrm h3\\ q_3-1\end{matrix}$}[thick,densely dotted]--(7.5,-2) node[anchor=north]{$j-i$};
\draw (2,0) node[anchor=south]{$\begin{matrix}\mathrm h1\\ q_1-1\end{matrix}$}--(8.75,-2) node[anchor=north]{$q_1-1$};
\draw (10.2,0)[thick,densely dotted]--(10,-2) node[anchor=north]{$q_3-j$};
\end{tikzpicture}
\]
The number of crossings is
\[
(q_1-1)(p_2+p_3+q_2+j-2) + (q_2-1)(p_3+i-1) + p_1(p_3+1) + (p_2-1)(p_1-k)
\]
and the sign is $-1$ to this power.
\end{example}

\subsection{The representation}
\begin{definition}
\label{Rep0 def}
Let $n\in\N$, $\bp,\bq\in\N^n$, $Q\in\Qcat(n)$, and $f_a\in C^{p_a,q_a}(\A,\A)$ for $a=1,\dots,n$. Define
\[
\widehat Q(f_1,\dots,f_n) := \sum_{\bzeta\in\Ind(\W_Q;\bp)} \sum_{I\in\Ind(\Tree_Q;\bq)} \sgn_Q(\bzeta,I)\, \rep(\bzeta,I;f_1,\dots,f_n) .
\]
Also denote this as
\[
\Rep_0(Q):=\widehat Q .
\]
\end{definition}

\begin{remark}
$\abs{f_a} = p_a+q_a-1=\Abs{\Wone_a(p_a,q_q)}$ is the length of the word in Definition~\ref{Sign shuffle}, whereas $\abs{\widehat Q(f_1,\dots,f_n)} = p'+q'-1 = \Abs{\Wtwo^{\mathrm v}_Q(\bzeta) \smile \Wtwo^{\mathrm h}(I)}$ is the length of $\Wtwo_Q(\bzeta,I)$ after $\Wtwo_Q^0$.
\end{remark}
\begin{remark}
$\widehat Q$ decreases shifted bidegree by $(\deg Q,0)$, i.e.,
\[
\bigl\|\widehat Q(f_1,\dots,f_{\#Q})\bigr\| = \norm{f_1} + \dots + \norm{f_{\#Q}}-(\deg Q,0) ,
\]
so
\[
\bigl|\widehat Q(f_1,\dots,f_{\#Q})\bigr| = \abs{f_1} + \dots + \abs{f_{\#Q}}-\deg Q .
\]
In other words, $\widehat Q$ has homological bidegree $(\deg Q,0)$.
\end{remark}

Let $\Endop$ denote the endomorphism operad of a differential graded or bigraded vector space.
\begin{lemma}
\label{S-equivariance}
The linear map 
\begin{align*}
\Rep_0 : &\Qop \to\Endop[\s_2 C^{\bullet,\bullet}(\A,\A)]\\ 
&Q\mapsto \widehat Q
\end{align*}
 is permutation equivariant.
\end{lemma}
\begin{proof}
First, note that $\Rep_0$ does actually define a linear map thanks to Lemma~\ref{Quilt coloring}.

In the initial word $\Wone(\bp,\bq)$ above, there are $\abs{f_a}$ letters for the vertex $a$, and these are taken in the labelled order of the vertices. The shuffled word $\Wtwo_Q(\bzeta,I)$ is independent of the labelled order. Therefore, the effect of permuting the vertices is to multiply by the correct sign for permuting the arguments $f_1,\dots,f_{\#Q}$.
\end{proof}

\begin{lemma}
\label{Rep composition}
For $P,Q\in \Qcat$ and any $a\in\com{\#P}$,
\[
\widehat{P\circ_a Q} = \widehat{P}\circ_a\widehat{Q}  .
\]
(Using the Koszul sign convention.) More explicitly,
\begin{multline*}
(\widehat{P\circ_aQ})(f_1,\dots,f_{\#P+\#Q-1}) = \\(-1)^{(\abs{f_1}+\dots\abs{f_{a-1}})\deg Q} \widehat{P}(f_1,\dots,f_{a-1},\widehat{Q}[f_a,\dots,f_{a+\#Q-1}],f_{a+\#Q},\dots,f_{\#P+\#Q-1}) .
\end{multline*}
\end{lemma}
\begin{remark}
The logic of the sign convention is that in this equation $Q$ has moved past $f_1$, $\dots$, $f_{a-1}$. There is a factor of $(-1)^{\abs{f_1}\deg Q}$ for moving past $f_1$, and so on.
\end{remark}
\begin{proof}
By Lemma~\ref{S-equivariance}, it is sufficient to consider the case that $a=1$, which simplifies the signs. Let $m:=\#P$ and $n:=\#Q$.  We need to check whether
\beq
\label{composition1}
(\widehat{P\circ_1Q})(f_1,\dots,f_{n+m-1}) \stackrel?= 
\widehat{P}(\widehat{Q}[f_1,\dots,f_{n}],f_{n+1},\dots,f_{n+m-1}) .
\eeq

An arbitrary term on the right hand side is given by some colorings $(\bzeta,I)$ of $P$ and $(\bxi,J)$ of $Q$. By Lemma~\ref{rep lemma}, this term is proportional to 
\[
\rep(\bzeta\circ_1\bxi,I\circ_1 J;f_1,\dots,f_{n+m-1}) .
\]
Let $T$ be the underlying tree of $I\circ_1J$, so that $I\circ_1J$ is a coloring of $T$ and $T\in\Ext(\Tree_P,\Tree_Q,1)$. By Lemma \ref{Coloring completeness}, there exists a word $W\in\Ext(\W_P,\W_Q,1)$ of which $\bzeta\circ_1\bxi$ is a coloring. By Lemma~\ref{Quilt extension lemma}, $(W,T)\in\Qcat$.  So, there is a proportional term on the left hand side.

An arbitrary term on the left hand side is given by some $R\in\Ext(P,Q,1)$ and a coloring $(\etab,H)$ of $R$. By Lemma~\ref{Coloring completeness}, there exist colorings $\bzeta$ of $\W_P$ and $\bxi$ of $\W_Q$ such that $\etab=\bzeta\circ_1\bxi$. Similarly, there exist colorings $I$ of $\Tree_P$ and $J$ of $\Tree_Q$ such that $H=I\circ_1J$. So, there is a proportional term on the right hand side.

So, there are terms proportional to $\rep(\bzeta\circ_1\bxi,I\circ_1J;f_1,\dots)$ on both sides of eq.~\eqref{composition1}. On the left, the coefficient is 
\[
\sgn_{P,Q,1}(R)\sgn_R(\bzeta\circ_1\bxi,I\circ_1J)
\]
and on the right
\[
\sgn_P(\bzeta,I)\sgn_Q(\bxi,J) .
\]

It is sufficient to prove equality in the generic case, when the signs are directly defined by shuffles.

First note that for the $\Comcat$-extension $\com{n} \stackrel\alpha\into \com{n+m-1} \stackrel\beta\onto \com{m}$ at $1$, $\alpha a = a$ and $\beta b = \max\{b-n+1,1\}$. Referring to Definition~\ref{Ext sign},  if $a$ is interposed in $Q$, then it is simply identified with the vertex $a$ of $R$. If $b\neq1$ is interposed in $P$, then it is identified with $b+n-1$ in $R$; if $1$ is interposed in $P$, then it is identified with the vertex of $R$ whose label is the same as the root of $Q$.

Next, consider what happens to certain letters and words if we rename 
\[
\Wtwo^{\mathrm v}_Q(\bxi)\smile \Wtwo^{\mathrm h}(J)\smile \Wone_{n+1}(p_{n+1},q_{n+1})\smile\dots\smile\Wone_{n+m-1}(p_{n+m-1},q_{n+m-1}) 
\]
so that it becomes
\[
\Wone_1(p',q') \smile \Wone_2(p_{n+1},q_{n+1})\smile\dots\smile\Wone_n(p_{n+m-1},q_{n+m-1}) ,
\]
where
\begin{align*}
p' &= \sum_{a=1}^n p_a - \deg Q, &
q' &= 1 + \sum_{a=1}^n(q_a-1) .
\end{align*}
(This corresponds to using the output of $\widehat Q$ as the first input of $\widehat P$.) If $b\neq1$, then $0_{\mathrm v (b+n-1)}$ is renamed  $0_{\mathrm v b}$. 
If $a$ is the root of $Q$, then $0_{\mathrm v a}$ is the first letter of $\Wtwo^{\mathrm v}_Q(\bxi)$ and so is renamed $0_{\mathrm v 1}$.
This shows that if $a$ is an interposed vertex of $P$, and $b$ is the vertex of $R$ identified with $a$ 
in the sense of Definition~\ref{Ext sign} then $0_{\mathrm v b}$ becomes $0_{\mathrm v a}$ in this renaming.

Next, consider the word $\Wtwo^{\mathrm v}_R(\bzeta\circ_1\bxi)$ and what happens under this renaming. For $1\leq a \leq n$, the letter $i_{\mathrm v a}$ is in position $\zeta_1(\xi_a[i])$; this gets renamed to $\xi_a(i)_{\mathrm v1}$. For $2\leq b \leq m$, the letter $i_{\mathrm v (b+n-1)}$ is in position $\zeta_b(i)$ and gets renamed to $i_{\mathrm v b}$. In this way, $\Wtwo^{\mathrm v}_R(\bzeta\circ_1\bxi)$ is renamed to $\Wtwo^{\mathrm v}_P(\bzeta)$. Similarly, $\Wtwo^{\mathrm h}(I\circ_1 J)$ is renamed to $\Wtwo^{\mathrm h}(I)$.

Consider the following sequence of words and shuffles.
\begin{itemize}
\item
Begin with $\Wone(\bp,\bq)$.
\item
Shuffle this to 
\[
\Wtwo_Q(\bxi,J)\smile \Wone_{n+1}(p_{n+1},q_{n+1})\smile\dots\smile\Wone_{n+m-1}(p_{n+m-1},q_{n+m-1}) .
\]
(Note that this begins with $\Wtwo^0_Q$.) The sign of this shuffle is $\sgn_Q(\bxi,J)$.
\item
Shuffle this to $\Wtwo^0_Q$, then $0_{\mathrm v a}$ for each of the interposed vertices of $P$ in reverse $\downarrow$ order identified with vertices of $R$, then $\Wtwo^{\mathrm v}_R(\bzeta\circ_1\bxi)$ and $\Wtwo^{\mathrm h}(I\circ_1 J)$.
The sign is $\sgn_P(\bzeta,I)$, because of the previous 2 paragraphs.
\item
Shuffle this to $\Wtwo_R(\bzeta\circ_1\bxi,I\circ_1J)$. This shuffles the first 2 subwords of the previous sequence to $\Wtwo^0_R$. The sign is by definition $\sgn_{P,Q,1}(R)$ (except that everything is in reverse order).
\item
Shuffle this back to the first sequence, $\Wone(\bp,\bq)$. The sign is $\sgn_R(\bzeta\circ_1\bxi,I\circ_1J)$, by definition.
\end{itemize}
The composition of these shuffles is the identity, so the product of their signs is $1$.  Therefore, the terms proportional to $\rep(\bzeta\circ_1\bxi,I\circ_1J;f_1,\dots)$ in eq.~\eqref{composition1} are equal and the equation is true.
\end{proof}

\begin{lemma}
\label{Rep differential}
For any $Q\in\Qcat$,
\[
\widehat{\partial Q} = \dS\circ\widehat Q - (-1)^{\deg Q} \sum_{a=1}^{\#Q} \widehat Q\circ_a \dS .
\]
More explicitly,
\begin{multline}
\label{rep differential}
\widehat{\partial Q}(f_1,\dots,f_{\#Q}) = \dS\widehat Q(f_1,\dots,f_{\#Q})\\ - (-1)^{\deg Q} \widehat Q(\dS f_1,f_2,\dots) - (-1)^{\deg Q + \abs{f_1}}\widehat Q(f_1,\dS f_2,\dots) - \dots \\ - (-1)^{\deg Q + \abs{f_1} + \dots + \abs{f_{\#Q-1}}}\widehat Q(f_1,\dots,f_{\#Q-1},\dS f_{\#Q}) .
\end{multline}
\end{lemma}
\begin{proof}
Let $n := \#Q$, $f_a\in C^{p_a,q_a}(\A,\A)$, and $p' = \sum_{a=1}^n p_a - \deg W$. Inserting Definitions \ref{word boundary},  \ref{dS def}, \ref{Rep0 def} of $\partial$, $\dS$, and $\Rep_0$ --- and using Lemma~\ref{rep lemma} --- gives 
\[
\widehat{\partial Q}(f_1,\dots,f_n) = \sum_{I\in\Ind(\Tree_Q;\bq)} \sum_{\underline a\in\W_Q} \sum_{\etab\in\Ind(\partial_{\underline a}\W_Q;\bp)} \sgn_Q(\underline a) \sgn_{\partial_{\underline a}Q}(\etab,I)\, \rep(\etab,I;f_1,\dots,f_n) 
\]
(where $\Ind(\partial_{\underline a}\W_Q;\bp)=\emptyset$ if $a$ is not repeated),
\[
\dS\widehat Q(f_1,\dots,f_n) = \sum_{I\in\Ind(\Tree_Q;\bq)} \sum_{j=0}^{p'+1} \sum_{\bzeta\in\Ind(\W_Q;\bp)}(-1)^j\sgn_Q(\bzeta,I)\,\rep(\varepsilon_j\circ\bzeta,I;f_1,\dots,f_n) ,
\]
and
\[
\widehat{Q}(\dots,\dS f_a,\dots) = \sum_{I\in\Ind(\Tree_Q;\bq)} \sum_{\bxi\in\Ind(\W_Q;E_a\bp)}\sum_{k=0}^{p_a+1} (-1)^k \sgn_Q(\bxi,I)\,\rep(\bxi\circ_a\varepsilon_k,I;f_1,\dots,f_n) .
\]
Note that the sum over $I\in\Ind(\Tree_Q;\bq)$ is the same in every term and that everything acts on $f_1,\dots,f_n$. In this way, the equation to be checked can be written as
\begin{multline}
\label{expanded differential}
\sum_{\underline a\in\W_Q} \sum_{\etab\in\Ind(\partial_{\underline a}\W_Q;\bp)} \sgn_Q(\underline a) \sgn_{\partial_{\underline a}Q}(\etab,I)\, \rep[\etab,I] 
\stackrel?= \\
 \sum_{j=0}^{p'+1} \sum_{\bzeta\in\Ind(\W_Q;\bp)}(-1)^j\sgn_Q(\bzeta,I)\,\rep[\varepsilon_j\circ\bzeta,I] \\
- (-1)^{\deg Q}\sum_{a=1}^n \sum_{k=0}^{p_a}\sum_{\bxi\in\Ind(\W_Q;E_a\bp)}(-1)^{\abs{f_1}+\dots+\abs{f_{a-1}}}(-1)^k\sgn_Q(\bxi,I)\,\rep[\bxi\circ_a\varepsilon_k,I]
\end{multline}
We will see that every term here appears exactly twice, with consistent signs.

1. First, consider an arbitrary term in the middle sum of eq.~\eqref{expanded differential} given by some $j$ and $\bzeta$. 
There exists a unique $\underline a \in\W_Q$ and $k\in [p_a+1]$ such that $j = (E_{Q,\underline a}\bzeta)_a(k)$ and $\underline a = \pi_a(k)$ (where $E_{Q,\underline a}$ is defined in Def.~\ref{Expansion}). With this notation,
\[
\varepsilon_j \circ \bzeta = (E_{Q,\underline a}\bzeta)\circ_a \varepsilon_k ,
\]
hence there is a proportional term in the last sum with $\bxi=E_{Q,\underline a}\bzeta$.
Conversely, for $\underline a \in\W_Q$ and $k\in [p_a+1]$ such that $\underline a = \pi_a(k)$, this holds if $j := \bxi_a(k)$ is not in the image of any $\bxi_b$ for $b\neq a$.

Note that $j = \zeta_a(k)$, unless $\underline a$ is the last term of $\W_Q$ and $k=p_a+1$, in which case $\zeta_a(k)$ is undefined.

To see that those two terms cancel in eq.~\eqref{expanded differential}, 
we need to check whether
\[
(-1)^j\sgn_Q(\bzeta,I) \stackrel?= (-1)^{\deg Q} (-1)^{\abs{f_1}+\dots+\abs{f_{a-1}}}(-1)^k\sgn_Q(E_{Q,\underline a}\bzeta,I) .
\]
It is sufficient to check this in the generic case that signs are defined directly by shuffles (Def.~\ref{Sign shuffle}).
Using an extra letter, $\Diamond$, consider the following sequence of words and shuffles.
\begin{itemize}
\item
Begin with  $\Diamond\smile\Wone(\bp,\bq)$.
\item
Move $\Diamond$ to just before the word $\Wone_a(p_a,q_a)$. The sign of this is $(-1)^{\abs{f_1}+\dots+\abs{f_{a-1}}}$.
\item
Move $\Diamond$ to just before $k_{\mathrm v a}$. The sign of this is $(-1)^k$. 
\begin{itemize}
\item
Rename the vertical letters for $a$, so that $\Diamond\mapsto k_{\mathrm va}$, and $i_{\mathrm va}\mapsto (i+1)_{\mathrm va}$ for $i\geq k$. This gives $\Wone(E_a\bp,\bq)$.
\end{itemize}
\item
Shuffle this to $\Wtwo_Q(E_{Q,\underline a}\bzeta,I)$. The sign is $\sgn_Q(E_{Q,\underline a}\bzeta,I)$.
\begin{itemize}
\item
Undo the above renaming. This gives $\Wtwo_Q(\bzeta,I)$, but with $\Diamond$ inserted at position $\xi_a(k)=j$ in $\Wtwo^{\mathrm v}_Q(\bzeta)$
\end{itemize}
\item
Shuffle this to $\Wtwo^0_Q\smile \Diamond\smile \Wtwo^{\mathrm v}_Q(\bzeta)\smile \Wtwo^{\mathrm h}(I)$. The sign of this is $(-1)^j$.
\item
Move $\Diamond$  to the beginning, past $\Wtwo^0_Q$ (which has length $\deg Q$). This gives $\Diamond\smile\Wtwo_Q(\bzeta,I)$. The sign of this is $(-1)^{\deg Q}$.
\item
Shuffle this to where we started: $\Diamond\smile\Wone(\bp,\bq)$. The sign is (by definition) $\sgn_Q(\bzeta,I)$.
\end{itemize}
The composition of these shuffles is the identity, therefore the product of the signs is $1$ and the terms cancel in eq.~\eqref{expanded differential}.

2. Now let $\bzeta\in\Ind(\W_Q;\bp)$, suppose that $\W_Q = \dots \underline a\,\underline b\dots$, and consider $j\in[p_a+1]$ and $k\in[p_b]$ such that $(E_{Q,\underline a}\bzeta)_a(j) = (E_{Q,\underline a}\bzeta)_b(k)$ (which equals $1+\zeta_b(k)$). In this case
\[
(E_{Q,\underline a}\bzeta)\circ_a \varepsilon_j = (E_{Q,\underline b} \bzeta)\circ_b \varepsilon_k
\]
so this gives two terms of the last sum in eq.~\eqref{expanded differential}. 
To show that they cancel, we need to consider two slightly different cases. 

2a. The first case is if $\underline a$ is a caesura, i.e., for some $b$,
\[
\W_Q = \dots \underline a\,\underline b\dots a\dots \ ,
\]
hence $b$ is interposed.
This implies that there is no earlier $b$, so $k=0$. With that in mind, we need to check in this case whether 
\[
(-1)^{\abs{f_1}+\dots+\abs{f_{a-1}}}(-1)^j\sgn_Q(E_{Q,\underline a}\bzeta,I)
\stackrel?= - (-1)^{\abs{f_1}+\dots+\abs{f_{b-1}}}\sgn_Q(E_{Q,\underline b}\bzeta,I) .
\]
\begin{itemize}
\item
Begin with   $\Wtwo_Q(\bzeta,I)$,  but with an extra letter $\Diamond$ inserted between the vertical letters for $a$ and $b$, i.e.,
\[
\dots (j-1)_{\mathrm va} \,\Diamond\, 1_{\mathrm vb} \dots \ .
\]
\begin{itemize}
\item
Rename the vertical $a$ letters so that $\Diamond\mapsto j_{\mathrm va}$ and $i_{\mathrm va}\mapsto (i+1)_{\mathrm va}$ for $i\geq j$. This gives $\Wtwo_Q(E_{Q,\underline a}\bzeta,I)$.
\end{itemize}
\item
Shuffle this to $\Wone(E_a\bp,\bq)$. The sign is $\sgn_Q(E_{Q,\underline a}\bzeta,I)$.
\begin{itemize}
\item
Undo the renaming. This gives $\Wone(\bp,\bq)$ with $\Diamond$ inserted just after $(j-1)_{\mathrm va}$.
\end{itemize}
\item
Move $\Diamond$ to just before $\Wone_a(p_a,q_a)$. The sign is $(-1)^j$. 
\item
Move $\Diamond$ to the beginning, giving $\Diamond\smile\Wone(\bp,\bq)$. The sign of this is $(-1)^{\abs{f_1}+\dots+\abs{f_{a-1}}}$. 
\item
Move $\Diamond$ to just before $\Wone_b(p_b,q_b)$. The sign of this is $(-1)^{\abs{f_1}+\dots+\abs{f_{b-1}}}$.
\item
Move $\Diamond$ past $0_{\mathrm v b}$. The sign of this is $-1$. 
\begin{itemize}
\item
Rename the vertical $b$ letters so that $\Diamond\mapsto1_{\mathrm vb}$ and $i_{\mathrm vb}\mapsto(i+1)_{\mathrm vb}$ for $i\geq1$. This gives $\Wone(E_b\bp,\bq)$.
\end{itemize}
\item
Shuffle to $\Wtwo_Q(E_{Q,\underline b}\bzeta,I)$. The sign is $\sgn_Q(E_{Q,\underline b}\bzeta,I)$.
\begin{itemize}
\item
 Undo the renaming. This brings us back to where we started.
 \end{itemize}
\end{itemize}
This shows that in this case, the two terms cancel in  eq.~\eqref{expanded differential}.

2b. The other case is if $\underline a$ is the last occurrence of $a$. This implies that $j=p_a+1$. We need to check whether
\[
(-1)^{\abs{f_1}+\dots+\abs{f_{a-1}}}(-1)^{p_a+1}\sgn_Q(E_{Q,\underline a}\bzeta,I)
\stackrel?= - (-1)^{\abs{f_1}+\dots+\abs{f_{b-1}}}(-1)^k\sgn_Q(E_{Q,\underline b}\bzeta,I) .
\]
\begin{itemize}
\item
Begin with  $\Wtwo_Q(\bzeta,I)$,  with an extra letter $\Diamond$ inserted between the vertical letters for $a$ and $b$, i.e.,
\[
\dots (p_a-1)_{\mathrm va} \,\Diamond\, k_{\mathrm vb} \dots \ .
\]
\begin{itemize}
\item
Rename $\Diamond\mapsto (p_a)_{\mathrm va}$, giving $\Wtwo_Q(E_{Q,\underline a}\bzeta,I)$.
\end{itemize}
\item
Shuffle this to $\Wone(E_a\bp,\bq)$.
The sign is $\sgn_Q(E_{Q,\underline a}\bzeta,I)$.  
\begin{itemize}
\item
Undo the renaming. This gives $\Wone(\bp,\bq)$ with $\Diamond$ inserted to give 
\[
\dots (p_a-1)_{\mathrm va}\Diamond 1_{\mathrm ha}\dots \ .
\]
\end{itemize}
\item
Move $\Diamond$ to just before $\Wone_a(p_a,q_a)$. The sign is $(-1)^{p_a}$. 
\item
Move $\Diamond$ to the beginning. The sign of this is $(-1)^{\abs{f_1}+\dots+\abs{f_{a-1}}}$. 
\item
Move $\Diamond$ to just before $\Wone_b(p_b,q_b)$. The sign of this is $(-1)^{\abs{f_1}+\dots+\abs{f_{b-1}}}$.
\item
Move $\Diamond$ to just before $k_{\mathrm v b}$. The sign of this is $(-1)^k$.  
\begin{itemize}
\item
Rename the vertical $b$ letters so that $\Diamond\mapsto k_{\mathrm vb}$ and $i_{\mathrm vb}\mapsto(i+1)_{\mathrm vb}$ for $i\geq k$. This gives $\Wone(E_b\bp,\bq)$.
\end{itemize}
\item
 Shuffle to $\Wtwo_Q(E_{Q,\underline b}\bzeta,I)$. The sign is $\sgn_Q(E_{Q,\underline b}\bzeta,I)$.
\begin{itemize}
\item
Undo the previous renaming. This brings us back to where we started.
\end{itemize}
\end{itemize}
This shows us that these terms cancel again in this case.

3. Now consider an arbitrary nonzero term on the left side of eq.~\eqref{expanded differential}, given by $\underline a\in\W_Q$ (with $a$ repeated) and $\etab\in\Ind(\partial_{\underline a}\W_Q;\bp)$.  There exists a coloring $\bxi\in\Ind(Q;E_a\bp)$ such that $\pi_a^{-1}(\underline a)$ has only one element (call it $k$) and $\bxi\circ_a\varepsilon_k = \etab$, so there is a proportional term of the last sum in eq.~\eqref{expanded differential}.

Conversely, if $a$ is repeated and $\bxi\in\Ind(Q;E_a\bp)$ such that $\pi_a^{-1}(\underline a)=\{k\}$, then $\bxi\circ_a\varepsilon_k \in  \Ind(\partial_{\underline a}\W_Q;\bp)$.

To check whether
\[
\sgn_Q(\underline a) \sgn_{\partial_{\underline a}Q}(\etab,I) \stackrel?= 
- (-1)^{\deg Q} (-1)^{\abs{f_1}+\dots+\abs{f_{a-1}}}(-1)^k\sgn_Q(\bxi,I) ,
\]
we need to consider three separate cases.

3a. First, suppose that $\underline a$ is the first occurrence of $a$, and that $a$ is interposed in $Q$. In this case, $k=0$. Let $b$ be the vertex after $\underline a$; $b$ is interposed in both $Q$ and $\partial_{\underline a}Q$, but $a$ is not interposed in $\partial_{\underline a}Q$, so $\Wtwo^0_{\partial_{\underline a}Q}$ is $\Wtwo^0_Q$ with $0_{\mathrm va}$ deleted.

Consider the following sequence:
\begin{itemize}
\item
Begin with $\Diamond\smile\Wtwo_{\partial_{\underline a}Q}(\etab,I)$.
\item
Shuffle this to $\Diamond\smile\Wone(\bp,\bq)$. The sign is $\sgn_{\partial_{\underline a}Q}(\etab,I)$.
\item
Move $\Diamond$ to just before $\Wone_a(p_a,q_a)$. The sign of this is $(-1)^{\abs{f_1}+\dots+\abs{f_{a-1}}}$.
\begin{itemize}
\item
Rename the vertical $a$ letters so that $\Diamond\mapsto0_{\mathrm va}$, and $i_{\mathrm va}\mapsto(i+1)_{\mathrm va}$. This gives $\Wone(E_a\bp,\bq)$.
\end{itemize}
\item
Shuffle this to $\Wtwo_Q(\bxi,I)$. The sign is $\sgn_Q(\bxi,I)$.
\begin{itemize}
\item
Undo the previous renaming. This gives $\Wtwo^0_{\partial_{\underline a}Q}$ with $\Diamond$ inserted after $0_{\mathrm v b}$, then $\Wtwo^{\mathrm v}_{\partial_{\underline a}Q}(\etab)$, and then $\Wtwo^{\mathrm h}(I)$.
\end{itemize}
\item
Move $\Diamond$ to just after $\Wtwo^0_{\partial_{\underline a}Q}$. The sign is $(-1)^m$, where $m$ is the number of interposed vertices preceding $a$ in $\W_Q$, because this moves $\Diamond$ past $0_{\mathrm vu}$ for each such vertex, $u$. This means that $a$ is the $(m+1)$'st interposed vertex and $b$ is the $(m+2)$'nd. Because interposed vertices are paired with caesurae, $\underline a$ is the $(m+2)$'nd caesura. Therefore $\sgn_Q\underline a = (-1)^{m+2} = (-1)^m$.
\item
Move $\Diamond$ to the beginning, bringing us back to where we started. This moves $\Diamond$ past $\Wtwo^0_{\partial_{\underline a}Q}$, so the sign is $-(-1)^{\deg Q}$.
\end{itemize}
The product of these signs is $1$, therefore the coefficients are equal in this case.

3b. Now, suppose that $\underline a$ is the last occurrence of $a$. Let $b$ be the (interposed) vertex occurring after the previous occurrence of $a$,  hence 
\[
\W_Q = \dots  a\, b\dots\underline a\dots \ .
\]
In this case, $k=p_a+1$, and $b$ is interposed in $Q$ but not in $\partial_{\underline a}Q$, so $\Wtwo^0_{\partial_{\underline a}Q}$ is $\Wtwo^0_Q$ with $0_{\mathrm vb}$ deleted.
\begin{itemize}
\item
Begin with $\Diamond\smile\Wtwo_{\partial_{\underline a}Q}(\etab,I)$.
\item
Shuffle this to $\Diamond\smile\Wone(\bp,\bq)$. 
The sign is $\sgn_{\partial_{\underline a}Q}(\etab,I)$.
\item
Move $\Diamond$ to just before $\Wone_a(p_a,q_a)$. The sign of this is $(-1)^{\abs{f_1}+\dots+\abs{f_{a-1}}}$.
\item
Move $\Diamond$ to just after $(p_a-1)_{\mathrm va}$. The sign of this is $(-1)^{p_a} = - (-1)^{p_a+1}$. 
\begin{itemize}
\item
Rename $\Diamond\mapsto(p_a)_{\mathrm va}$, giving $\Wone(E_a\bp,\bq)$.
\end{itemize}
\item
Shuffle this to $\Wtwo_Q(\bxi,I)$. The sign is $\sgn_Q(\bxi,I)$
\begin{itemize}
\item
Undo the previous renaming. This gives $\Wtwo^0_Q$, then $\Wtwo^{\mathrm v}_{\partial_{\underline a}Q}(\etab)$ with $0_{\mathrm v b}$ replaced by $\Diamond$, and then $\Wtwo^{\mathrm h}(I)$.
\end{itemize}
\item
Swap $0_{\mathrm vb}$ with $\Diamond$. This almost gives $\Wtwo_{\partial_{\underline a}Q}(\etab,I)$, but with $\Diamond$ inserted in the $\Wtwo^0_{\partial_{\underline a}Q}$ part.
The sign of this transposition is $-1$.
\item
Move $\Diamond$ to give $\Wtwo^0_{\partial_{\underline a}Q}\smile\Diamond\smile\Wtwo^{\mathrm v}_{\partial_{\underline a}Q}(\etab)\smile\Wtwo^{\mathrm h}(I)$.
The sign of this is $(-1)^m$, where $m$ is the number of interposed vertices preceding $b$ in $\W_Q$. This means that $b$ is the $(m+1)$'st interposed vertex, the previous occurrence of $a$ is the $(m+1)$'st caesura, and $\sgn_Q\underline a = (-1)^{m+2}=(-1)^m$.
\item
Move $\Diamond$ to the beginning, giving $\Diamond\smile\Wtwo_{\partial_{\underline a}Q}(\etab,I)$ again. The sign is $-(-1)^{\deg Q}$, as  before.
\end{itemize}
Again, the coefficients are equal.

3c. Finally, consider the generic case that $\underline a$ is a caesura and either $\underline a$ is  not the first occurrence or $a$ is not interposed. Let $b$ be the vertex following $\underline a$;  $b$ is interposed in $Q$ but not in $\partial_{\underline a}Q$, so $\Wtwo^0_{\partial_{\underline a}Q}$ is $\Wtwo^0_Q$ with $0_{\mathrm vb}$ deleted. 
\begin{itemize}
\item
Begin with $\Diamond\smile\Wtwo_{\partial_{\underline a}Q}(\etab,I)$.
\item
Shuffle this to $\Diamond\smile\Wone(\bp,\bq)$. The sign is $\sgn_{\partial_{\underline a}Q}(\etab,I)$.
\item
Move $\Diamond$ to just before $\Wone_a(p_a,q_a)$. The sign of this is $(-1)^{\abs{f_1}+\dots+\abs{f_{a-1}}}$.
\item
Move $\Diamond$ to just before $k_{\mathrm va}$. The sign of this is $(-1)^k$. 
\begin{itemize}
\item
Rename the vertical $a$ letters so that $\Diamond\mapsto k_{\mathrm va}$ and $i_{\mathrm va}\mapsto(i+1)_{\mathrm va}$ for $i\geq k$. This gives $\Wone_Q(E_a\bp,\bq)$.
\end{itemize}
\item
Shuffle this to $\Wtwo_Q(\bxi,I)$. The sign is $\sgn_Q(\bxi,I)$.
\begin{itemize}
\item
Undo the previous renaming. This gives $\Wtwo^0_Q$, then $\Wtwo^{\mathrm v}_{\partial_{\underline a}Q}(\etab)$ with $0_{\mathrm v b}$ replaced by $\Diamond$, and then $\Wtwo^{\mathrm h}(I)$.
\end{itemize}
\item
Swap $0_{\mathrm vb}$ with $\Diamond$. The sign of this is $-1$.
\item
Move $\Diamond$ to just after $\Wtwo^0_Q$. The sign of this is $(-1)^m$, where  $m$ is the number of interposed vertices before $b$ in $\W_Q$, so $b$ is the $(m+1)$'st interposed vertex and $\underline a$ is the $(m+1)$'st caesura, therefore $\sgn_Q\underline a = (-1)^{m+1} = -(-1)^m$.
\item
Move $\Diamond$ to the beginning, giving $\Diamond\smile\Wtwo_{\partial_{\underline a}Q}(\etab,I)$ again. The sign is $-(-1)^{\deg Q}$, as  before.
\end{itemize}
Again, the coefficients are equal.

Finally, to be sure that all terms have been accounted for, consider an arbitrary term of the last sum in eq.~\eqref{expanded differential}, given by some $a$, $k$, and $\bxi$. Let $\underline a := \pi_a(k)$. 

If $k$ is the only element of $\pi^{-1}_a(\underline a)$, then $a$ must be repeated in $\W_Q$, otherwise $p_a+1$ would have to be $0$, which is a contradiction, so this was included in Case 2.

If $k$ is neither first nor last in $\pi^{-1}_a(\underline a)$, then $\xi_a(k)$ is not in the image of any $\xi_b$ for $b\neq a$. This was included in Case 1. If $k$ is  last  and $\underline a$ is last in $\W_Q$, or if $k$ is first and $\underline a$ is first, then the same conclusions hold.

Suppose that $\underline a$ is not the last term of $\W_Q$, hence $\W_Q=\dots \underline a\,\underline b\dots$, for some $b$. Suppose that $\pi_a^{-1}(\underline a)$ has more than one element, so that $\bxi = E_{Q,\underline a}\bzeta$ for some unique $\bzeta \in\Ind(\W_Q;\bp)$.
If $k$ is the last  element of $\pi^{-1}_a(\underline a)$, then for $j$ the first element of $\pi_b^{-1}(\underline b)$, $\xi_a(k)=\xi_b(j)$, and this is Case 2.

Suppose that $\underline a$ is not the first letter of $\W_Q$, hence $\W_Q=\dots \underline b\,\underline a\dots$, for some $b$. Suppose that $\pi_a^{-1}(\underline a)$ has more than one element, so that $\bxi = E_{Q,\underline a}\bzeta$ for some unique $\bzeta \in\Ind(\W_Q;\bp)$. If $k$ is the first element of $\pi_a^{-1}(\underline a)$, then for $j-1$ the last element of $\pi_b^{-1}(\underline b)$, $\xi_a(k)=\xi_b(j-1)$, so
\[
(E_{Q,\underline b}\bzeta)_b(j) = 1 + \zeta_b(j-1)= 1 + \xi_b(j-1) = 1 + \xi_a(k) 
= 1 + \zeta_a(k) = (E_{Q,\underline b}\bzeta)_a(k) 
\]
and this is Case 2 again.
\end{proof}

Lemmata \ref{S-equivariance}, \ref{Rep composition}, and \ref{Rep differential} prove the main result of this section:
\begin{theorem}
\label{Rep0 thm}
For any diagram of vector spaces $\A:\Xc\to\Vect_\Bbbk$
\[
\Rep_0 : \Qop \to \Endop[\s_2 C^{\bullet,\bullet}(\A,\A),\dS]
\]
is a dg-operad homomorphism, i.e., $\s_2 C^{\bullet,\bullet}(\A,\A)$ is a $\Qop$-algebra.
\end{theorem}

\subsection{$\Qop$-subalgebras}
Some important subcomplexes of $C^{\bullet,\bullet}(\A,\A)$ are closed under the action of $\Rep_0$.

\begin{definition}
 \cite{g-s1988b}
 \label{Asimplicial}
The \emph{asimplicial subcomplex} is 
\[
C^{\bullet,\bullet}_{\mathrm a}(\A,\A) = \bigoplus_{p=0}^\infty \bigoplus_{q=1}^\infty C^{p,q}(\A,\A) .
\]
\end{definition}
\begin{theorem}
\label{Asimplicial algebra}
The asimplicial subcomplex $\s_2 C_{\mathrm a}^{\bullet,\bullet}(\A,\A)$ is a $\Qop$-subalgebra.
\end{theorem}
\begin{proof}
Note that a homogeneous cochain $f$ is in $C_{\mathrm a}^{\bullet,\bullet}$ if and only if the second component of the shifted bidegree $\norm f$ is nonnegative.

Now, for $Q\in\Qop(n)$ and $f_1,\dots,f_n\in C_{\mathrm a}^{\bullet,\bullet}$,
\[
\bigl\|\widehat Q(f_1,\dots,f_n)\bigr\| = \sum_{u=1}^n\norm{f_u} - (\deg Q,0) ,
\]
so the second component is nonnegative and 
\[
\widehat Q(f_1,\dots,f_n) \in C_{\mathrm a}^{\bullet,\bullet}(\A,\A) .
\]
Thus, the asimplicial subcomplex is closed under $\Qop$ operations.
\end{proof}

\begin{definition}
\label{Reduced}
For $p,q\in\N$, $\bar C^{p,q}(\A,\A)$ is the set of $f\in C^{p,q}(\A,\A)$ such that if some component of $\bphi\in B_p\Xc$ is an identity morphism, then $f[\bphi]=0$. The \emph{reduced subcomplex} \cite{g-s1988b} is spanned by these cochains.
\end{definition}

The reduced subcomplex is indeed a subcomplex (closed under $\dS$) and the complex $C^{\bullet,\bullet}(\A,\A)$ is a direct sum of $\bar C^{\bullet,\bullet}(\A,\A)$ with an acyclic subcomplex, therefore the cohomology is the same.

\begin{theorem}
\label{Reduced thm}
The reduced subcomplex $\s_2\bar C^{\bullet,\bullet}(\A,\A)$ is a $\Qop$-subalgebra.
\end{theorem}
\begin{proof}
For any $Q\in\Qcat(n)$, $\bp,\bq\in\N^n$, and $f_a\in\bar C^{p_a,q_a}(\A,\A)$, we need to show that $\widehat Q(f_1,\dots,f_n)\in \bar C^{p',q'}(\A,\A)$, where again $p'=p_1+\dots+p_n-\deg Q$ and $q'=q_1+\dots+q_n+1-n$.

It is sufficient to prove for any $\bzeta\in\Ind(\W_Q;\bp)$ and $I\in\Ind(\Tree_Q;\bq)$, that 
\[
\rep(\bzeta,I;f_1,\dots,f_n) \in \bar C^{p',q'}(\A,\A) .
\] 
Consider any $\bphi\in B_{p'}\Xc$ such that for some $0\leq i\leq p'-1$, $\phi_{i+1}$ is an identity.

By Definition~\ref{Word coloring}, 
\[
\bigcup_{a=1}^n \Im\zeta_a = [p']
\]
so there exists $a\in\com{n}$ such that $i\in\Im\zeta_a$.
Let $\underline a\in \W_Q$ be the last letter such that $i\in\zeta_a(\phi_a^{-1}\underline a)$.

If $\underline a$ is not the final letter of $\W_Q$, then write $\W_Q = \dots \underline a\,\underline b\dots$. Since $\underline a$ is the last letter satisfying the above condition, $i\not\in \zeta_b(\pi_b^{-1}\underline b)$. By Definition~\ref{Word coloring},
\[
i < \min\zeta_b(\pi_b^{-1}\underline b) = \max\zeta_a(\pi_a^{-1}\underline a) ,
\]
so $i+1\leq \max\zeta_a(\pi_a^{-1}\underline a)$. Since $\zeta_a(\pi_a^{-1}\underline a)$ is a set of consecutive integers and contains $i$, it must contain $i+1$ as well.

If $\underline a$ is the final letter of $\W_Q$, then $\max \zeta_a(\pi_a^{-1}\underline a) = p'$. Since that is a set of consecutive integers, all the integers from $i$ to $p'$ are in $\Im\zeta_a$. In particular, $i+1$ is.

In either case, there exists $j\in[p_a]$ such that $\zeta_a(j)=i$ and $\zeta_a(j+1)=i+1$. This shows that $(\zeta_a^*\bphi)_{j+1} = \phi_{i+1}$, which is an identity morphism.

By Definition~\ref{IndQuilt action}, $\rep(\bzeta,I;f_1,\dots,f_n)[\bphi]$ is a composition of multilinear maps, including $f_a[\zeta_a^*\bphi]=0$, so $\rep(\bzeta,I;f_1,\dots,f_n)[\bphi]=0$.

Since this holds for any $\bphi$ containing an identity, it shows that $\rep(\bzeta,I;f_1,\dots,f_n) \in \bar C^{p',q'}(\A,\A)$.
\end{proof}
Note that this proof depended upon the definition of a coloring. It does not imply that the reduced subcomplex is closed under the action of $\CQ$.

\begin{definition}
$\bar C_{\mathrm a}^{\bullet,\bullet}(\A,\A) := \bar C^{\bullet,\bullet}(\A,\A) \cap C_{\mathrm a}^{\bullet,\bullet}(\A,\A)$
\end{definition}
\begin{corollary}
$\s_2\bar C_{\mathrm a}^{\bullet,\bullet}(\A,\A)$ is a $\Qop$-subalgebra.
\end{corollary}

\section{The $\mQuilt$ Operad}
\label{mQuilts}
In this section, I construct another operad, $\mQuilt$, from $\Qop$ and an additional generator and show that the Hochschild complex of a diagram of algebras is an $\mQuilt$-algebra.
\begin{definition}
\label{ad def}
In any graded operad, let $\ad$ denote the adjoint action of the $1$-ary part. That is, if $\#A = 1$, then
\beq
\label{ad}
\ad_A B := A\circ_1 B -(-1)^{\deg A \deg B} \sum_{a=1}^{\#B}B \circ_{a} A.
\eeq
\end{definition}
This is the commutator with respect to  infinitesimal composition, $\circ'$ \cite[Sec.~6.1.3]{l-v2012}.

In the following, $\circ$ will denote the full composition operation; this can be constructed from partial composition and \emph{vice versa}.
\begin{definition}
\label{mQuilt def}
Consider the graded operad  generated by $\Qop$ and an additional generator $\m$ of arity $\#\m=0$ and degree $\deg \m=-1$ with the following relations:
\begin{enumerate}
\item
\label{MC axiom}
\[
0 = \Qtwo{1}{2}\circ(\m,\m) .
\]
\item
\label{binary axiom}
For some quilt $Q$ and $u\in \com{\#Q}$, if $u$ has more than 2 children (i.e., edges of the form $(u,v)\in\E_Q$) then $0=Q\circ_u\m$.
\item
\label{flat1 axiom}
If $u\in \com{\#Q}$ is repeated in $\W_Q$, then $0=Q\circ_u\m$.
\item
\label{flat2 axiom}
If $u\in\com{\#Q}$ and $\W_Q = \dots vuv\dots$, then $0=Q\circ_u\m$.
\item
\label{rearrangement axiom}
If $\Tree_Q=\Tree_P$, $u\in\com{\#Q}=\com{\#P}$, and $\W_Q$ and $\W_P$ are equal except for the position of $u$, then $Q\circ_u\m=P\circ_u\m$.
\end{enumerate}

Let 
\beq
\label{Delta def}
\Delta := \Qtwo{1}{2}\circ_1\m-\Qtwo{1}{2}\circ_2\m .
\eeq
Extend the boundary operator $\partial$ by $\partial\m:=0$ and let 
\beq
\label{dm}
\partial' := \partial + \ad_\Delta .
\eeq
Let $\mQuilt$ denote this graded operad with differential (boundary operator) $\partial'$.
\end{definition}

\begin{definition}
For $Q\in\Qcat$, the element 
\[
Q\circ (\id,\dots,\id,\underbrace{\m,\dots,\m}_{n\text{ times}}) 
\in \mQuilt
\]
where the last $n$ vertices occur in $\downarrow$ order
will be denoted by a diagram with the last $n$ vertices labelled by $\m$. 
\end{definition}

Any composition of a quilt with $\m$'s can be rewritten in this form.
\begin{example}
\[
\Qtwo12\circ_1\m = \Qtwo21\circ_2\m = \Qtwo\m1\  ,
\]
so
\[
\Delta = \Qtwo{\m}{1} - \Qtwo{1}{\m}\  .
\]
\end{example}
\begin{example}
Relation \ref{mQuilt def}\eqref{MC axiom} is simply
\[
0 = \Qtwo{\m}{\m} \ .
\]
\end{example}
Note that $\mQuilt$ is spanned by elements of this form, but there are additive relations among them.
\begin{remark}
Some convention of ordering the last $n$ vertices is necessary, because reordering them can introduce a sign.

On the other hand, putting $\m$ in the \emph{last} $n$ places is not so important. That merely serves to keep the numbering of vertices the same.
\end{remark}

\begin{theorem}
The operator $\partial'$ satisfies $\left(\partial'\right)^2=0$ and makes $\mQuilt$ a differential graded operad.
\end{theorem}
\begin{proof}
First, note that $\mQuilt$ with $\partial$ is a dg-operad by construction.

Secondly, $\ad_\Delta$ is a derivation of $\mQuilt$ by construction. This doesn't even depend on the definition of $\Delta$.

Thirdly, $\ad_\Delta$ anticommutes with $\partial$ because $\deg\Delta=-1$ and $\partial\Delta=0$, so we just need to compute $(\ad_\Delta)^2Q$. 

Using the infinitesimal composition notation \cite{l-v2012} 
\[
Q\circ'\Delta := \sum_{a=1}^{\#Q}Q\circ_a\Delta ,
\]
$\ad_\Delta Q = \Delta\circ_1 Q - (-1)^{\deg Q}Q\circ'\Delta$. So,
\begin{align*}
(\ad_\Delta)^2Q &=\Delta\circ_1(\ad_\Delta Q) + (-1)^{\deg Q} (\ad_\Delta Q)\circ' \Delta \\
&= \Delta\circ_1(\Delta\circ_1 Q) - (-1)^{\deg Q}\Delta\circ_1(Q\circ'\Delta)\\ &\qquad + (-1)^{\deg Q}(\Delta\circ_1Q)\circ'\Delta - (Q\circ'\Delta)\circ'\Delta \\
&= (\Delta\circ_1\Delta)\circ_1 Q - (Q\circ'\Delta)\circ'\Delta .
\end{align*}
Because $\deg\Delta=-1$, for $a\neq b$, $(Q\circ_a\Delta)\circ_b\Delta = - (Q\circ_b\Delta)\circ_a\Delta$, therefore $(Q\circ'\Delta)\circ'\Delta = Q\circ'(\Delta\circ_1\Delta)$ and 
\[
(\ad_\Delta)^2Q = (\Delta\circ_1\Delta)\circ_1 Q - Q\circ'(\Delta\circ_1\Delta) .
\]

There are 4 terms to $\Delta\circ_1\Delta$.
\beq
\label{nilpotent1}
\Qtwo\m1\circ_1\Qtwo\m1 = \Qthreeb\m\m1  
\eeq
\beq
\label{nilpotent2}
-\Qtwo\m1\circ_1\Qtwo1\m =  -\Qthreeb\m1\m
\eeq
\begin{align}
\label{nilpotent3}
- \Qtwo1\m\circ_1\Qtwo\m1  &= \Qfive\m\blank1\m\blank + \Qthreeb\m1\m + \Qfive\m1\blank\blank\m  \\
&= \Qfive\m\m\blank\blank1 + \Qthreeb\m1\m + \Qfive\m\blank\m1\blank 
\nonumber
\end{align}
The sign is cancelled by a rearrangement of $\m$'s, and the first and last terms have been rewritten by Relation \ref{rearrangement axiom}. Clearly, the second term cancels with \eqref{nilpotent2}.
\beq
\label{nilpotent4}
\Qtwo1\m\circ_1\Qtwo1\m =-\Qfive1\blank\m\m\blank - \Qthreeb1\m\m - \Qfive1\m\blank\blank\m 
\eeq
The second term here is 
\[
\Qthreeb1\m\m = \Qtwo12\circ_2 \Qtwo\m\m = 0
\]
by Relation \ref{MC axiom}. By Relation \ref{rearrangement axiom},
\[
\Qfive1\blank\m\m\blank = \Qfive1\blank23\blank\circ(\id,\m,\m) = \Qfive13\blank\blank2\circ(\id,\m,\m) = - \Qfive12\blank\blank3\circ(\id,\m,\m) = - \Qfive1\m\blank\blank\m ,
\]
so the remaining terms of \eqref{nilpotent4} cancel. The only  terms left are from \eqref{nilpotent1} and \eqref{nilpotent3}:
\[
\Delta\circ_1\Delta = 
\Qthreeb\m\m1 +\Qfive\m\m\blank\blank1 + \Qfive\m\blank\m1\blank 
= \Qtwo12\circ_1\Qtwo\m\m  = 0 .
\] 

So, $(\ad_\Delta)^2=0$, and therefore $(\partial')^2=0$.
\end{proof}

\begin{remark}
Note that the inclusion of $\Qop$ into $\mQuilt$ is a homomorphism of graded operads, but not of \emph{differential} graded operads. On the other hand, there is a dg-operad homomorphism from $\mQuilt$ to $\Qop$ defined by $\m\mapsto 0$.
\end{remark}

\subsection{Representation}
\label{Rep mQuilt}
Let $\Alg_\Bbbk$ be the category of associative algebras (not necessarily unital) and homomorphisms over a field $\Bbbk$ of arbitrary characteristic.
\begin{definition}
\label{dH def}
Let $\A:\Xc\to\Alg_\Bbbk$ be a diagram of algebras. Define $\hat\m \in C^{0,2}(\A,\A)$ by
\[
\hat\m(x;a,b) := ab 
\]
for all $x\in\Obj\Xc$ and $a,b\in\A(x)$, i.e., multiplication in $\A(x)$ expressed as a cochain. Define the \emph{Hochschild coboundary} $\dH:C^{\bullet,\bullet}(\A,\A)\to C^{\bullet,\bullet}(\A,\A)$ by
\[
\dH f := \Rep_0\left(\Qtwo12-\Qtwo21\right)(\hat\m,f) 
\]
for all $f\in C^{\bullet,\bullet}(\A,\A)$.

The \emph{Hochschild complex} of $\A$ is $\tot C^{\bullet,\bullet}(\A,\A)$ with coboundary $\delta := \dS+\dH$. The \emph{Hochschild cohomology} $H^\bullet(\A,\A)$ is the cohomology of this complex.
\end{definition}
\begin{theorem}
\label{mQuilt algebra}
For any diagram of algebras, $\A:\Xc\to\Alg_\Bbbk$,
there is a dg-operad homomorphism
\[
\Rep : \mQuilt \to \Endop[\s\tot C^{\bullet,\bullet}(\A,\A),\delta] 
\]
defined by $\Rep(Q):=\Rep_0(Q)= \widehat Q$ (Def.~\ref{Rep0 def}) for $Q\in\Qop$, and $\Rep(\m) :=\hat\m$. This makes $\s\tot C^{\bullet,\bullet}(\A,\A)$ an $\mQuilt$-algebra.
\end{theorem}
\begin{proof}
To check that this is a graded operad homomorphism, we just need to check compatibility with each of the relations defining $\mQuilt$. 
\begin{enumerate}
\item
\[
\widehat{\Qtwo\m\m} = \widehat{\Qtwo12}(\hat\m,\hat\m) = 0
\]
is equivalent to the associativity of $\hat\m$, so we have compatibility with 
Relation~\ref{MC axiom}. 
\item
Relation~\ref{binary axiom} follows because multiplication is (only) a binary operation. 
\item
$\hat\m$ has vertical degree $0$. If $u$ is repeated in $\W_Q$, then there exist no colorings of $Q$ compatible with this degree, hence 
\[
\widehat Q(f_1,\dots,f_{u-1},\hat\m,f_{u+1},\dots) = 0 .
\]
This is consistency with Relation~\ref{flat1 axiom}.
\item
Similarly, if $\W_Q=\dots vuv\dots$, then in the computation of $\widehat Q(\dots,\hat\m,\dots)$, the intervals that are intended to make up the image of $\zeta_v$ overlap at a point. This means that this is not the image of an injective map, so the action of $Q$ here must be $0$, and we have compatibility with Relation~\ref{flat2 axiom}.
\item
Relation~\ref{rearrangement axiom} follows because the functor $\A$ gives homomorphisms, hence they intertwine the products in the various algebras.
\end{enumerate}

This homomorphism property also gives that $\dS\hat\m = 0$. This is the condition of consistency with $\partial\m=0$. This shows that $\mQuilt$ with the differential $\partial$ acts on $C^{\bullet,\bullet}(\A,\A)$ with the differential $\dS$.

By construction, $\hat\Delta = \dH$. This implies that $\widehat{\ad_\Delta Q} = \ad_{\dH}\widehat Q$.
\end{proof}

\begin{remark}
More generally, this holds if $\A$ is a functor from $\Xc$ to a nonsymmetric colored operad of abelian groups, equipped with ``multiplication'' (in the sense of \cite{g-v1995}) $\hat\m\in C^{0,2}(\A,\A)$ that is natural ($\dS\hat\m=0$) and associative,
\[
\widehat{\Qtwo12}(\hat\m,\hat\m) = 0 .
\]
Everything in this paper holds in this generality, except for statements about the Maurer-Cartan equation that involve division. 
In particular, $\A$ could be a diagram of algebras over any commutative ring.
\end{remark}

\begin{definition}
 \cite{g-s1988b}
The \emph{asimplicial cohomology} $H^\bullet_{\mathrm a}(\A,\A)$ of a diagram of algebras $\A:\Xc\to\Alg_\Bbbk$ is the cohomology of $\tot C^{\bullet,\bullet}_{\mathrm a}(\A,\A)$ with the coboundary $\delta = \dS+\dH$.
\end{definition}
\begin{theorem}
\label{Asimplicial algebra2}
Under the action $\Rep$, the asimplicial subcomplex $\s\tot C_{\mathrm a}^{\bullet,\bullet}(\A,\A)$, the reduced subcomplex $\s\tot \bar C^{\bullet,\bullet}(\A,\A)$, and their intersection $\s\tot\bar C_{\mathrm a}^{\bullet,\bullet}(\A,\A)$ are $\mQuilt$-subalgebras.

The inclusions of the reduced complexes give isomorphisms $H^\bullet[\tot \bar C^{\bullet,\bullet}(\A,\A),\delta]\cong H^\bullet(\A,\A)$ and $H^\bullet[\tot \bar C_{\mathrm a}^{\bullet,\bullet}(\A,\A),\delta]\cong H_{\mathrm a}^\bullet(\A,\A)$.
\end{theorem}
\begin{proof}
 The $\mQuilt$ operations simply include $\hat\m\in \bar C_{\mathrm a}^{0,2}(\A,\A)$ in some arguments, and the conclusion is the same as for Theorems~\ref{Asimplicial algebra} and \ref{Reduced thm}. In particular, this is also why these are closed under $\dH$ and are thus subcomplexes.
 
As mentioned before, $C^{\bullet,\bullet}(\A,\A)$, with coboundary $\dS$, is the direct sum of the reduced subcomplex and an acyclic subcomplex. The spectral sequence shows that this is also true for the total complex with coboundary $\delta$. The same is true for the asimplicial subcomplexes.
\end{proof}
\begin{remark}
This reduced subcomplex is normalized with respect to the cosimplicial structure coming from the nerve of $\Xc$.
The ``normalized'' subcomplex in \cite{g-s1988b}  corresponds to the familiar normalized subcomplex for Hochschild cohomology of a unital algebra, but is only available for a diagram of \emph{unital} algebras, since only then is $C^{\bullet,\bullet}(\A,\A)$ bicosimplicial.
\end{remark}

\subsection{Boundary computation}
To compute $\ad_\Delta Q$ in practice, we will need a few ways of modifying $Q$.
\begin{definition}
Let $Q\in\Qcat(n)$. 

Given $u\in\com{n}$, define $\Tree_{Q^{(1)}_u}$ by attaching $n+1$ as the parent of $u$ in $\Tree_Q$. If $u$ already had a parent, then that becomes the parent of $n+1$. Define $\W_{Q^{(1)}_u}$ by inserting $n+1$ before the first $u$ in $\W_Q$. See Figure~\ref{Modification1}.

Given a corner $\underline u\in\Place_Q$, define $\Tree_{Q^{(2)}_{\underline u}}$ by attaching $n+1$ to $u$ at $\underline u$. Define $\W_{Q^{(2)}_{\underline u}}$ by inserting $n+1$ after the last $u$ in $\W_Q$. See Figure~\ref{Modification2}.
\begin{figure}
\begin{minipage}[b]{.49\textwidth}
\[
\raisebox{-.5\height}{\begin{tikzpicture}[scale=.7,node font=\footnotesize]
\node (a){}
child {node (b){$u$}
 child {node (c){}} 
 child {node (d){}}};
\end{tikzpicture}}
\to
\raisebox{-.5\height}{\begin{tikzpicture}[scale=.7,node font=\footnotesize]
\node (a){}
child {node (z){$n+1$}
child {node (b){$u$}
 child {node (c){}}
 child {node (d){}}}};
\end{tikzpicture}}
\]
\caption{Construction of $Q^{(1)}_u$.\label{Modification1}}
\end{minipage}
\hfill
\begin{minipage}[b]{.49\textwidth}
\[
\raisebox{-.5\height}{\begin{tikzpicture}[scale=.7,node font=\footnotesize]
\node (a){}
child {node (b){$u$}
 child {node (c){}}
 child {node (d){}}};
\end{tikzpicture}}
\to
\raisebox{-.5\height}{\begin{tikzpicture}[scale=.7,node font=\footnotesize]
\node (a){}
child {node (b){$u$}
 child {node (c){}}
 child {node (z){$n+1$}}
 child {node (d){}}};
\end{tikzpicture}}
\]
\caption{Construction of $Q^{(2)}_{\underline u}$.\label{Modification2}}
\end{minipage}
\end{figure}

Given $u\in\com{n}$, let  $v$ be is its $\vartriangleleft_Q$-first child. Define $\Tree_{Q^{(3)}_u}$ by removing the edge $(u,v)$ and attaching $n+1$ as the parent of $v$ and $u$. If $u$ already has a parent, then that becomes the parent of $n+1$. Define $\W_{Q^{(3)}_u}$ by inserting $n+1$ before the first $u$. See Figure~\ref{Modification3}.
\begin{figure}
\begin{minipage}[b]{.49\textwidth}
\[
\raisebox{-.5\height}{\begin{tikzpicture}[scale=.7,node font=\footnotesize]
\node (a){}
child {node (b){$u$}
 child {node (c){$v$}}
 child {node (d){}}
 child {node (e){}}};
\end{tikzpicture}}
\to
\raisebox{-.5\height}{\begin{tikzpicture}[scale=.7,node font=\footnotesize]
\node (a){}
child {node (z){$n+1$}
 child {node (c){$v$}}
child {node (b){$u$}
 child {node (d){}}
 child {node (e){}}}};
\end{tikzpicture}}
\]
\caption{Construction of $Q^{(3)}_u$.\label{Modification3}}
\end{minipage}
\begin{minipage}[b]{.49\textwidth}
\[
\raisebox{-.5\height}{\begin{tikzpicture}[scale=.7,node font=\footnotesize]
\node (a){}
child {node (b){$u$}
 child {node (c){}}
 child {node (d){}}
 child {node (e){$w$}}};
\end{tikzpicture}}
\to
\raisebox{-.5\height}{\begin{tikzpicture}[scale=.7,node font=\footnotesize]
\node (a){}
child {node (z){$n+1$}
child {node (b){$u$}
 child {node (d){}}
 child {node (e){}}}
  child {node (c){$w$}}};
\end{tikzpicture}}
\]
\caption{Construction of $Q^{(4)}_u$.\label{Modification4}}
\end{minipage}
\end{figure}

Given $u\in\com{n}$, let $w$ be is its $\vartriangleleft_Q$-last child. Define $\Tree_{Q^{(4)}_u}$ by removing the edge $(u,w)$ and attaching $n+1$ as the parent of  $u$ and $w$. If $u$ already has a parent, then that becomes the parent of $n+1$.  Define $\W_{Q^{(3)}_u}$ by inserting $n+1$ before the first $u$. See Figure~\ref{Modification4}.

Given $u\in\com{n}$ and $\vartriangleleft_Q$-consecutive children $v,w\in\com{n}$, define $\Tree_{Q^{(5)}_{u,v,w}}$ by attaching $n+1$ as a child of $u$ and the parent of $v$ and $w$. Define $\W_{Q^{(5)}_{u.v,w}}$ by inserting $n+1$ after the last $u$. See Figure~\ref{Modification5}.
\begin{figure}
$
\raisebox{-.5\height}{\begin{tikzpicture}[scale=.7,node font=\footnotesize]
\node (a){}
child {node (b){$u$}
 child {node (c){}}
 child {node (d){$v$}}
 child {node (e){$w$}}
 child {node (f){}}};
\end{tikzpicture}}
\to
\raisebox{-.5\height}{\begin{tikzpicture}[scale=.7,node font=\footnotesize]
\node (a){}
child {node (b){$u$}
 child {node (c){}}
 child {node (z){$n+1$}
 child {node (d){$v$}}
 child {node (e){$w$}}}
 child {node (f){}}};
\end{tikzpicture}}
$
\caption{Construction of $Q^{(5)}_{u,v,w}$.\label{Modification5}}
\end{figure}
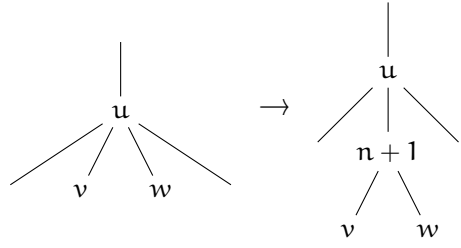
\end{definition}

I will not call upon the following result explicitly, but it stands behind many of the calculations in the next section.
\begin{lemma}
Given $Q\in\Qcat(n)$,
\[
\ad_\Delta Q = \sum_{u} (-1)^{1+\deg Q} \left(Q^{(3)}_u+Q^{(4)}_u\right)\circ_{n+1}\m + \sum_{u,v,w} (-1)^{\deg Q} Q^{(5)}_{u,v,w}\circ_{n+1}\m
\]
where the sums are over the sets for which these modified quilts are defined. That is, the first sum is over $u\in\com{n}$ with children; the second sum is over pairs of successive children. 
\end{lemma}
\begin{proof}
First, by direct computation,
\begin{multline*}
\begin{aligned}
(-1)^{\deg Q}\ad_\Delta Q &= (-1)^{\deg Q} \left(\Qtwo12\circ_1\m-\Qtwo12\circ_2\m\right)\circ_1Q \\
&\qquad - \sum_{u=1}^n Q\circ_u \left(\Qtwo12\circ_1\m-\Qtwo12\circ_2\m\right) \\
&=
\left(\Qtwo21\circ_1 Q - \Qtwo12\circ_1 Q\right)\circ_{n+1}\m \\
&\qquad +
\sum_{u=1}^n \left(-Q\circ_u\Qtwo21 + Q\circ_u\Qtwo12\right)\circ_{u+1}\m 
\end{aligned}
\\
=
\left[\Qtwo21\circ_1 Q - \Qtwo12\circ_1 Q + \sum_{u=1}^n \left(-Q\circ_u\Qtwo21 + Q\circ_u\Qtwo12\right)^{(u+1\dots n+1)}\right]\circ_{n+1}\m ,
\end{multline*}
where $(u+1\dots n+1)$ is the cyclic permutation taking $n+1\mapsto u+1$, \emph{et cetera}.

Now, consider each of these terms. The following signs are all positive, because $\deg\Qtwo12=0$. 
\beq
\label{delta1}
\left(\Qtwo21\circ_1Q\right)\circ_{n+1}\m = Q^{(1)}_v \circ_{n+1}\m
\eeq
where $v\in\com{n}$ is the root of $Q$.
\beq
\label{delta2}
\left(\Qtwo12\circ_1Q\right)\circ_{n+1}\m = \sum_{\underline u\in\Place_Q} Q^{(2)}_{\underline u}\circ_{n+1}\m .
\eeq
\beq
\label{delta3}
\left(Q\circ_u\Qtwo21\right)^{(u+1\dots n+1)}\circ_{n+1}\m
= \left(Q^{(1)}_u + Q^{(3)}_u+Q^{(4)}_u\right) \circ_{n+1}\m
\eeq
where the latter 2 terms only occur if $u$ has children.
\begin{multline}
\label{delta4}
\left(Q\circ_u\Qtwo12\right)^{(u+1\dots n+1)}\circ_{n+1}\m = \\
\sum_{\substack{\underline a\in\Place_Q,\\ a=u}} Q^{(2)}_{\underline a}\circ_{n+1}\m + 
\sum_{(u,v)\in\E_Q} Q^{(1)}_v\circ_{n+1}\m +
\sum_{\substack{(u,v),(u,w)\in\E_Q\\\text{consecutive}}} Q^{(5)}_{u,v,w} \circ_{n+1}\m .
\end{multline}
When summed over $u$, the first term on the right of eq.~\eqref{delta4} becomes a sum over all of $\Place_Q$ and cancels with \eqref{delta2}. The second term becomes a sum over all $v\in\com{n}$ except the root; this combines with \eqref{delta1} to give a sum over all vertices and cancels with the first term of \eqref{delta3}. The third term adds up to a sum over all pairs of consecutive children.
\end{proof}

\section{Gerstenhaber Algebra up to Homotopy}
\label{Gerstenhaber}
In this section, I present a direct and explicit  proof that the Hochschild cohomology and asimplicial cohomology of a diagram of algebras are Gerstenhaber algebras. For each identity that the product and bracket should satisfy, there is an explicit homotopy on cochains defined by the action of an element of $\mQuilt$.

First, define some particularly useful elements of $\mQuilt$.
\begin{definition}
\begin{align*}
M_2 &:= \Qfive\m1\blank\blank2, 
&
P_2 &:= \Qtwo12 + \Qthreea\m21,
&
L_2 := P_2 - P_2^{(12)}\ .
\end{align*}
\end{definition}
\begin{remark}
$M$ stands for ``multiplication'', $L$  for ``Lie'', and $P$ for ``pre-Lie''.
\end{remark}

By direct computation, $\widehat{M_2}(f,g) = (-1)^{\abs f}f\smile g$, where $\smile$ is the cup product of cochains \cite{g-s1988a,g-s1988b,haw2018}, so $\widehat{M_2}$ is the degree shifted cup product. Likewise, $\widehat{P_2}(f,g) = f\comp g$ is the generalized composition product, so $\widehat{L_2}(f,g) = f\comp g - (-1)^{\abs f \abs g}g\comp f = [f,g]$ is the generalized Gerstenhaber bracket.

As we shall see, this cup product and bracket induce a Gerstenhaber algebra structure on Hochschild cohomology. The operations $M_2$ and $L_2$ satisfy the relations of the (shifted) Gerstenhaber operad either exactly or up to homotopy.

The simplest relation is associativity, which looks slightly unfamiliar because of the degree shift.
\begin{lemma}
\label{Associative}
$0=M_2\circ_1 M_2 + M_2\circ_2 M_2$.
\end{lemma}
\begin{proof}
By Relation~\ref{mQuilt def}\eqref{MC axiom},
\begin{align*}
0 &= 
\Qwrap{\begin{array}{|c|c|c|}
\hline
\multicolumn{3}{|c|}{4}\\
\hline
1 & \blank & \blank \\
\hline
\blank & 2 & \blank \\
\hline
\blank & \blank & 3\\
\hline
\end{array}}
\circ_4\Qtwo\m\m
=
\Qwrap{\begin{array}{|c|c|c|}
\hline
\multicolumn{3}{|c|}{\m}\\
\hline
\multicolumn{2}{|c|}{\m} & \blank\\
\hline
1 & \blank & \blank \\
\hline
\blank & 2 & \blank \\
\hline
\blank & \blank & 3\\
\hline
\end{array}}
+
\Qwrap{\begin{array}{|c|c|c|}
\hline
\multicolumn{3}{|c|}{\m}\\
\hline
\blank &\multicolumn{2}{|c|}{\m} \\
\hline
1 & \blank & \blank \\
\hline
\blank & 2 & \blank \\
\hline
\blank & \blank & 3\\
\hline
\end{array}}
\\
&= \Qfive\m1\blank\blank2 \circ_1 \Qfive\m1\blank\blank2 + \Qfive\m1\blank\blank2 \circ_2 \Qfive\m1\blank\blank2
= M_2\circ_1 M_2 + M_2\circ_2 M_2
\end{align*}
In the first line, another 8 terms vanish because of Relation~\ref{mQuilt def}\eqref{binary axiom}.
\end{proof}

Next is homotopy commutativity of the cup product. Again, this looks different because of the degree shift.
\begin{lemma}
\label{Commutative lem}
$M_2+M_2^{(12)} = - \partial' P_2$
\end{lemma}
\begin{proof}
\[
\partial' \Qtwo12 = \ad_\Delta \Qtwo12 = - \Qfive\m1\blank\blank2 - \Qfive\m\blank12\blank
= -M_2 - \Qfive\m\blank12\blank
\]
\[
\partial' \Qthreea\m21 = \partial \Qthreea\m21 
= - \Qfive\m2\blank\blank1 + \Qfive\m\blank12\blank
= - M_2^{(12)} + \Qfive\m\blank12\blank
\]
\end{proof}
\begin{remark}
Alternatively, the cup product could have been defined using
\[
\Qfive\m\blank21\blank ,
\]
which is homotopic to $M_2$  by 
\[
\Qthreea\m12 .
\]
$M_2$ was chosen for consistency with \cite{g-s1988b}.
\end{remark}

In the case of a single algebra, the Gerstenhaber bracket satisfies the Jacobi identity exactly. Here, it is only satisfied up to homotopy. 
\begin{definition}
\label{P3}
\begin{align*}
P'_3 &:= \Qthreea132 + \Qfoura\m123 ,
&
L_3 &:= \sum_{\sigma\in\perm_3}\sgn(\sigma)\,{P'_3}^\sigma .
\end{align*}
\end{definition}
\begin{remark}
I am avoiding a clash of notation with a slightly different $P_3$ later (Def.~\ref{P def}).
\end{remark}
\begin{remark}
The following is a special case of Theorem~\ref{Linfty thm2} below. I am presenting it here for the sake of the explicit proof.
\end{remark}
\begin{lemma}
\label{Jacobi lem}
The Jacobiator of $L_2$ is
\[
L_2\circ_1L_2 + (L_2\circ_1L_2)^{(123)} + (L_2\circ_1L_2)^{(321)} = \partial' L_3 .
\qedhere
\]
\end{lemma}
\begin{proof}
The Jacobiator of $L_2$ is the alternating sum of all permutations of the associator of $P_2$. Using the identity,
\[
0 = \Qfourd4321\circ_4\Qtwo\m\m = \Qsixe\m\m\blank321+\Qsixc\m\blank\m321
= \Qfivec\m\m132+\Qsixc\m\blank\m321,
\]
the associator is
\begin{multline*}
P_2\circ_1P_2 - P_2 \circ_2 P_2 = \Qfive12\blank\blank3+\Qfive1\blank23\blank+\Qfive\m31\blank2+\Qfive\m21\blank3+\Qfive\m\blank132+\Qfive\m213\blank \\
 - \Qfourb\m213 - \Qfourc1\m32 
- \Qeightb\m\blank\m3\blank1\blank2
+\Qeightb\m\blank\m\blank213\blank .
\end{multline*}
Note that the third and fourth terms cancel when this is antisymmetrized. 

To compute $\partial' P'_3$, we need  
\beq
\label{dP1}
\partial' \Qthreea132 = 
\Qfive1\blank23\blank - \Qfive13\blank\blank2 + \Qfive\m1\blank32 - \Qfourc1\m32 + \Qfive\m\blank132
\eeq
and, using $0 = \Qfoura4123\circ_4\Qtwo\m\m$,
\beq
\label{dP2}
\partial'\Qfoura\m123 = - \Qfive\m1\blank32 + \Qfourb\m123 - \Qfive\m123\blank + \Qeighta\m\m\blank1\blank2\blank3
+ \Qeighta\m\m\blank\blank123\blank .
\eeq
Adding these together, the third term of \eqref{dP1} cancels the first term of \eqref{dP2}. The last 2 terms of \eqref{dP2} can be rewritten using Relations \ref{mQuilt def}\eqref{MC axiom} and \ref{mQuilt def}\eqref{binary axiom} as above. This leaves
\[
\partial' P'_3 = \Qfive1\blank23\blank - \Qfive13\blank\blank2  - \Qfourc1\m32 + \Qfive\m\blank132 + \Qfourb\m123 - \Qfive\m123\blank
-\Qeightb\m\blank\m1\blank2\blank3
- \Qeightb\m\blank\m\blank123\blank .
\]
Each term here corresponds to a term of the associator. Either the terms are equal, or they differ by a permutation and the sign of that permutation. As already noted, the two other terms of the associator cancel each other when antisymmetrized. This means that the antisymmetrization of the associator equals the antisymmetrization of $\partial' P'_3$.
\end{proof}

Finally, we need to show that the Gerstenhaber bracket, $L_2$ is a derivation of $M_2$ up to homotopy. In the case of a single algebra, this proof breaks naturally into two pieces; the composition product  is exactly a derivation of the cup product in its first argument but only a derivation up to homotopy in its second argument. In the present case, these properties are both only true up to homotopy, but it is still convenient to deal with them separately.

First,  consider how close  $P_2$ is to being a derivation of $M_2$ in its first argument.
\begin{definition}
\[
C_3 :=  \Qfive\m1\blank32
\]
\end{definition}
\begin{lemma}
\label{Derivation lem1}
$P_2\circ_1 M_2 - (M_2\circ_1 P_2)^{(23)} - M_2\circ_2 P_2 = 
\partial' C_3$.
\end{lemma}
\begin{proof}
\beq
\label{PM}
P_2\circ_1 M_2 = \Qsevena\m1\blank\blank23\blank + \Qsevena\m1\blank\blank2\blank3 +
\Qnineb\m\blank\m31\blank\blank\blank2 +
\Qnineb\m\blank\m\blank1\blank3\blank2
\eeq
\beq
\label{MP1}
M_2\circ_1P_2 = \Qsevena\m1\blank2\blank\blank3 + \Qnineb\m\blank\m21\blank\blank\blank3
\eeq
\beq
\label{MP2}
M_2\circ_2P_2 = \Qsevena\m1\blank\blank2\blank3 +\Qninea\m\m\blank1\blank\blank\blank32
\eeq
Relations~\ref{mQuilt def}\eqref{MC axiom} and \ref{mQuilt def}\eqref{rearrangement axiom} has been used to rewrite the last terms of \eqref{MP1} and \eqref{MP2}.

Combining these, the third term of \eqref{PM} cancels with the second term of \eqref{MP1}, and the second term of \eqref{PM} cancels with the first term of \eqref{MP2}. This leaves
\[
 P_2\circ_1 M_2 - (M_2\circ_1 P_2)^{(23)} - M_2\circ_2 P_2 
 = \Qsevena\m1\blank\blank23\blank - \Qsevena\m1\blank3\blank\blank2 + \Qnineb\m\blank\m\blank1\blank3\blank2 -\Qninea\m\m\blank1\blank\blank\blank32
 \] 
On the other hand,
\[
\partial' C_3 = \partial' \Qfive\m1\blank32 = -\Qsevena\m1\blank3\blank\blank2 + \Qsevena\m1\blank\blank23\blank  - \Qninea\m\m\blank\blank1\blank3\blank2 - \Qninea\m\m\blank1\blank\blank\blank32
\]
These are equal, once the penultimate term has been rewritten by Relation~\ref{mQuilt def}\eqref{MC axiom}.
\end{proof}

\begin{definition}
\[
D_3 := \Qfive12\blank\blank3 + \Qfive\m21\blank3 + \Qeighta\m\m\blank2\blank1\blank3 .
\]
\end{definition}
\begin{lemma}
\label{Derivation lem2}
$P_2\circ_2 M_2 - M_2\circ_1P_2-(M_2\circ_2P_2)^{(12)} = \partial' D_3$
\end{lemma}
\begin{proof}
Two terms were computed in the previous lemma. That leaves
\[
P_2\circ_2M_2 =  \Qsixa1\m2\blank\blank3 + \Qeightc\m\m\blank2\blank\blank31
\]
so
\begin{multline*}
P_2\circ_2 M_2 - M_2\circ_1P_2-(M_2\circ_2P_2)^{(12)} = \\
\Qsixa1\m2\blank\blank3 - \Qsevena\m1\blank2\blank\blank3 - \Qsevena\m2\blank\blank1\blank3
+ \Qeightc\m\m\blank2\blank\blank31
- \Qnineb\m\blank\m21\blank\blank\blank3
- \Qninea\m\m\blank2\blank\blank\blank31
\end{multline*}

Now the boundaries:
\[
\partial' \Qfive12\blank\blank3 = - \Qsevena\m\blank12\blank\blank3 + \Qsixa1\m2\blank\blank3 - \Qsevena\m1\blank2\blank\blank3
\]
\[
\partial' \Qfive\m21\blank3 = - \Qsevena\m2\blank\blank1\blank3 + \Qsevena\m\blank12\blank\blank3 - \Qnineb\m\blank\m21\blank\blank\blank3 - \Qnineb\m\blank\m2\blank1\blank3\blank
\]
\[
\partial' \Qeighta\m\m\blank2\blank1\blank3
= 
- \Qninea\m\m\blank2\blank\blank\blank31 + \Qeightc\m\m\blank2\blank\blank31
- 
\Qninea\m\m\blank2\blank1\blank3\blank .
\]
Adding these up, the first term of the first line cancels the second term of the second line, and the last terms of the last two lines cancel by Relation~\ref{mQuilt def}\eqref{MC axiom}.
\end{proof}

\begin{definition}
\label{Gerstenhaber def}
The suspended Gerstenhaber operad, $\Susp\Gerst$ \cite{l-v2012} is the graded operad generated by $m,\ell\in\Susp\Gerst(2)$ with $\deg m=1$, $\deg\ell=0$, and the relations
\begin{align*}
0 &= m+m^{(12)} = \ell+ \ell^{(12)} & 0 &= m\circ_1 m + m\circ_2 m \\
0 &= \ell\circ_1\ell + (\ell\circ_1\ell)^{(123)} + (\ell\circ_1\ell)^{(321)} &
0 &= \ell\circ_1 m - (m\circ_1\ell)^{(23)} - m\circ_2\ell .
\end{align*}
\end{definition}

\begin{theorem}
\label{Gerstenhaber thm}
There exists a graded  operad homomorphism 
\[
\Ghom :\Susp\Gerst \to H_\bullet(\mQuilt)
\]
defined by $\Ghom(m) = [M_2]$ and $\Ghom(\ell) = [L_2]$. Consequently, for any diagram of algebras $\A:\Xc\to\Alg_\Bbbk$, the cup product and generalized Gerstenhaber bracket induce on Hochschild cohomology $H^\bullet(\A,\A)$ and asimplicial cohomology $H_{\mathrm a}^\bullet(\A,\A)$ the structure of Gerstenhaber algebras.
\end{theorem}
\begin{proof}
Firstly, by direct computation, $\partial'M_2=0$. By Lemma~\ref{Commutative lem}, $\partial'L_2 = \partial'P_2-\partial'P_2^{(12)}=0$. Therefore, $[M_2],[L_2]\in H_\bullet(\mQuilt)$.

The definition of $\Ghom$ is consistent with $0=m+m^{(12)}$ by Lemma~\ref{Commutative lem}. It is consistent with $0=\ell+ \ell^{(12)}$ by the definition of $L_2$. It is consistent with $0= m\circ_1 m + m\circ_2 m$ by Lemma~\ref{Associative}. It is consistent with $0=\ell\circ_1\ell + (\ell\circ_1\ell)^{(123)} + (\ell\circ_1\ell)^{(321)}$ by Lemma~\ref{Jacobi lem}. 

Applying the permutation $(123)$ to the conclusion of Lemma~\ref{Derivation lem2} gives
\[
P_2^{(12)}\circ_1M_2 - \left(M_2\circ_1P_2^{(12)}\right)^{(23)}-M_2\circ_2P_2^{(12)} = \partial'D_3^{(123)} .
\]
Subtracting this from the conclusion of Lemma~\ref{Derivation lem1} gives
\[
L_2\circ_1 M_2 - (M_2\circ_1 L_2)^{(23)} - M_2\circ_2 L_2 = \partial'\left(C_3+D_3^{(123)}\right) .
\]
This shows consistency with  $0=\ell\circ_1 m - (m\circ_1\ell)^{(23)} - m\circ_2\ell$.

Theorem \ref{mQuilt algebra} shows that $\s\tot C^{\bullet,\bullet}(\A,\A)$ is a (differential graded) $\mQuilt$-algebra. This implies that its cohomology $\s\, H^\bullet(\A,\A)$  is an $H_\bullet(\mQuilt)$-algebra. The homomorphism $\Ghom$ then induces a Gerstenhaber algebra structure. By construction, the shifted product $m$ is given up to cohomology by $\widehat{M_2}$, which is the shifted cup product, hence the product on cohomology is given by the cup product on cochains. Also by construction, the Lie bracket on cohomology is given by the generalized Gerstenhaber bracket, $\widehat{L_2}$.

By Theorem \ref{Asimplicial algebra2}, the asimplicial  subcomplex is an $\mQuilt$-subalgebra, so the same conclusions follow for asimplicial cohomology.
\end{proof}

I stated the following result without proof in \cite{haw2018}. It is now simple to prove with the tools at hand.
\begin{corollary}
The space of degree $1$ Hochschild cocycles
\[
Z^1(\A,\A) := \{f\in C^{0,1}(\A,\A)\oplus C^{1,0}(\A,\A) \mid \delta f =0\}
\]
is a Lie algebra under the operation $\widehat{L_2}$.
\end{corollary}
\begin{proof}
Let $f_a\in Z^1(\A,\A)$ with bidegree $(p_a,q_a)$ for $a=1,2,3$. Because these are closed, their Jacobiator simplifies to $\delta[\widehat{L_3}(f_1,f_2,f_3)]$. Up to permutation, there are 2 types of terms in $\widehat{L_3}(f_1,f_2,f_3)$. The first is
\[
\widehat{\Qthreea132}(f_1,f_2,f_3) .
\]
Because $q_1<2$, there is no possible coloring of this quilt with these degrees, therefore this term vanishes. The second is
\[
\widehat{\Qfoura\m123}(f_1,f_2,f_3) = \widehat{\Qfoura4123}(f_1,f_2,f_3,\hat\m) .
\]
Because $p_2<2$, there is no possible coloring, therefore this term vanishes. Therefore, $\widehat{L_3}(f_1,f_2,f_3)=0$ and the Jacobiator vanishes.
\end{proof}
\begin{definition}
The \emph{squaring map}  \cite[\S 21.2]{g-s1988b}  is a function $\Sq:C^{\bullet,\bullet}(\A,\A) \to C^{\bullet,\bullet}(\A,\A)$ defined by $\Sq(f) = f\comp f$, where $\comp$ is again the generalized composition product, $\widehat{P_2}$.
\end{definition}
If $\Char\Bbbk\neq2$ and the total degree of $f$ is even, then $\Sq(f) = \frac12[f,f]$, and by Theorem~\ref{Gerstenhaber thm} this induces a function $H^{2j}(\A,\A)\to H^{4j-1}(\A,\A)$.

\begin{theorem}
If $\Bbbk$ is a field of characteristic $2$, then the squaring map  induces a function $H^k(\A,\A)\to H^{2k-1}(\A,\A)$.
\end{theorem}
\begin{proof}
In characteristic $2$, there are no signs to worry about. 

For any $f,g\in C^{\bullet,\bullet}(\A,\A)$, $\partial'M_2=0$ means that
\[
\delta(f\smile g)=\delta f\smile g + f\smile \delta g .
\]
Lemma~\ref{Commutative lem} means that
\[
\delta (f\comp g) = \delta f\comp g + f\comp\delta g+f\smile g+g\smile f .
\]

Suppose $\delta f=0$.  This shows that $\delta\Sq(f) = \delta(f\comp f) = 2f\smile f=0$, so $\Sq(f)$ is closed. Furthermore,
\[
\delta (f\comp g) = f\comp\delta g+f\smile g+g\smile f ,
\]
\[
\delta (g\comp f) = \delta g\comp f +f\smile g+g\smile f,
\]
and 
\begin{align*}
\delta (g\comp \delta g) &= \delta g\comp\delta g + \delta g\smile g + g\smile \delta g \\
&=  \delta g\comp\delta g + \delta(g\smile g) ,
\end{align*}
so
\begin{align*}
\Sq(f+\delta g)-\Sq(f) &= \delta g\comp f + f\comp\delta g + \delta g\comp\delta g\\
&= \delta (f\comp g+ g\comp f+  g\comp\delta g + g\smile g) .
\end{align*}
Therefore, the cohomology class of $\Sq(f)$ only depends upon the cohomology class of $f$.
\end{proof}

\section{Strong Homotopy Lie Algebras}
\label{L-infinity}
In this section, I will show that the Hochschild complex (and asimplicial complex) of a diagram of algebras $\A$, graded by the shifted total degree, is an $\Linfty$-algebra and that deformations of $\A$ are the solutions of the Maurer-Cartan equation.
\begin{definition}
\label{Linfty def}
$\Linfty$ is the dg-operad generated by the elements $\ell_n\in \Linfty(n)$ for $n=2,3,\dots$ satisfying $\deg\ell_n=n-2$, 
\beq
(\ell_n)^\sigma = \sgn(\sigma)\,\ell_n
\eeq
for all $\sigma \in\perm_n$, and
\beq\label{Linfinity_relation}
0 = \partial\ell_n + \sum_{\substack{p,q\geq2\\ p+q=n+1}}\sum_{\sigma\in\Sh_{p-1,q}} (-1)^{(p-1)q}\sgn\sigma\,\left(\ell_p\circ_p\ell_q\right)^{\sigma^{-1}} ,
\eeq
where $\Sh_{p-1,q}\subset\perm_n$ is the set of  $(p-1,q)$-shuffles.
\end{definition}

\subsection{$\Linfty$ in $\Qop$} 
The first step is to construct a dg-operad homomorphism, $\homone:\Linfty\to\Qop$. It is enough to describe this by giving each of the images of generators. These must satisfy the same relations as the generators.

\begin{lemma}
If $Q\in\Qcat(n)$ with $n\geq2$, and $a\in\com{n}$ is the root of $\Tree_Q$, then $\W_Q=a\dots$, $a$ only occurs once in $\W_Q$, and $\deg Q\leq n-2$
\end{lemma}
\begin{proof}
If $\W_Q = \dots u \dots a \dots$, then by Axiom~\ref{Quilt def}\eqref{Quilt1}, $u\ngtr_Q a$, but this contradicts $a$ being the root. Therefore, $\W_Q=a\dots$, with $a$ only occurring once.

Since $a$ only occurs first, $\W_Q$ can be written as the concatenation of $a$ with the remainder. By Lemma~\ref{Word length}, $\deg Q \leq n-2$.
\end{proof}
The sole exception to the latter conclusion is the identity element,  $\id = \fbox{1}\in\Qcat(1)$ with $\deg\fbox{1}=0\nleq-1$ for $n=1$.
\begin{definition}
$Q\in\Qcat$ is \emph{maximal} if $\deg Q = \#Q-2$.
\end{definition}
Since $\ell_n\in \Linfty(n)$ has $\deg\ell_n=n-2$, if $\homone:\Linfty\to\Qop$ is a dg-operad homomorphism, then $\homone(\ell_n)$ must be an integer combination of maximal quilts. In fact, I will construct it as an alternating sum of \emph{all} the maximal quilts.

\begin{lemma}
\label{Maximal}
If $Q$ is a maximal quilt, then $\#Q\geq2$ and $\W_Q = ab\dots b$, where $a$ is the root and $b$ a leaf (has no children), and these are the only $\vartriangleleft_Q$-maximal vertices. (In this notation, $b\dots b$ could represent a single $b$.)
\end{lemma}
\begin{proof}
Firstly, $\deg Q = \#Q-2 \geq 0$ implies $\#Q\geq2$.

By the previous lemma, the root is the first vertex of $\W_Q$. Denote this as $a$ and the next vertex as $b$, so $\W_Q = ab\dots$. By Lemma~\ref{Word length}, there is only 1 last-first pair, therefore the last $b$ must be at the end of $\W_Q$, so $\W_Q = ab\dots b$.

If $u$ is any other vertex, then $\W_Q = ab\dots u\dots b$, so by Axiom~\ref{Quilt def}\eqref{Quilt2}, $u\vartriangleleft_Q b$. In particular, $u\ngtr_Q b$, so $b$ is a leaf.
\end{proof}

Note that in a maximal quilt, all vertices are interposed, except the first 2 in $\downarrow$ order (i.e., the first 2 in $\W_Q$). This simplifies calculations of signs for extensions below.

If $Q$ is a face of a maximal quilt or an extension of a maximal quilt by a maximal quilt, then $\deg Q = \#Q-3$. Such quilts will be crucial in checking that the $\Linfty$-relations are satisfied.

\begin{lemma}
\label{Submaximal}
If $Q\in\Qcat$ with $\deg Q = \#Q-3$, then $\#Q\geq3$ and $\W_Q$ takes the form
\begin{enumerate}
\item
$\W_Q = ab\dots bc\dots c$, or
\item
$\W_Q = \dots ab\dots bc \dots ca\dots$.
\end{enumerate}
\end{lemma}
\begin{proof}
Firstly, $\deg Q = \#Q-3\geq0$ implies that $\#Q\geq3$.

By Lemma~\ref{Word length}, there must be exactly 2 last-first pairs in $\W_Q$. The first pair is of the root and whatever comes next. Denote the second pair as $bc$. Since $b$ is not the root, something must precede the first $b$; denote this as $\underline a$, so $\W_Q= \dots \underline a\,b\dots bc\dots c\dots$.

There are now 2 possibilities: $a$ is or is not the root.
\begin{enumerate}
\item
If $a$ is the root, then it is the first vertex, so $\W_Q = ab\dots bc\dots c\dots$. Since $ab$ and $bc$ are the only last-first pairs, the last $c$ must be the end of $\W_Q$, therefore $\W_Q = ab\dots bc\dots c$.
\item
Suppose that $a$ is not the root. If $\underline a$ is the last $a$, then $\W_Q$ is a concatenation of $4$ subwords: the root, the remainder up to $\underline a$, $b\dots b$, and the remainder $c\dots$. By Lemma~\ref{Word length}, this would imply that $\deg Q\leq \#Q-4$, which it isn't. Therefore, there is another $a$, somewhere after the last $c$.

Because $\deg Q=\#Q-3$, $\W_Q$ contains precisely 2 last-first pairs: the root and whatever follows it, and $bc$. The last $c$ cannot be followed directly by the first occurrence of some vertex, therefore $\W_Q = \dots ab\dots bc \dots ca\dots$. \qedhere
\end{enumerate}
\end{proof}

\begin{lemma}
\label{Extensions}
Now suppose that $R\in\Qcat(n)$ is an extension up to permutation in the sense  that there exist $p,q\in\N$ with $p+q=n+1$, maximal quilts $P\in\Qcat(p)$ and $Q\in\Qcat(q)$, and a shuffle $\sigma\in\Sh_{p-1,q}$, such that $R^\sigma \in\Ext(P,Q,p)$. 
Denote $K= \sigma(\{p,\dots,p+q-1\})\subset \com{p+q-1}$. 
\begin{enumerate}
\item
If $\W_R = ab\dots bc\dots c$ then either 
\begin{itemize}
\item
$K=\{u \text{ occurring in }ab\dots b\} =:K_1$, or
\item
$K=\{u\mid u\geq_R b\} =:K_2$ and $c>_R b$.
\end{itemize}
\item
If $\W_R = \dots ab\dots bc\dots ca\dots$, then
\begin{itemize}
\item
$K= \{ u \text{ occurring in }b\dots bc\dots c\mid u\geq_R b\}=:K_3$ and $c>_R b$.
\end{itemize}
\end{enumerate} 

Conversely, if $R\in\Qcat$ with $\deg R = \#R-2$, then in each of these 3 cases,  $R$ actually is an extension up to permutation. 
\end{lemma}
\begin{proof}
Let $\com{q}\stackrel\alpha\into\com{n}\stackrel\beta\onto\com{p}$ be the $\Comcat$-extension at $p$.
Note that $K$ must be the set of vertices of a subtree of $\Tree_R$. Applying $\beta\circ\sigma^{-1}$ to each term of $\W_R$ and eliminating repetitions gives $\W_P$.

First, consider Case (1) that $\W_R = ab\dots bc\dots c$.

Since $P$ is maximal, $\W_P$ can only have one last-first pair, so one of the pairs in $\W_R$ must be killed, therefore either $a,b\in K$ or $b,c\in K$.

Suppose that $a,b\in K$ and $c\notin K$. Since $K$ is collapsed to a single vertex and relabelled as $p$, this means that $\W_P = p\dots$, therefore $p$ is the root of $P$ and only occurs once in $\W_P$.  This means that no vertex of $R$ after the last $b$ can be in $K$ and no vertex before the last $b$ can not be in $K$. Therefore, $K$ is the set of vertices in the subword up to the last $b$ in $\W_R$. This is $K_1$.

Now, suppose that $b,c\in K$, but $a\notin K$. In this case, $p$ is the second vertex in $\W_P$. By Lemma~\ref{Maximal}, $p$ is also the final vertex of $\W_P$, is $\vartriangleleft_P$-maximal, and is a leaf.  Since $b\in K$, any $u\geq_R b$ must map to a descendent of $p$ in $\Tree_{P}$, which can only be $p$, so $u\in K$. On the other hand, $K$ must be the set of vertices of a subtree of $\Tree_R$. Since $b$ is a child of $a\notin K$, $b$ must be the root of this subtree, so $u\in K$ implies that $u\geq_R b$. 
Therefore
\[
K = \{u \mid u\geq_R b\} .
\]
This is $K_2$. Note that since $c\in K$, we have $b<_Rc$.

Now consider Case (2) that $\W_R = \dots ab\dots bc\dots ca\dots$. As above, the quotient must collapse precisely one of the last-first pairs in $\W_R$. 

Suppose that the root and the second vertex are in $K$. The second vertex must also be the final vertex, therefore  $\W_P$ begins and ends with $p$. This cannot happen for $P$ maximal (or any quilt except $\fbox1$) so this is a contradiction.

Therefore, $b,c\in K$. Since the subtree with vertices $K$ is just $\Tree_Q$ relabelled (by $\sigma\circ\alpha$) $b$ must be its root, therefore $u\in K \implies u\geq_Rb$ (and in particular, $b<_Rc$). Similarly, the subword of $\W_R$ made up of vertices from $K$, with repetitions eliminated is just $\W_Q$ relabelled, therefore $c$ is both the second and last vertex from $K$ in $\W_R$. This means that any $u\in K$ must occur in $\W_R$ between the first $b$ and the last $c$.

What if $u>_Rb$ and $\W_R = \dots ab\dots u\dots ca\dots$, but $u\notin K$? Let $v=\beta(\sigma^{-1}u)$ be the relabelled image of $u$ in $P$. This implies that in the quotient word $\W_P$, $v$ occurs between the  occurrences of $p$. By Axiom~\ref{Quilt def}\eqref{Quilt2}, this implies that $v\vartriangleleft_Pp$. On the other hand, $u\geq_Rb$ implies (by the Definition \ref{Tree extension} of a $\Treecat$-extension) that $v\geq_Pp$. This is a contradiction, therefore $K=\{ u \text{ occurring in }b\dots bc\dots c\mid u\geq_R b\}$. This is $K_3$.

Conversely, if an extension exists, then given $R$ and $K$, we can reconstruct everything else. Clearly, $q=\abs{K}$, $p=\#R+1-q$, and $\sigma\in \Sh_{p-1,q}$ is the unique shuffle that takes $\{p,\dots,p+q-1\}$ to $K$.

First, consider the conditions on $K\subset\com{\#R}$ for it to determine an extension in this way:
\begin{enumerate}
\item
\label{Subtree}
$K$ must be the set of vertices of a subtree. That is, it must be $\leq_R$-convex and contain a $\leq_R$-minimal element.
\item
\label{Sandwich1}
If $\W_R = \dots u\dots v\dots w\dots$ with $u,w\in K$ and $v\notin K$, then $v\vartriangleleft_R u,w$.
\item
\label{Repeat}
For $u,v\in K$, if a $u$ and a $v$ occur in $\W_R$ separated only by vertices not in $K$, then $u=v$.
\end{enumerate}

In the cases at hand, $\W_R$ contains 2 last-first pairs, and these are already accounted for. Suppose that Condition~\ref{Sandwich1} is satisfied and that $u,v\in K$ and $\underline u,\underline v\in\W_R$ are separated only by vertices not in $K$. If $\W_R = \dots \underline u\dots w\dots \underline v\dots w\dots$, then by  Axiom~\ref{Quilt def}\eqref{Quilt2}, $v\vartriangleleft_Rw$, but by the second condition $v \vartriangleright_Rw$. This is a contradiction, therefore any vertex between $\underline u$ and $\underline v$ cannot appear again after $\underline v$. Likewise, it cannot appear before $\underline u$, either. Since there can be no more last-first pairs, $\underline u$ cannot be last, and $\underline v$ cannot be first. This contradicts the no interlacing condition of Definition~\ref{Wordcat def}, unless $u=v$. Therefore, Condition~\ref{Sandwich1} implies Condition~\ref{Repeat}.

If $u\in K_1$ and $v<_R u$, then by Axiom~\ref{Quilt def}\eqref{Quilt1}, $\W_R\neq \dots u\dots v\dots$, so $v$ must occur before $u$, and so $v\in K_1$. This shows that $K_1$ is the set of vertices of a subtree, with $a$ as the root, so Condition~\ref{Subtree} is satisfied. Condition~\ref{Sandwich1} is vacuously true, because all vertices in $K_1$ appear before all the other vertices in $\W_R$.

By construction, $K_2$ is the set of vertices of the maximal subtree with root $b$, so Condition~\ref{Subtree} is satisfied. Now suppose that $\W_R = \dots u\dots v\dots w\dots$ with $u,w\geq_Rb$ and $v\not\geq_Rb$. Clearly, $v\neq a$, which is the only vertex $<_Rb$, and since $b$ is $\vartriangleright_R$-maximal, $v\vartriangleright_Rb$. By Axiom~\ref{Tree}\eqref{Inheritance}, this proves Condition~\ref{Sandwich1}.

If $u\in K_3$ and $b\leq_Rv\leq_Ru$, then by Axiom~\ref{Quilt def}\eqref{Quilt1}, $\W_R = \dots b\dots v\dots u\dots c\dots$, so $v\in K_3$. This shows that $K_3$ is the set of vertices of a subtree with root $b$, and Condition~\ref{Subtree} is satisfied. 
Suppose that $\W_R = \dots u\dots v\dots w\dots$ with $u,w\in K_3$ and $v\notin K_3$. This implies that $v$ occurs between $b$ and $c$, so the only way to not be in $K_3$ is  $v\ngeq_Rb$. Since $v$ occurs after $b$, we also have $v\nleq_Rb$. If $\W_R=\dots b\dots v\dots b\dots$, then $v\vartriangleleft_Rb$, so by Axiom~\ref{Tree}\eqref{Inheritance}, $v\vartriangleleft_Rb\leq_Ru,w$ implies $v\vartriangleleft_Ru,w$. If $\W_R=\dots c\dots v\dots c\dots$, then $v\vartriangleleft_Rc$, hence $v\not\vartriangleright_Rc$, and contrapositively by Axiom~\ref{Tree}\eqref{Inheritance}, $v\not\vartriangleright_Rb$, which eliminates any possibility but $v\vartriangleleft_Rb$, which again implies $v\vartriangleleft_Ru,w$.
\end{proof}
\begin{remark}
If $c\ngtr_Rb$, then $K_2=K_3=\{b\}$ and only determines a trivial extension by $\fbox1$.
\end{remark}

\begin{definition}
\label{Max sign}
For $Q\in\Qcat$, $\sgn_\homone(Q)$  is $(-1)^{1+n(n-1)/2}$ times the sign of the permutation from the labelled (numerical) order of $\com{\#Q}$ to $\downarrow$ order. Equivalently, this is minus the sign of the permutation from labelled order to reverse $\downarrow$ order.

Let
\[
L^0_n := \sum_{\substack{Q\in \Qcat(n)\\ \deg Q = n-2}} \sgn_\homone(Q)\, Q \ .
\] 
\end{definition}

\begin{theorem}
\label{L-infty thm1}
For all $n\geq2$,
\beq\label{L-infinity check}
0 = \partial L^0_n + \sum_{\substack{p,q\geq2\\ p+q=n+1}}\sum_{\sigma\in\Sh_{p-1,q}} (-1)^{(p-1)q}\sgn\sigma\,\left(L^0_p\circ_pL^0_q\right)^{\sigma^{-1}} ,
\eeq
therefore there exists a unique dg-operad homomorphism
$\homone:\Linfty \to \Qop$, such that $\homone(\ell_n) =  L^0_n$.
\end{theorem}
\begin{proof}
Trivially, $L^0_n$ has the same arity, degree, and antisymmetry as $\ell_n$, so eq.~\eqref{L-infinity check} is all that  needs  to be checked. 

Note that a term of eq.~\eqref{L-infinity check} will be either of the form
\[
\sgn_\homone(Q) \sgn_Q(\underline a)\, \partial_{\underline a}Q
\]
or
\[
(-1)^{(p-1)q}\sgn\sigma\, \sgn_\homone(P)\, \sgn_\homone(Q)\, \sgn_{Q,P,p}(R^\sigma)\,R
\]
where $P$ and $Q$ are maximal quilts, $p=\#P$, $q=\#Q$, $\sigma\in\Sh_{p-1,q}$, and $R^\sigma\in\Ext(P,Q,p)$.
In particular, if there is a term proportional to $R$, then  $\deg R = \#R-3$. 
We shall see that each such quilt appears exactly twice in eq.~\eqref{L-infinity check}.

So, suppose that $R\in\Qcat$ with $\deg R = \#R-3$. Because of the antisymmetry of $L^0_n$, we can assume without loss of generality that $R$ is labelled in reverse $\downarrow$ order; this simplifies signs.

There are 2 possibilities from Lemma~\ref{Submaximal}:
\begin{enumerate}
\item $\W_R = ab\dots bc\dots c$. 

By Lemma~\ref{Extensions}, $R$ is an extension up to permutation determined by $K_1$. Denote the structures associated to this with a subscript $1$. Let $p_1:=\#P_1$ and $q_1:=\#Q_1$.

The assumption of reverse $\downarrow$ order implies that $\sigma_1=\id$, and that $P_1$ and $Q_1$ are labelled in reverse $\downarrow$ order, so $\sgn\sigma_1=1$, $\sgn_\homone P_1=\sgn_\homone Q_1 = -1$.

All vertices of $R$ are interposed, except $c=p_1-1$, $b=n-1$, and $a=n$. Each vertex $u\leq p_1-2$ of $P_1$ is interposed and identified (Def.~\ref{Ext sign}) with the $R$-vertex of the same name. Each vertex $v\leq q_1-2$ of $Q_1$ is interposed and identified with the $R$-vertex $v+p_1-1$. This shows that $\sgn_{P_1,Q_1,p_1}(R) = (-1)^{p_1q_1}$, which is the sign of shuffling the $p_1-2$ interposed vertices of $P_1$ past the $q_1-2$ interposed vertices of $Q_1$. Therefore, the coefficient of $R$ in the term coming from this extension is  $(-1)^{q_1}$.

In this case, $b$ appears before $c$ in $\W_R$, so $b\ngtr_R c$. This leaves 3 possible subcases:
\begin{enumerate}
\item   $b <_R c$. 

We cannot add an extra vertex inside the subwords, since they are of maximal length already. By Axiom~\ref{Quilt def}\eqref{Quilt1}, we cannot add a $c$ before $b$ or a $b$ after $c$. Therefore $R$ is not a face.

By Lemma~\ref{Extensions}, $R$ is an extension up to permutation in a second way, determined by $K_2$. Denote structures associated to this with a subscript $2$. Let $p_2:=\#P_2$ and $q_2:=\#Q_2$.

By definition, $\sigma_2\in\Sh_{p_2-1,q_2}$ is the shuffle that maps $\{p_2,\dots,n\}$ to $K_2$. Since $a=n\notin K_2$, $\sigma_2(p_2-1)=n$; 
since $b=n-1$ and $c=p_1-1\in K_2$ are the $\downarrow$ first vertices in $K_2$, $\sigma_2(n)=n-1$ and $\sigma_2(n-1)=p_1-1$. These are the only vertices that aren't interposed.

$P_2$ is labelled in $\downarrow$ order, except that $p_2$ and $p_2-1$ are transposed. Therefore, $\sgn_\homone P_2=+1$.
All vertices of $P_2$ are interposed, except $p_2-1$ and $p_2$. An interposed vertex $u\in\com{p_2-2}$ is identified with the eponymous interposed vertex $u$ of $R^{\sigma_2}$.

$Q_2$ is labelled in $\downarrow$ order, therefore $\sgn_\homone Q_2 = -1$. 
All vertices of $Q_2$ are interposed, except $q_2-1$ and $q_2$. An interposed vertex $v\in\com{q_2-2}$ is identified with $v+p_2-1$ in $R^{\sigma_2}$.

Consider the following sequence of words and permutations.
\begin{itemize}
\item
Begin with the descending sequence of numbers from $n$ to $1$. 
\item
Shuffle $p_2-1$ to the beginning, past $q_2$ terms. The sign is $(-1)^{q_2}$.
\item
Shuffle this to
\[
(p_2-1)\,n\,(n-1)\,(p_2-2)\dots1\,(n-2)\dots p_2 .
\]
The sign is $(-1)^{p_2q_2}$.
\item
Note that the  subword  $(p_2-1)\dots 1$  is identified with the interposed vertices of $P_2$ in $\downarrow$ order. The  subword $(n-2)\dots p_2$ is identified with the interposed vertices of $Q_2$ in $\downarrow$ order.  Shuffle this to $(p_2-1)\,n\,(n-1)$ and then the interposed vertices of $R^{\sigma_2}$ in $\downarrow$ order. The sign is by definition $\sgn_{P_2,Q_2,p_2}(R^{\sigma_2})$.
\item
Apply $\sigma_2$. This gives
\[
n\,(n-1)\,(p_1-1)\,(n-2)\dots p_1\, (p_1-2)\dots 1 .
\]
The sign is of course $\sgn\sigma_2$.
\item
Shuffle back to the initial sequence (descending order). This moves the vertex $c=p_1-1$ past $q_1-2$ terms, so the sign is $(-1)^{q_1}$.
\end{itemize}
This shows that $\sgn\sigma_2\,\sgn_{P_2,Q_2,p_2}(R^{\sigma_2}) = (-1)^{p_2(q_2-1)+q_1}$, and so the coefficient of $R$ coming from this extension is $-(-1)^{q_1}$ and this cancels the term coming from the  extension determined by $K_1$.

\item  $b \vartriangleleft_R c$. 

$R= \partial_{\underline c}R'$, where $\Tree_{R'}=\Tree_R$ and $\W_{R'} = a\underline c b\dots bc\dots c$.

Since $\underline c$ is the first caesura in $\W_{R'}$, $\sgn_{R'}\underline c=-1$. 

The $\downarrow$ order for $R'$ is almost the reverse labelled order, except that $c=p_1-1$ has been moved to second place, past $n-p_1=q_1-1$ letters,
so $\sgn_\homone R' = (-1)^{q_1}$. Therefore, the coefficient of $R$ in the term coming from this face is $-(-1)^{q_1}$. This cancels with the term coming from the extension determined by $K_1$.

\item $b\vartriangleright_R c$. 

In this case, $R=\partial_{\underline b}R'$, where $\Tree_{R'}=\Tree_R$ and $\W_{R'}=ab\dots b c\dots c\underline b$.

$R'$ is labelled in reverse $\downarrow$ order, so $\sgn_\homone R'=-1$. 

All vertices of $R'$ are interposed, except $a$ and $b$. The subword $ab\dots b c$ contains $q_1+1$ distinct vertices, and thus $q_1-1$ interposed vertices. Since every interposed vertex first occurs after a caesura, there are $q_1-1$ caesurae before the first $c$. This shows that 
the numbering of $\W_{R'}$ (as in Def.~\ref{Face sign}) is in part
\[
\W_{R'} = a\,\underset{1}b\dots \underset{q_1-1}b\,c\dots c\,\underset{q_1}{\underline b} ,
\]
so $\sgn_{R'}\underline b = (-1)^q$. 
Therefore, the coefficient of $R$ in the term coming from this face is $-(-1)^{q_1}$, and this cancels with  the term coming from the extension determined by $K_1$.
\end{enumerate}
\item
 $\W_R = \dots ab\dots bc\dots ca\dots$. 
 
In this case, $R = \partial_{\underline a}R'$, where $\Tree_{R'}=\Tree_R$ and 
\[
\W_{R'} = \dots ab\dots b\underline a c\dots ca\dots \ .
\]
The $\downarrow$ order is the same for $R$ and $R'$, so $\sgn_\homone R'=-1$.

Suppose that the numbering of $\W_{R'}$ is in part,
\[
\W_{R'} = \dots \underset{m}a\,b\dots b\,\underset{l}{\underline a} \,c\dots c\,a\dots \ ,
\]
so $\sgn_{R'}\underline a = (-1)^l$, and the coefficient of $R$ as a face of  $R'$ is $-(-1)^l$.

This numbering means that $\underline a$ is the $l$'th caesura, so $c$ is the $l$'th interposed vertex in $\downarrow$ order and the $l+2$'nd vertex of all. 
 Hence, $c=n-l-1$.

Since the $a$ preceding the first $b$ is the $m$'th caesura, $b$ is the $m$'th interposed vertex in $\downarrow$ order and the $m+2$'nd vertex overall. 

This also implies that there are $l-m-1$ caesurae in the subword $b\dots b$ here. Each  vertex in this subword first appears directly after a caesura, except for $b$ itself, and each caesura is directly before a first appearance, therefore the number of vertices in $b\dots b$ is $l-m$.

Since $b$ appears before $c$ in $\W_R$, $b\ngtr_R c$. This leaves 3 possibilities:
\begin{enumerate}
\item $b<_Rc$. 

$R$ is not a face of any other quilt.

By Lemma~\ref{Extensions}, $R$ is an extension up to permutation determined by $K_3$.

All vertices of $R$ are interposed, except $n$, $n-1$, and $c$. The $\downarrow$ first vertices in $K_3$ are $b$ and $c$, and $\alpha(v) = v+p-1$, therefore $\sigma(n) = b$ and $\sigma(n-1)=c$. The root of $R$ is $n\notin K_3$, so $\sigma(p-1)=n$. Likewise, because $n-1$ is not interposed, it is not in $K_3$, therefore $\sigma(p-2)=n-1$.

This shows that all vertices of $R^\sigma$ are interposed except $p-1$, $p-2$, and $n-1$.

$Q$ is labelled in reverse $\downarrow$ order, so $\sgn_\homone Q=-1$.

The interposed vertices of $Q$ are $1$ through $q-2$. An interposed vertex $v$ is identified with the interposed vertex $\alpha(v) = v+p-1$ of $R^\sigma$.

The interposed vertices of $P$ are $1$ through $p-3$ and $p$. An interposed vertex $u\leq p-3$ is identified with the eponymous vertex of $R^\sigma$, and $p$ must be identified with $n$, since that is the only remaining interposed vertex of $R^\sigma$.

Consider the following sequence of words and permutations.
\begin{itemize}
\item
Begin with the descending sequence of numbers from $n$ down to $1$.
\item
Shuffle this to 
\[
(p-1)\,(p-2)\,(n-1)\,n\,(n-2)\dots p\,(p-3)\dots1 .
\]
The sign of this shuffle is $-1$.
\item
Shuffle $n-2$ through $p$ ($q-2$ vertices) past $p-3$ through $1$ ($p-3$ vertices) giving
\[
(p-1)\,(p-2)\,(n-1)\,n\,(p-3)\dots1\,(n-2)\dots p .
\]
The sign is $(-1)^{(p-1)q}$.
\item
The subword $n\,(p-3)\dots1$ consists of the $R^\sigma$-vertices  corresponding to interposed vertices of $P$, in descending labelled order. Shuffle these to  $\downarrow$ order for $R^\sigma$ (which corresponds to $\downarrow$ order for $P$). 
This gives $(p-1)\,(p-2)\,(n-1)$ and then the vertices corresponding to interposed vertices of $P$ and then $Q$, each in $\downarrow$ order.
The sign is $-\sgn_\homone P$, by definition.
\item
The vertices after the first 3 are the interposed vertices of $R^\sigma$. Shuffle these to $\downarrow$ order for $R^\sigma$. The sign is (by definition) $\sgn_{P,Q,p}(R^\sigma)$.
\item
Apply $\sigma$. This gives
\[
n\,(n-1)\,c\,(n-2)\dots(c+1)\,(c-1)\dots1 .
\]
The sign is (of course) $\sgn\sigma$.
\item
The only vertex out of position is $c$. Shuffle this back to the initial word with everything in descending labelled order. This moves $c$ past $n-2-c=l-1$ vertices, so the sign is $-(-1)^l$. 
\end{itemize}
This shows that $\sgn_\homone(P)\,\sgn(\sigma)\,\sgn_{P,Q,p}(R^\sigma) = -(-1)^{(p-1)q+l}$. Recalling that $\sgn_\homone Q=-1$, this shows that the coefficient of $R$ as an extension is $(-1)^l$, and this term cancels the term from $R$ as a face of $R'$.

\item $b\vartriangleleft_Rc$. 

$R=\partial_{\underline c} R''$, where $\W_{R''} = \dots a\underline c b\dots bc\dots ca\dots$ and $\Tree_{R''}=\Tree_R$.

Lemma~\ref{Extensions} shows that $R$ is not an extension up to permutation. 

The numbering of $\W_{R''}$ is in part
\[
\W_{R''} = \dots \underset{m}a\underset{m+1}{\underline c} b\dots bc\dots ca\dots ,
\]
so $\sgn_{R''}\underline c = - (-1)^m$.

The $\downarrow$ ordering is $\dots a\dots b\dots c\dots$ in both $R$ and $R'$, but for $R''$ it is $\dots a\dots cb\dots$. The  $c$ has been shuffled to just before $b$. This moves it past $l-m$ vertices occurring in the subword $b\dots b$.  This shows that $\sgn_\homone R'' = -(-1)^{l-m}$. The coefficient of  $R$ as a face of $R''$ is  $(-1)^l$. This cancels $R$ as a face of $R'$.

\item
 $b\vartriangleright_Rc$. 
 
$R=\partial_{\underline b} R''$, where $\W_{R''} = \dots ab\dots bc\dots c\underline b a\dots$.

Lemma~\ref{Extensions} shows that $R$ is not an extension. 

The numbering of $\W_{R''}$ is in part
\[
\W_{R''} = \dots \underset{m}a b\dots \underset{l}bc\dots c\underset{l+1}{\underline b} a\dots ,
\]
so $\sgn_{R''}\underline b = -(-1)^l$. In this case, the $\downarrow$ ordering for $R''$ is the same as for $R$ and $R'$, so $\sgn_\homone R'' = -1$. The coefficient of  $R$ as a face of $\partial R''$ is $(-1)^l$, so it cancels $R$ as a face of $R'$.
\qedhere
\end{enumerate}
\end{enumerate}
\end{proof}

\begin{definition}
\label{P0 def}
$P^0_n$ is $(-1)^{1+n(n-1)/2}$ times the sum of all quilts $Q\in\Qcat(n)$ with $\deg Q=n-2$ that are labelled in $\downarrow$ order. 
\end{definition}
The proof of Theorem~\ref{L-infty thm1} actually more directly proves the following result, which is equivalent because $\Qop$ is torsion free (as an abelian group).
\begin{corollary}
\label{Coinvariant version}
For all $n\geq 2$,
\[
\partial P^0_n  + \sum_{\substack{p,q\geq2\\ p+q=n+1}}\sum_{j=1}^p (-1)^{(p-1)q+(p-j)(q-1)}\,P^0_p\circ_jP^0_q
\]
is contained in the abelian subgroup spanned by elements of the form $Q-(\sgn \sigma)Q^\sigma$ for $Q\in\Qcat(n)$ and $\sigma\in\perm_n$. In other words, its image is $0$ in the coinvariant group, $\Qop(n)_{\perm_n}$.
\end{corollary}

\subsection{$\Linfty$ in $\mQuilt$}
The next step is to construct a homomorphism from $\Linfty$ to $\mQuilt$. This will give the Hochschild bicomplex the structure of an $\Linfty$-algebra.

\begin{definition}
$L^1_n := L^0_{n+1}\circ_1\m$, $L^2_n := (L^0_{n+2}\circ_1\m)\circ_1\m = L^0_{n+2}\circ(\m,\m,\id,\dots)$, \emph{et cetera}, and $L_n=L^0_n+L^1_n\in\mQuilt(n)$.
\end{definition}

Note that $L_2$ and $L_3$ appeared already in Section~\ref{Gerstenhaber}, and $L^0_1 = \Delta$ from eq.~\eqref{Delta def}.

\begin{lemma}
\label{Vanishing terms}
If $n\geq3$, $Q\in \Qcat(n)$ is a maximal quilt, and $a\in\com{n}$ is not the root of $Q$, then $Q\circ_a\m=0$.

For any $n\geq0$, $L^2_n=0$.
\end{lemma}
\begin{proof}
By Relation \ref{mQuilt def}\eqref{flat1 axiom}, if $a$ is repeated in $\W_Q$, then $Q\circ_a\m=0$, so suppose that $a$ is not repeated.
By Lemma~\ref{Maximal}, the last vertex in $\W_Q$ has an earlier occurrence, so $a$ is not the last vertex. Let $\underline v$ be the next letter after $a$. 

By Lemmata \ref{Word length} and \ref{Maximal}, the first 2 letters of $\W_Q$ are the only last-first pair. Since $a$ is not repeated, it is a last occurrence, therefore $\underline v$ cannot be a first occurrence. The subword between the previous $v$ and $\underline v$ cannot contain a last-first pair or $v$, therefore it must begin and end with the same vertex, which is $a$. Since $a$ is not repeated,  $\W_Q = \dots v\, a\,\underline v\dots$. By  Relation \ref{mQuilt def}\eqref{flat2 axiom} this implies that $Q\circ_a\m=0$.

$L^0_{n+2}$ is an alternating sum of maximal quilts. For $n\geq1$, and $Q\in\Qcat(n)$ maximal, $Q\circ(\m,\m,\id,\dots)= 0$, because $1$ and $2$ can't both be the root of $Q$, therefore
\[
L^2_n = L^0_{n+2}\circ(\m,\m,\id,\dots)= 0.
\]
For $n=0$, Relation~\ref{mQuilt def}\eqref{MC axiom} directly implies that $L^2_0=0$.
\end{proof}

The homomorphism $\homone$ defines a Maurer-Cartan equation in any $\Qop$-algebra over a field of characteristic $0$. The additional generator $\m$ is formally a solution of this Maurer-Cartan equation. The first term of the equation is $\partial\m=0$. The second term vanishes by Relation~\ref{mQuilt def}\eqref{MC axiom}. All other terms vanish by Lemma~\ref{Vanishing terms}.

A solution of the Maurer-Cartan equation for some $\Linfty$-algebra determines another $\Linfty$-algebra structure on the same graded vector space. This fact seems to be generally known but rarely written down. It appears as Lemma 4.4 in \cite{vdl2002}. In this spirit, the formal solution $\m$ determines from $\homone$ another homomorphism from $\Linfty$ to $\mQuilt$. 

This is also what changes the differential from $\partial$ to $\partial' = \partial + \ad_\Delta$. In principle, $\Delta$ should have further terms with more $\m$'s, but these vanish by Lemma~\ref{Vanishing terms}.

\begin{theorem}
\label{Linfty thm2}
For all $n\geq2$,
\beq\label{L-infinity check2}
0 = \partial' L_n + \sum_{\substack{p,q\geq2\\ p+q=n+1}}\sum_{\sigma\in\Sh_{p-1,q}} (-1)^{(p-1)q}\sgn\sigma\,\left(L_p\circ_pL_q\right)^{\sigma^{-1}} ,
\eeq
therefore there exists a unique dg-operad homomorphism
$\homtwo:\Linfty \to \mQuilt$, such that $\homtwo(\ell_n) =  L_n$.
\end{theorem}
\begin{proof}
The \emph{idea} is to write $L_n = L_n^0+L_n^1+\frac12L_n^2$ --- since $L_n^2=0$ --- and then compute $\partial L_n$. However, we can't actually divide by $2$.

Firstly,  compute $-\partial L^1_n$ and $-\partial L^2_n$ by composing eq.~\eqref{L-infinity check} with $(\m,\id,\dots)$ or $(\m,\m,\id,\dots)$.

Consider a term of $-\partial L^k_n$ involving $L^j_p\circ_p L^{k-j}_q$ (with $p+q=n+1$). This comes from composing $L^0_{j+p}\circ_{j+p}L^0_{k-j+q}$ with $k$ $\m$'s. The shuffle $\sigma\in\Sh_{p+j-1,q+k-j}$ factorizes uniquely as 
\[
\sigma = (\sigma_1\smile\sigma_2)\circ\chi
\]
where $\sigma_1\in\Sh_{j,k-j}$, $\sigma_2\in\Sh_{p-1,q}$, $\smile$ means concatenation (placing the permutations side-by-side) and  $\chi\in\Sh_{j,p-1,k-j,q}$ is the unique shuffle that swaps the $p-1$ letters with the $k-j$ letters. The last shuffle has $\sgn\chi = (-1)^{(p-1)(k-j)}$. There is another factor of $(-1)^{(k-j+q)j}$ from moving $j$ $\m$'s past $L^0_{k-j+q}$. The sign of $\sigma_1$ cancels with the effect of $\m$ being odd, so different values of $\sigma_1$ give equal contributions, and there are $\abs{\Sh_{j,k-j}} = \binom{k}{j}$ of these. 
The sign of this term is thus 
\[
(-1)^{(j+p-1)(k-j+q)}\sgn \sigma = (-1)^{(p-1)q}\sgn\sigma_2 .
\]
Renaming $\sigma_2$ as $\sigma$, all this gives
\beq
\label{dLnk}
-\partial L^k_n = \sum_{j=0}^k\sum_{ p+q=n+1} \sum_{\sigma\in\Sh_{p-1,q}} \binom{k}{j}(-1)^{(p-1)q}\sgn\sigma\,\left(L^j_p\circ_p L^{k-j}_q\right)^{\sigma^{-1}} .
\eeq

Observe that for $k=2$, the only  term that doesn't obviously vanish is with $j=1$.  This gives 
\[
0 = -\partial L_n^2 = 2 \sum_{ p+q=n+1} \sum_{\sigma\in\Sh_{p-1,q}} (-1)^{(p-1)q}\sgn\sigma\,\left(L^1_p\circ_p L^1_q\right)^{\sigma^{-1}} .
\]
To avoid the factor of $2$ requires a slightly messier approach. First consider 
\[
0 = \sum_{\sigma\in\perm_n}\sum_{1\leq i<j\leq n+2} -\sgn\sigma\,\left[(P_{n+2}^0\circ_i\m)\circ_{j-1}\m\right]^\sigma .
\]
Computing $\partial$ of this by Corollary~\ref{Coinvariant version} gives 
\beq
\label{without 2}
0 = \sum_{ p+q=n+1} \sum_{\sigma\in\Sh_{p-1,q}} (-1)^{(p-1)q}\sgn\sigma\,\left(L^1_p\circ_p L^1_q\right)^{\sigma^{-1}} .
\eeq

Adding up instances of  eqs.~\eqref{dLnk} and \eqref{without 2} gives
\begin{align}
-\partial L_n &= -\partial L^0_n - \partial L^1_n  \nonumber\\
&= \sum_{ p+q=n+1} \sum_{\sigma\in\Sh_{p-1,q}} (-1)^{(p-1)q}\sgn\sigma\,\left(L_p\circ_pL_q\right)^{\sigma^{-1}} .
\label{L expansion}
\end{align}

The $p=1$, $q=n$ part of \eqref{L expansion} is
\[
\sum_{\sigma\in\Sh_{0,n}} \sgn\sigma\, (L_1\circ_1 L_n)^{\sigma^{-1}}
= L_1\circ_1 L_n = \Delta \circ_1 L_n 
\]
since $\Sh_{0,n}$ is trivial and $L_1=\Delta$.

The $p=n$, $q=1$ part of \eqref{L expansion} is
\beq
\label{last L}
\sum_{\sigma\in\Sh_{n-1,1}} (-1)^{n-1}\sgn\sigma\,(L_n\circ_n L_1)^{\sigma^{-1}} .
\eeq
Note that $\Sh_{n-1,1}$ consists of cyclic permutations of the form $(a\dots n)$, so for $\sigma = (a\dots n)$,
\[
(L_n\circ_n L_1)^{\sigma^{-1}} = L_n^{\sigma^{-1}} \circ_a L_1
= (\sgn\sigma) L_n \circ_a L_1 .
\]
This simplifies \eqref{last L} to
\[
(-1)^{n-1}\sum_{a=1}^n L_n\circ_a L_1 = -(-1)^{\deg L_n} L_n\circ_a \Delta ,
\]
since $\deg L_n = n-2$.
Together, these terms add up to $\ad_\Delta L_n$ by eq.~\eqref{ad}. Combining this with $\partial L_n$ gives $\partial' L_n$ and proves eq.~\eqref{L-infinity check2}.
\end{proof}

\subsection{Maurer-Cartan}
\label{Maurer-Cartan}
For any diagram of algebras $\A:\Xc\to\Alg_\Bbbk$,  
the  Hochschild complex $\s\tot C^{\bullet,\bullet}(\A,\A)$ is an $\Linfty$-algebra under the action $\Rep\circ\homtwo$. 
The reduced asimplicial subcomplex, $\s\tot\bar C_{\mathrm a}^{\bullet,\bullet}(\A,\A)$, is the $\Linfty$-subalgebra governing deformations of diagrams of algebras.

Specifically, in characteristic $\Char\Bbbk=0$, the Maurer-Cartan equation for $f\in C^{\bullet,\bullet}(\A,\A)$ is
\[
0 = \delta f + \sum_{n=2}^\infty \frac1{n!} \widehat{L_n}(\underbrace{f,\dots,f}_{n\text{ times}}) .
\]
This doesn't make sense in finite characteristic.

Because the definition of $L_n$ involves antisymmetrization, it is useful to write down an unsymmetrized version, which has $\frac1{n!}$ as many terms. 
\begin{definition}
\label{P def}
$P_n := P^0_n +  P^0_{n+1}\circ_1\m$ (where $P^0_n$ was defined in Def.~\ref{P0 def}).
\end{definition}
With this,
\[
L_n = \sum_{\sigma\in\perm_n} \sgn\sigma\, (P_n)^\sigma ,
\]
thanks to Lemma~\ref{Vanishing terms}.

With this definition,
\begin{align*}
P_2 &= \Qtwo12 + \Qthreea\m21, &
P_3 &= \Qthreea132 - \Qfoura\m213 ,
\end{align*}
\begin{multline*}
-P_4 = \Qfourd1432 + \Qsixb1\blank324\blank + \Qsixb13\blank2\blank4+  \Qfoura1324 \\
+ \Qfivea\m2134 + \Qfiveb\m2143 + \Qseven\m21\blank34\blank + \Qseven\m213\blank\blank4 .
\end{multline*}
\begin{remark}
Note that $P_3$ differs from $P'_3$ (Def.~\ref{P3}) by a permutation of the second term. 
\end{remark}

With this, the Maurer-Cartan equation can be rewritten as 
\beq
\label{MC finite characteristic}
0 = \delta f + \sum_{n=2}^\infty \widehat{P_n}(f,\dots,f)
\eeq
which \emph{does} make sense in finite characteristic. 

The infinite sum here isn't really infinite. 
\begin{proposition}
For $f\in C_{\mathrm a}^{\bullet,\bullet}(\A,\A)$, the Maurer-Cartan equation \eqref{MC finite characteristic} is only cubic:
\[
0 = \delta f + \widehat{P_2}(f,f) + \widehat{P_3}(f,f,f) .
\]
For  $f\in C^{\bullet,\bullet}(\A,\A)$, it is quartic:
\[
0 = \delta f + \widehat{P_2}(f,f) + \widehat{P_3}(f,f,f) + \widehat{P_4}(f,f,f,f) ,
\]
and
\[
\widehat{P_4}(f,f,f,f) = -\widehat{\Qfoura1324}(f,f,f,f) .
\]
\end{proposition}
\begin{proof}
For the Maurer-Cartan equation, $f\in C^{\bullet,\bullet}(\A,\A)$ has $\abs f = 1$. This is shifted from the total of the bidegree, so there can be components in bidegrees $(0,2)$, $(1,1)$ and $(2,0)$.

After the linear term, the right side  of eq.~\eqref{MC finite characteristic} is a sum of terms of the form $\widehat{P_n}(f_1,\dots,f_n)$, where each $f_a\in C^{p_a,q_a}(\A,\A)$ is a component of $f$ (quite possibly repeated). This can be written as a combination of terms of the form $\widehat{P^0_n}(f_1,\dots,f_n)$ if we allow $f_1$ to (possibly) be $\hat\m$. 

This is, in turn, a combination of terms of the form 
\beq
\label{Q term}
\widehat Q(f_1,\dots,f_n),
\eeq
 where $Q$ is a maximal quilt, labelled in $\downarrow$ order. By Lemma~\ref{Maximal}, $\W_Q = 12\dots2$, where $1$ is the root and $2$ is a leaf.
 
Consider a nonvanishing term of the form \eqref{Q term} with $n\geq3$
The root $1$ in $\Tree_Q$ cannot have more than $q_1\leq2$ children. 
 By the same reasoning as in Lemma~\ref{Vanishing terms}, for any $a\neq1$, $f_a$  cannot have bidegree $(0,2)$, therefore $a$ cannot have more than one child in $\Tree_Q$. 
 This shows that $1$ has two children --- namely, $3$ and $2$ --- with the other vertices in a single chain descending from $3$. Since $Q$ is maximal, this means that $\W_Q = 1232\dots 2n2$ (with $2$ repeated $n-1$ times). 
 This implies that $p_2\geq n-2$, since otherwise no coloring exists.
 
In the full complex, $C^{\bullet,\bullet}(\A,\A)$, $p_2\leq 2$, so this implies $n\leq4$, and the Maurer-Cartan equation is quartic. In the asimplicial subcomplex, $C_{\mathrm a}^{\bullet,\bullet}(\A,\A)$, $p_2\leq 1$, so this implies $n\leq3$, and the Maurer-Cartan equation is cubic. 
\end{proof}

The Maurer-Cartan equation really consists of 3 or 4 separate equations in different bidegrees.

\begin{definition}
Denote 
\[
\MC(f) :=  \delta f + \widehat{P_2}(f,f) + \widehat{P_3}(f,f,f) + \widehat{P_4}(f,f,f,f) 
\]
and 
\[
\MC_0(f) := \dS f + \widehat{P_2^0}(f,f) + \widehat{P_3^0}(f,f,f) + \widehat{P_4^0}(f,f,f,f) .
\]
Write $\MC^{(p,q)}$ for the component in $C^{p,q}(\A,\A)$ (with $p+q=3$). 
\end{definition}

\begin{definition}
A \emph{deformation} of a diagram of algebras, $\A:\Xc\to\Alg_\Bbbk$ is a cochain $f+g$ with $f\in C^{0,2}(\A,\A)$ and $g\in C^{1,1}(\A,\A)$ such that $\hat\m+f$ and $\A+g$ define a diagram of algebras. That is, for each $x\in B_0\Xc$, $\hat\m(x)+f(x)$ is an associative product on $\A(x)$, for each $\phi\in B_1\Xc$, $\A[\phi]+g[\phi]$ is a homomorphism of the resulting algebras, and this defines a functor.
\end{definition}
\begin{theorem}
\label{MC thm1}
For a diagram of algebras $\A:\Xc\to\Alg_\Bbbk$, the solutions of the Maurer-Cartan equation in the reduced asimplicial subcomplex $\bar C_{\mathrm a}^{\bullet,\bullet}(\A,\A)$ are the  deformations of $\A$ as a diagram of algebras.
\end{theorem}
\begin{proof}
First, note that adding the identity $0=\MC_0(\hat\m)$ gives $\MC(f) = \MC_0(\hat\m+f)$.

In this case, an element with shifted total degree $1$ can have 2 components. Let $f\in C^{0,2}(\A,\A)$ and $g\in C^{1,1}(\A,\A)$. The first piece of the Maurer-Cartan equation is
\[
0 = \MC^{(0,3)}(f+g) = \widehat{\Qtwo12}(\hat\m+f,\hat\m+f) .
\]
Evaluating at any $x\in B_0\Xc = \Obj\Xc$, this is
\[
0 =  (\hat\m[x]+f[x])\circ_1(\hat\m[x]+f[x]) - (\hat\m[x]+f[x])\circ_2 (\hat\m[x]+f[x]) ,
\]
which is precisely the condition that  $\hat\m[x]+f[x]$ be an associative product on the vector space $\A(x)$.

Next,
\[
\MC^{(1,2)}(f+g)= \dS(\hat\m+f) + \widehat{\Qtwo12}(\hat\m+f,g) + \widehat{\Qtwo12}(g,\hat\m+f) + \widehat{\Qthreea132}(\hat\m+f,g,g) .
\]
Evaluating this at $\phi:x\to y$ in $B_1\Xc$ gives
\[
\MC^{(1,2)}(f+g)[\phi] = (\A[\phi]+g[\phi])\circ (\hat\m[x]+f[x]) - (\hat\m[y]+f[y]) \circ (\A[\phi]+g[\phi])^{\otimes 2} .
\]
Setting this to $0$ is precisely the condition that $\A[\phi]+g[\phi]$ be a homomorphism between these associative products.

The ``reduced'' condition is simply that for any $x\in B_0\Xc$, $g[\id_x]=0$, so $\A[\id_x]+g[\id_x]=\id_{\A(x)}$.

Finally, 
\[
\MC^{(2,1)}(f+g) = \dS g + \widehat{\Qtwo12}(g,g) .
\]
Evaluating this at $(\psi,\phi)\in B_2\Xc$, and using the fact that $\A$ is a functor, gives
\[
\MC^{(2,1)}(f+g)[\psi,\phi] = (\A[\psi]+g[\psi])\circ (\A[\phi]+g[\phi]) - (A[\psi\circ\phi]+g[\psi\circ\phi]) .
\]
Setting this to $0$ is the condition that $\A+g$ be a functor.

This shows that a solution of the Maurer-Cartan equation in $C_{\mathrm a}^{\bullet,\bullet}(\A,\A)$  is a deformation of the diagram of algebras $\A$. The deformed diagram has multiplication given by $\hat\m+f$ and the morphism part of the functor is $\A+g$.
\end{proof}

\begin{definition}
\label{Skew}
A \emph{skew diagram of algebras} $(\A,u)$ over a category $\Xc$ consists of
\begin{itemize}
\item
for every $x\in B_0\Xc$, an associative algebra $\A(x)$,
\item
for every $\phi:x\to y$ in $\Xc$, a homomorphism $\A[\phi]:\A(x)\to\A(y)$, and
\item
for every $z\stackrel\psi\leftarrow y\stackrel\phi\leftarrow x$, an  element of the unitalization $u(\psi,\phi)\in 1+\A(z)$,
\end{itemize}
such that:
\begin{itemize}
\item
For any $z\stackrel\psi\leftarrow y\stackrel\phi\leftarrow x$ and $a\in\A(x)$,
\beq
\label{skew functor}
\A(\psi;\A(\phi;a))\cdot u(\psi,\phi) = u(\psi,\phi)\cdot\A(\psi\circ\phi;a) ;
\eeq
\item
for any $(\chi,\psi,\phi)\in B_3\Xc$, 
\beq
\label{u cocycle}
\A(\chi;u(\psi,\phi))\cdot u(\chi,\psi\circ\phi) = u(\chi,\psi)\cdot u(\chi\circ\psi,\phi) .
\eeq
\end{itemize}
\end{definition}
This is equivalent to a twisted precosheaf of algebras \cite[Def.~2.15]{d-l-l2017}.

In particular, a diagram of algebras, $\A$, is trivially a skew diagram of algebras $(\A,1)$. A diagram of algebras can be deformed as a skew diagram.

\begin{theorem}
\label{MC thm2}
For a diagram of algebras $\A:\Xc\to\Alg_\Bbbk$, the solutions of the Maurer-Cartan equation in $C^{\bullet,\bullet}(\A,\A)$ are the deformations of $\A$ as a skew diagram of algebras.
\end{theorem}
\begin{proof}
In this case, an element with shifted total degree $1$ can have 3 components. Let $f\in C^{0,2}(\A,\A)$, $g\in C^{1,1}(\A,\A)$ and $h\in C^{2,0}(\A,\A)$, and consider $\MC(f+g+h)=\MC_0(\hat\m+f+g+h)$.

The term $h$ does not contribute to the bidegree $(0,3)$ and $(1,2)$ components of the Maurer-Cartan equation, so as in Theorem~\ref{MC thm1} those remain the associativity and homomorphism conditions.

Because $\hat\m$ and $f$ only appear in the combination $\hat\m+f$, it is sufficient to consider the case that $f=0$. This reduces clutter.

The first component involving $h$ is
\begin{align*}
\MC^{(2,1)}(g+h) &= \MC_0^{(2,1)}(\hat\m+g+h)\\
 &= \dS g + \widehat{\Qtwo12}(g,g) + \widehat{\Qtwo12}(\hat\m,h)  + \widehat{\Qthreea132}(\hat\m,h,g) \\ &\qquad+ \widehat{\Qthreea132}(\hat\m,g,h) - \widehat{\Qfoura1324}(\hat\m,h,g,g) .
\end{align*}
The last 4 terms, evaluated at $(\psi,\phi)\in B_2\Xc$ and applied to an algebra element $a$, give
\[
\A'(\psi;\A'(\phi;a))\cdot h(\psi,\phi) - h(\psi,\phi)\cdot \A'(\psi\circ\phi;a)
\]
where $\A'[\phi] := \A[\phi]+g[\phi]$. This can be combined with the other terms nicely if we use the unitalization of the algebras. Let $1$ be the unit. This gives,
\[
\MC^{(2,1)}(g+h)(\psi,\phi;a) = \A'(\psi;\A'(\phi;a))\cdot [1+h(\psi,\phi)] - [1+h(\psi,\phi)]\cdot \A'(\psi\circ\phi;a) .
\]
Setting this to $0$ is condition \eqref{skew functor} for $(\A',1+h)$.

Finally, 
\begin{align*}
\MC^{(3,0)}(g+h) &= \MC_0^{(3,0)}(\hat\m+g+h) \\
&= \dS h + \widehat{\Qtwo12}(g,h) + \widehat{\Qthreea132}(\hat\m,h,h) - \widehat{\Qfoura1324}(\hat\m,h,g,h) .
\end{align*}
Evaluating this at $(\chi,\psi,\phi)\in B_3\Xc$ gives
\begin{align*}
\MC^{(3,0)}(g+h)(\chi,\psi,\phi) &= \A'(\chi;h(\psi,\phi))-h(\chi\circ\psi,\phi) + h(\chi,\psi\circ\phi) - h(\chi,\psi) \\
&\qquad+\A'(\chi;h(\psi,\phi))\cdot h(\chi,\psi\circ\phi) - h(\chi,\psi)\cdot h(\chi\circ\psi,\phi) \\
&= \A'(\chi;1+h(\psi,\phi))\cdot [1+h(\chi,\psi\circ\phi)] \\
& \qquad - [1+h(\chi,\psi)]\cdot [1+h(\chi\circ\psi,\phi)] .
\end{align*}
Setting this to $0$ is the cocycle condition \eqref{u cocycle} for $(\A',1+h)$.

Together, this shows that if $g+h$ is a solution of the Maurer-Cartan equation, then $(\A',1+h)$ is a skew diagram of algebras. More generally, if $f+g+h$ is a solution of Maurer-Cartan, then the same is true with the algebra multiplication modified to $\hat\m+f$.
\end{proof}
In particular, a solution of the Maurer-Cartan equation in $\bar C^{\bullet,\bullet}(\A,\A)$ gives a skew diagram of algebras that respects the identity morphisms in $\Xc$.

\subsection{One or two algebras}
\label{One or two}
\subsubsection{One}
Let $\onecat$ be the category with a single object, $1$ and only the identity morphism. A diagram $\A:\onecat\to\Alg_\Bbbk$ is equivalent to a single algebra, $A=\A(1)$. The full Hochschild bicomplex is unnecessarily messy in this case, but the reduced complex simplifies to the Hochschild complex of $A$:
\[
\bar C^{p,q}(\A,\A) = 
\begin{cases}
C^q(A,A) & p=0\\
0 & p\geq1 .
\end{cases}
\]
On this reduced complex, all of the higher brackets vanish, $\widehat L_n=0$ for $n\geq3$, and the $\Linfty$-algebra is just the usual differential graded Lie algebra.

\subsubsection{Two} 
Let $\twocat$ be the category with two objects, $1$ and $2$, and one nonidentity morphism. $\gamma:1\to2$. A diagram $\A:\twocat\to\Alg_\Bbbk$ is equivalent to two algebras, $A=\A(1)$ and $B=\A(2)$, and one homomorphism $\A[\gamma]:A\to B$.

The reduced Hochschild bicomplex is relatively simple:
\[
\bar C^{0,\bullet}(\A,\A) \cong C^\bullet(A,A)\oplus C^\bullet(B,B),
\]
\[
\bar C^{1,\bullet}(\A,\A) \cong C^\bullet(A,B),
\]
where the isomorphism is given by evaluation at $\gamma$, and $\bar C^{p,\bullet}(\A,\A)=0$ for $p\geq2$. This is the same complex for a morphism described in \cite{g-s1983}. The binary bracket $\widehat L_2$ is equivalent to the one discussed there, and does not simplify significantly from the general case.

The higher brackets do simplify. The only maximal quilts that have a nonzero action on this complex are those with only 2 rows, thus
\begin{align*}
\widehat{P_3} &= \widehat{\Qthreea132}, &
\widehat{P_4} &= -\widehat{\Qfourd1432}, &
\widehat{P_5} &= -
\widehat{\Qwrap{\begin{array}{|c|c|c|c|}\hline\multicolumn{4}{|c|}{1}\\\hline5&4&3&2\\\hline\end{array}}}, &
P_6 &= 
\widehat{\Qwrap{\begin{array}{|c|c|c|c|c|}\hline\multicolumn{5}{|c|}{1}\\\hline6&5&4&3&2\\\hline\end{array}}},
\end{align*}
and so on. For $n\geq3$ and $f_1,\dots,f_n\in \bar C^{\bullet,\bullet}(\A,\A)$,
\[
\widehat P_n(f_1,\dots,f_n) \in \bar C^{1,\bullet}(\A,\A)
\]
only depends on $f_1[2]$, and $f_2[\gamma],\dots,f_n[\gamma]$. It does not use the multiplication in either algebra.

This is equivalent (up to signs) to the $\Linfty$-algebra described in \cite{f-m-y2009} and presumably to that described in \cite{bor2007}.

\section{Conclusions}
The dg-operad of abelian groups $\Qop$ (Def.~\ref{Quilt def}) is a suboperad of the Hadamard product of the Gerstenhaber-Voronov operad, $\Wordop$, and $\Brace$. The colored operad of sets $\CQ$ (Def.~\ref{CQ def}) is a suboperad of the Hadamard product of $\MultiSemiD$ (constructed from the semisimplex category) and $\NSOp$ (governing non-symmetric operads). 

$\CQ$ acts (Lem.~\ref{rep lemma}) on the underlying vector spaces of the Hochschild bicomplex $C^{\bullet,\bullet}(\A,\A)$ of a diagram of vector spaces, $\A:\Xc\to\Vect_\Bbbk$. A coloring of a quilt is a lift to an element of $\CQ$. I defined an action $\Rep_0$ of $\Qop$ on the Hochschild bicomplex  by a sum over colorings with signs defined by shuffles. In this way, $\s_2C^{\bullet,\bullet}(\A,\A)$ is a $\Qop$-algebra (Thm.~\ref{Rep0 thm}). The asimplicial and reduced subcomplexes are $\Qop$-subalgebras (Thms.~\ref{Asimplicial algebra} and \ref{Reduced thm}).

I defined another dg-operad $\mQuilt$ (Def.~\ref{mQuilt def}) by adjoining to $\Qop$ an extra generator $\m$ that satisfies the properties of multiplication in a diagram of algebras, viewed as a Hochschild cocycle. In this way, the Hochschild complex (with shifted degree) $\s\tot C^{\bullet,\bullet}(\A,\A)$ of a diagram of algebras $\A:\Xc\to\Alg_\Bbbk$ is an $\mQuilt$-algebra (Thm.~\ref{Rep mQuilt}). The asimplicial and reduced subcomplexes are $\mQuilt$-subalgebras (Thm.~\ref{Asimplicial algebra2}).

I constructed representatives in $\mQuilt$ for the generators of $\Susp\Gerst$ and explicit homotopies for their relations, thus showing (Thm.~\ref{Gerstenhaber thm}) that there is a graded operad homomorphism $\Ghom :\Susp\Gerst \to H_\bullet(\mQuilt)$, and hence the Hochschild cohomology and asimplicial cohomology of a diagram of algebras are Gerstenhaber algebras.

I constructed   dg-operad homomorphisms $\homone:\Linfty\to\Qop$ (Thm.~\ref{L-infty thm1}) and  $\homtwo:\Linfty\to\mQuilt$ (Thm.~\ref{Linfty thm2}). These make the Hochschild bicomplex of a diagram of vector spaces or algebras  into an $\Linfty$-algebra. 

This $\Linfty$-algebra structure defines Maurer-Cartan equations in the Hochschild complex.
I showed that the solutions of the Maurer-Cartan equation in the reduced asimplicial subcomplex are the deformations of a diagram of algebras (Thm.~\ref{MC thm1}), and  the solutions of the Maurer-Cartan equation in the full Hochschild bicomplex are the deformations of a diagram of algebras into skew diagrams of algebras (Thm.~\ref{MC thm2}).

These results are actually not limited to diagrams of algebras over a field. They also apply to a diagram of algebras over a commutative ring. More generally, $\A$ could be a functor from $\Xc$ to a nonsymmetric colored operad of abelian groups, along with a multiplication $\hat\m\in C^{0,2}(\A,\A)$.

I have not computed the homology of my dg-operads, $\Qop$ and $\mQuilt$.
If Conjecture~\ref{Acyclic} is true that $\Qop$ is acyclic, then $\Qquotient:\Qop\to\Brace$ induces an isomorphism from $H_\bullet(\Qop)$ to $\Brace$.

The Hochschild cohomology of an algebra is a Gerstenhaber algebra, with no universal identities aside from the defining ones, therefore  $\Ghom:\Susp\Gerst\to H_\bullet(\mQuilt)$ is injective. It seems likely that it is an isomorphism, but it would be very interesting if it is not. A cohomology class not contained in the image of $\Ghom$ would be a new operation on the Hochschild cohomology of a diagram of algebras.


Looking forward, this $\Linfty$-algebra governing deformations of diagrams of algebras could be the key to a deeper understanding of these deformations. It opens the door to possible further results such as formality theorems.

The methods which I introduce here have been adapted and extended to study deformations of more general structures in \cite{d-h-l2022}.

\subsection*{Acknowledgements}
I wish to thank Chris Fewster and Kasia Rejzner for many conversations, Uli Kr\"ahmer, Ieke Moerdijk, Wendy Lowen, and Sarah Whitehouse for their suggestions, and  Murray Gerstenhaber, Alexander Voronov, and Bruno Vallette for answering some questions by email. Martin Markl drew my attention to \cite{f-m-y2009}.
I would especially like to thank James Griffin, who pointed out that $\Wordop$ is a known operad, and directed me to relevant references.

The paper \cite{k-v-z1997} of Kimura, Voronov, and Zuckerman also played a key role in this project. That paper lists some relations in low arity for a type of homotopy Gerstenhaber operad (although with inconsistent signs). Constructing operations on cochains to satisfy those relations gave me enough examples to guess the general form of useful operations, leading me to define $\Qop$. Note that \cite{k-v-z1997} should be read in conjunction with \cite{vor2000}.

\begin{appendix}
\section{Summary of Notation}
\label{Notation}
For the most part, I have tried to follow the notation and terminology in the book by Loday and Vallette \cite{l-v2012}. 

Fonts are used in the following way.
Operads and colored operads are denoted with text italic names. Categories and similar structures are denoted with sans serif names. Morphisms are denoted by lower case Greek letters. Functors, operad homomorphisms, and similar things are denoted by upper case calligraphic letters. Elements of operads are usually denoted by upper case letters, but there are many exceptions.

\begin{list}{}{\setlength{\labelwidth}{8em}}
\item[$\underline a$]  A specific occurrence of $a$ in a word. 
\item[{$[n]$}]  $=\{0,1,\dots,n\}$ 
\item[$\com{n}$]  $=\{1,2,\dots,n\}$ (Def.~\ref{Comcat def}).
\item[$\#$]  Arity. 
\item[$\norm{\,\cdot\,}$]  Shifted bidegree  (Def.~\ref{Gradings}).
\item[$\abs{\;\cdot\;}$]  Cardinality of a set. Length of a word (Def.~\ref{Word}). Shifted total degree  (Def.~\ref{Gradings}).
\item[$\leq_T$]  Vertical partial order of a tree $T$  (Def.~\ref{Tree}).
\item[$\trianglelefteq_T$]  Horizontal partial order of a tree $T$  (Def.~\ref{Tree}).
\item[$\downarrow$] The order of first occurrence in a word (Def.~\ref{Down order}).
\item[$\widehat\;$] $\Rep_0$ or $\Rep$.
\item[$\Hadtimes$] Hadamard product of operads of sets (Sec.~\ref{Colored operads}).
\item[$\Had$] Hadamard product of algebraic operads \cite[Sec.~5.3.2]{l-v2012}.
\item[$\circ$] Full composition in an operad. Composition of functions.
\item[$\circ_a$] Partial composition at $a$ (Defs.~\ref{Tree composition}, \ref{Word composition}, \ref{MultiSemiD def}, \ref{NSOp def}).
\item[$\comp$] The composition product of Hochschild cochains \cite{g-s1988a,g-s1988b,haw2018}.
\item[$\smile$] Concatenation of words (Def.~\ref{Word}). Concatenation of permutations (Thm.~\ref{Linfty thm2}). The cup product of Hochschild cochains \cite{g-s1988a,g-s1988b,haw2018}.
\item[$\A$] Some functor $\A:\Xc\to\Vect_\Bbbk$ (Sec.~\ref{Representation}) or $\A:\Xc\to\Alg_\Bbbk$ (Sec.~\ref{mQuilts} and later).
\item[$\ad$]  The adjoint action of an operad (Def.~\ref{ad def}).
\item[$\Alg_\Bbbk$] Category of algebras over $\Bbbk$.
\item[$\Brace$]  Brace operad (Defs.~\ref{Treecat def} and \ref{Tree composition})  \cite{cha2002,foi2002}.
\item[$\Xc$] Some fixed small category. 
\item[$C^{\bullet,\bullet}$]  Hochschild bicomplex  (Def.~\ref{Bicomplex def}) with differential $\dS$ for a diagram of vector spaces or $\delta = \dS+\dH$ for a diagram of algebras.
\item[$C_{\mathrm a}^{\bullet,\bullet}$]  $\subset C^{\bullet,\bullet}$. Asimplicial subcomplex (Def.~\ref{Asimplicial}). 
\item[$\bar C^{\bullet,\bullet}$] $\subset C^{\bullet,\bullet}$. Reduced subcomplex  (Def.~\ref{Reduced}).
\item[$\Place_T$]  Word of corners of a tree $T$ (Sec.~\ref{Words sec}).
\item[$\Ind$] Set of colorings (Defs.~\ref{Tree coloring} and \ref{Word coloring}).
\item[$\CQ$] (Def.~\ref{CQ def}) a colored operad of sets.
\item[$\Comop$]  Commutative operad (Sec.~\ref{Com}) \cite{l-v2012}.
\item[$\partial_{\underline a}$] Face of a word (Def.~\ref{Face def}) or quilt.
\item[$\partial$] Boundary in a dg-operad (Def.~\ref{word boundary}).
\item[$\partial'$] Boundary in $\mQuilt$ (Def.~\ref{mQuilt def}).
\item[$\deg$]  Homological degree or total cohomological degree (Defs.~\ref{Degree def} and \ref{Gradings}).
\item[$\delta$] $=\dS+\dH$.
\item[$\dH$]  $=\hat\Delta$. Hochschild coboundary (Def.~\ref{dH def}) \cite{haw2018}.
\item[$\dS$]  Simplicial coboundary  (Def.~\ref{dS def}) \cite{haw2018}.
\item[$\Delta$] $\in\mQuilt$. (Def.~\ref{mQuilt def}).
\item[$\Simplex$] Simplex category. Objects are $[n]$. Morphisms are weakly increasing functions.
\item[$\Comcat$] Category with objects $\com{n}$ and weakly increasing functions as morphisms (Def.~\ref{Comcat def}). Isomorphic to $\Simplex$.
\item[$\SemiD$] $\subset\Simplex$. Subcategory with strictly increasing functions as morphisms (Def.~\ref{Semisimplex def}).
\item[$E_{W,\underline a}$] Expansion of a coloring of a word $W$ (Def.~\ref{Expansion}).
\item[$\E_T$]  Edges of a tree, $T$, or quilt (Def.~\ref{Tree}).
\item[$\varepsilon_i$]  Face map in $\Simplex$ (Def.~\ref{Face map def}).
\item[$\Endop$]  Endomorphism operad of a dg-vector space \cite{l-v2012}.
\item[$\Ext$] The set of extensions (Defs.~\ref{Tree extension}, \ref{Word extension}, and \ref{Quilt extension}).
\item[$\Wordcat$] A set of words (Def.~\ref{Wordcat def}).
\item[$\Wordop$]  Gerstenhaber-Voronov operad (Defs.~\ref{Wordop def}, \ref{word boundary}, and \ref{Word composition}).
\item[$\Ghom$] $:\Susp\Gerst \to H_\bullet(\mQuilt)$ (Thm.~\ref{Gerstenhaber thm}).
\item[$\Gerst$]  Gerstenhaber operad (Def.~\ref{Gerstenhaber def}) \cite[Sec.~13.3.12]{l-v2012}.
\item[$\Qquotient$] $:\Qop\to\Brace$ (Def.~\ref{Quotient}).
\item[$\Wone$] The ``initial'' word used in defining the sign of a coloring (Def.~\ref{Initial word}).
\item[$\homtwo$] $:\Linfty\to\mQuilt$ (Thm.~\ref{Linfty thm2}).
\item[$\homone$] $:\Linfty\to\Qop$ (Thm.~\ref{L-infty thm1}).
\item[$\rep$] Action of $\CQ$ on $C^{\bullet,\bullet}(\A,\A)$ (Def.~\ref{IndQuilt action}).
\item[$\Lieop$]  Lie operad.
\item[$\Linfty$]  Strong homotopy Lie operad (Def.~\ref{Linfty def}).
\item[$\m$] Additional generator for $\mQuilt$.
\item[$\hat\m$] $\in C^{0,2}(\A,\A)$ The product in a diagram of algebras, expressed as a Hochschild cochain (Def.~\ref{dH def}).
\item[$\mQuilt$]  Operad generated by $\Qop$ and $\m$ with differential $\partial'=\partial+\ad_\Delta$ (Def.~\ref{mQuilt def}).
\item[$\MultiSemiD$] Colored operad constructed from the semisimplex category $\SemiD$ (Def.~\ref{MultiSemiD def}).
\item[$\N$] $=\{0,1,2,\dots\}$.
\item[$\NSOp$] The colored operad of sets that governs nonsymmetric operads (Def.~\ref{NSOp def}).
\item[$\pi_a$] Function in Definition~\ref{Word coloring} of a coloring of a word.
\item[$\pi_W$] Canonical function from a word to its alphabet (Def.~\ref{Word}).
\item[$\Qcat$] The set of quilts (Def.~\ref{Quilt def}).
\item[$\Qop$]  Quilt operad (Def.~\ref{Quilt def}).
\item[$\Rep_0$] $:\Qop\to\Endop[\s_2C^{\bullet,\bullet}(\A,\A)]$ (Def.~\ref{Rep0 def}).
\item[$\Rep$] $:\mQuilt\to\Endop[\s\tot C^{\bullet,\bullet}(\A,\A)]$ (Thm.~\ref{mQuilt algebra}).
\item[$\s$]  Degree $-1$ (homological degree $1$) generator for suspension \cite[Sec.~1.5.1]{l-v2012}.
\item[$\s_2$] Bidegree $(0,-1)$ generator for suspension.
\item[$\Susp$]  Suspension of an operad  \cite[Sec.~7.2.2]{l-v2012}.
\item[$\perm_n$] The group of permutations of $\com{n}$.
\item[$\perm$] The disjoint union of all finite symmetric groups.
\item[$\Wtwo$] The ``shuffled'' word used in defining the sign of a coloring (Def.~\ref{Shuffled word}).
\item[$\sgn$] Sign of a permutation. 
\item[$\sgn_W$] Sign for a face of a word  $W$ or quilt (Def.~\ref{Face sign}). 
\item[$\sgn_{P,Q,a}$] Sign for an extension (Def.~\ref{Ext sign}). 
\item[$\sgn_Q$] Sign for a coloring of a quilt $Q$ (Def.~\ref{CQ sign}). Sign for a face of a quilt.
\item[$\sgn_\homone$] The sign for a maximal quilt (Def.~\ref{Max sign}).
\item[$\Tree_Q$] Underlying tree of a quilt $Q$ (Def.~\ref{Quilt def}).
\item[$\tot$]  Total complex of a bicomplex.
\item[$\Treecat$] The set of planar rooted trees (Def.~\ref{Treecat def}).
\item[$\V_T$]  Vertices of $T$  (Def.~\ref{Tree}).
\item[$\Vect_\Bbbk$] The category of vector spaces over a field $\Bbbk$.
\item[$\W_Q$]  Underlying word of a quilt $Q$.
\end{list}

\end{appendix}

\end{document}